\documentclass[10pt]{amsart}

\usepackage{amsmath,amssymb,graphicx,amscd,pinlabel,yhmath,enumitem}

\usepackage[all]{xy}

\usepackage{hyperref}
\hypersetup{colorlinks=true,linkcolor=brown,filecolor=magenta,urlcolor=cyan}

\usepackage[dvipsnames]{xcolor}
\SelectTips{cm}{}

\numberwithin{equation}{section}
\allowdisplaybreaks

\theoremstyle{definition}
\newtheorem{definition}{Definition}[section]
\newtheorem{example}[definition]{Example}
\newtheorem{remark}[definition]{Remark}

\theoremstyle{theorem}
\newtheorem{theorem}[definition]{Theorem}
\newtheorem{proposition}[definition]{Proposition}
\newtheorem{lemma}[definition]{Lemma}
\newtheorem{cor}[definition]{Corollary}
\newtheorem*{claim*}{Claim}

\author{Gw\'ena\"el Massuyeau}
\address{Institut de Math\'ematiques de Bourgogne, UMR 5584, CNRS, 
Universit\'e de Bourgogne, 21000 Dijon, France}
\email{{gwenael.massuyeau@u-bourgogne.fr}}

\newcommand\nc{\newcommand}

\newcommand{\red}[1]{\textcolor{red}{#1}}
\newcommand{\blue}[1]{\textcolor{blue}{#1}}
\newcommand{\magenta}[1]{\textcolor{magenta}{#1}}
\nc\id{\operatorname{id}}
\nc\ad{{\operatorname{ad}}}
\nc\op{{\operatorname{op}}}
\nc\End{\operatorname{End}}
\nc{\interior}{\operatorname{int}}
\nc\Span{\operatorname{Span}}
\nc\ev{{\operatorname{ev}}}
\nc\coev{{\operatorname{coev}}}
\nc\Mag{\operatorname{Mag}}
\nc\Mon{\operatorname{Mon}}
\nc\Gr{\operatorname{Gr}}
\nc\Hom{\operatorname{Hom}}
\nc\gr{\operatorname{gr}}
\nc\conv{\operatorname{conv}}
\nc\Ob{\operatorname{Ob}}
\nc\dbl{\mathsf{d}}
\nc\sfH{\mathsf{H}}
\nc\sfI{\mathsf{I}}
\nc\grp{{\mathrm{grp}}}
\nc\sfP{\mathsf{P}}
\nc\Aut{\operatorname{Aut}}
\nc\diffeo{\operatorname{diffeo}}
\nc\K {{\mathbb K}}
\nc\Q {{\mathbb Q}}
\nc\Z {{\mathbb Z}}
\nc\R {{\mathbb R}}
\nc\N {{\mathbb N}}
\nc\calC {{\mathcal C}}
\nc\calK {{\mathcal{K}}}
\nc\calF {{\mathcal{F}}}
\nc\calI {{\mathcal{I}}}
\nc\calJ {{\mathcal{J}}}
\nc\calV {{\mathcal{V}}}
\nc\calD {{\mathcal{D}}}
\nc\calM {{\mathcal{M}}}
\newcommand{\centre}[1]{\begin{array}{c} #1 \end{array}}

\title{Surgery equivalence relations for $3$-manifolds}

\date{July 21, 2023}

\begin{document}

\begin{abstract}
By classical results of Rochlin, Thom, Wallace and Lickorish, it is well-known that any two 3-manifolds (with diffeomorphic boundaries) are related one to the other by surgery operations. Yet, by restricting the type of the surgeries, one can define  several families of non-trivial equivalence relations 
on the sets of (diffeomorphism classes of) 3-manifolds.
In this expository paper, which is based on lectures given at the school ``\emph{Winter Braids XI}'' (Dijon, December 2021),
  we  explain how certain filtrations of mapping class groups of surfaces enter into the definitions 
  and the mutual comparison of these surgery equivalence relations.
We also survey the ways in which concrete invariants of 3-manifolds (such as finite-type invariants) 
can be used to characterize such relations.
\end{abstract}

\maketitle

{\small \tableofcontents}

\section*{Introduction}

It is a classical result of  Rochlin and Thom, dating back to the early~50's,
that any closed oriented $3$-manifold $M$ is the boundary of a compact oriented $4$-manifold~$W$.
By elementary differential topology arguments (considering a handle decomposition of $W$), 
it follows that  $M$ is obtained   from the $3$-sphere $S^3$ by finitely many \emph{knot surgeries}.
Here a ``knot surgery'' in a $3$-manifold $V$ merely consists in removing a regular neighborhood $\operatorname{N}(K)$  
of a knot  $K$ in $V$ and gluing it back 
while exchanging the meridian with a parallel curve of $K$ on $\partial \operatorname{N}(K)$.

Here is another (equivalent) way of viewing any closed oriented $3$-manifold $M$ as the result of ``modifying''  $S^3$ in some way.
Consider a \emph{Heegaard splitting} of $M$,
i.e$.$ the decomposition of $M=H\cup H'$ into two handlebodies $H,H'$ of the same genus $g$ such that $H\cap H'=\partial H= \partial H'$:
the existence of such a decomposition arises again from elementary differential topology (considering, this time, a handle decomposition of $M$ itself).
Since there also exists a Heegaard splitting of $S^3$ of genus $g$, and since any two handlebodies of genus $g$ are diffeomorphic,
one can find a compact oriented surface $T$ in $S^3$ and a self-diffeomorphism $t$ of $T$
such that  $M$ 
 is obtained from the $3$-sphere $S^3$ by cutting along $T$ and gluing back with $t$. We call this operation a \emph{twist} along $T$ by $t$.
 
 Since ``knot surgeries'' and ``twists'' (as defined above) are thus too general to define interesting relations between $3$-manifolds, 
 it is natural to impose some conditions on these operations. 
 For instance, if one desires a twist to preserve the homology type of $3$-manifolds,
we should require the gluing diffeomorphism to act trivially in homology;
similarly, one can ensure that a knot surgery preserves
the homology type by requiring the knot to be null-homologous 
and by choosing the parallel in a convenient way.
Stronger conditions on knot surgeries or twists can guarantee 
preservation of stricter features of the $3$-manifolds: for instance, their ``nilpotent homotopy types'',
or, their invariance under certain families of topological invariants. 
It turns out that, in the past 40 years, several families of highly non-trivial equivalence relations have been defined for $3$-manifolds
by restricting the type of the ``knot surgeries'' or ``twists.''

In this expository paper, we  aim at surveying the study of such \emph{surgery equivalence relations} 
which, for some of them, have been introduced several times in the literature with different  descriptions.
More specifically, via the above notion of ``twists'',
we shall review how  certain filtrations of mapping class groups of surfaces enter into the definitions 
  and the mutual comparison of these  equivalence relations.
Furthermore, we will survey the ways in which  concrete invariants of 3-manifolds (such as finite-type invariants) 
can be used to characterize such relations.

This expository paper is based on lectures  given at the school ``\emph{Winter Braids XI}'', 
which was held at the IMB (Dijon) in December 2021.
So, in \S 1, we start with preliminary contents for readers
who might not be so familiar with certain constructions of differential topology (e.g$.$ handle decompositions)
or basic results of low-dimensional topology (including the generation of the mapping class groups
in relation with the above-mentioned theorem of Rochlin~\cite{Rochlin} and Thom \cite{Thom51}).
Next, in \S 2, we review the definitions of three families of surgery equivalence relations:
the \emph{$k$-equivalence relations}  defined by Cochran, Gerges \& Orr \cite{CGO},
the \emph{$Y_k$-equivalence relations} 
defined under different names by Goussarov \cite{Go99} and Habiro \cite{Hab00},
and the \emph{$J_k$-equivalence relations}  which arise naturally from the study of the latter.
It follows from their definitions that all these relations are ``hierarchized'' as follows:
$$
\begin{array}{cccccccccccc}
Y_1\hbox{\small-eq$.$} & \! \Longleftarrow \! & Y_2\hbox{\small-eq$.$}  
& \! \Longleftarrow \! & Y_3\hbox{\small-eq$.$}   & \! \Longleftarrow \! & \cdots & Y_k\hbox{\small-eq$.$}   
&  \! \Longleftarrow \! & Y_{k+1}\hbox{\small-eq$.$}   & \! \Longleftarrow \! & \cdots  \\
  \parallel & & \Downarrow & & \Downarrow && &  \Downarrow & & \Downarrow && \\
J_1\hbox{\small-eq$.$}   & \! \Longleftarrow \! & J_2\hbox{\small-eq$.$}   & \! \Longleftarrow \! & J_3\hbox{\small-eq$.$}   & \! \Longleftarrow \! 
& \cdots & J_k\hbox{\small-eq$.$}   &  \! \Longleftarrow \! & J_{k+1}\hbox{\small-eq$.$}   & \! \Longleftarrow \! & \cdots \\
 & & \parallel & &  && & & &  && \\
 &  & 2\hbox{\small-eq$.$}   & \! \Longleftarrow \! & 3\hbox{\small-eq$.$}   & \! \Longleftarrow \! 
& \cdots & k\hbox{\small-eq$.$}   &  \! \Longleftarrow \! & {(k+1)}\hbox{\small-eq$.$}   & \! \Longleftarrow \! & \cdots 
\end{array}
$$
For instance, $Y_1$-equivalence (resp$.$  $2$-equivalence) is generated 
by the twists  (resp$.$ the knot surgeries) of the above-mentioned kinds that preserve the homology type of $3$-manifolds.
We give particular emphasis on the $Y_k$-equivalence relations: indeed, their definition and their study
are closely tied to those of the lower central series of the subgroup 
of the mapping class group acting trivially in homology, namely the \emph{Torelli group} of a surface.
The main advantage of the $Y_k$-equivalence, with respect
to the $J_k$-equivalence and the $k$-equivalence,  consists 
in the existence of a kind  of ``surgery calculus''  --- known as \emph{clasper calculus} ---
which is very efficient to describe the associated quotient sets of $3$-manifolds.

The final section, \S 3, is devoted to the problem 
of characterizing all these equivalence relations. 
We start by reviewing a result of Matveev \cite{Matveev} 
which classifies $Y_1$-equivalence  for closed $3$-manifolds, 
and we extract from the literature several results 
for the characterization of the other equivalence relations in low degree $k$.
We also consider the problem of characterizing them in arbitrary degree $k$:
in the case of the $Y_k$-equivalence relations, 
such a problem is connected to the theory of finite-type invariants which we briefly outline.
In fact, the exact connection between this theory and the family of $Y_k$-equivalence relations
can be viewed as an instance of the so-called ``\emph{Dimension Subgroup Problem}'' in group theory.

Our exposition will be mainly directed towards \emph{closed} oriented $3$-manifolds
and \emph{homology cylinders} over a compact oriented surface $\Sigma$.
The latter constitute a particular, but very important, class of compact oriented $3$-manifolds 
with boundary parametrized by $\partial(\Sigma \times [-1,+1])$:
in fact, homology cylinders even constitute a monoid into which the Torelli group of $\Sigma$ naturally embeds via the mapping cylinder construction,
and which is essentially the monoid of $\Z$-homology $3$-spheres in the case $\Sigma:=D^2$.
Since the works of Goussarov \cite{Go99}, Habiro \cite{Hab00} and Garoufalidis \& Levine \cite{GL05}, 
most of the study on surgery equivalence relations for $3$-manifolds have been focused on monoids of homology cylinders
in relation with the theory of finite-type invariants and the algebraic structure of mapping class groups.

The case of  $3$-manifolds with \emph{arbitrary} boundary is not so much developed in the literature,
although we should mention the notable exception of knots and (string-)links exteriors.
In the study of knots and {(string-)}\-links, 
the $Y_k$-equivalence relations are replaced by the more specific ``$C_k$-equivalence relations''
(which can be formulated in purely knot-diagrammatic terms),
and the role played by the lower central series of the Torelli group for $3$-manifolds
is  played by the lower central series of the pure braid group (which is much better understood):  
then, the study in this case turns out to be rather particular, but it also shares many similarities and connections with the general case.
This study started in relation with the theory of Vassiliev invariants through the works of Stanford \cite{Stanford} and Habiro \cite{Hab00}, 
before being developed and generalized in several directions (see \cite{MY} and references therein).
Yet, for a better delimitation of the problematics, 
the present survey will not consider the specific case of knots and (string-)links.\\

\noindent
\textbf{Acknowledgment.} 
This work has been partly funded by the project ``ITIQ-3D'' of the Région Bourgogne Franche--Comté 
and the project ``AlMaRe'' (ANR-19-CE40-0001-01).
 The IMB receives support from the EIPHI Graduate School (ANR-17-EURE-0002).
 The author is grateful to the referee for the careful reading of the manuscript.

\section{Basics about $3$-manifolds and mapping class groups} \label{sec:basics}

We start this expository paper by reviewing basic facts  and constructions for $3$-manifolds
and mapping class groups of surfaces.\\[-0.2cm]

\noindent
\textbf{Conventions.}
All manifolds are assumed to be smooth and, unless otherwise stated, they are connected and oriented.
For any integer~$n\geq 0$,  $D^n \subset \R^n$  is the $n$-dimensional euclidean disk
and $S^{n} := \partial D^{n+1}$ is the $n$-dimensional sphere. \hfill $\blacksquare$

\subsection{Surgeries and handle decompositions}

We first recall the general definitions of surgeries and handle decompositions in any dimension $m\geq 1$, 
before illustrating these constructions by specializing to the dimension $m=3$.

Let $M$ be a (possibly disconnected) $m$-manifold, let $k\in \{1,2,\dots,m\} $  
and let $i: S^{k-1} \times D^{m+1-k} \hookrightarrow \interior(M)$ be an embedding.
The $m$-manifold
$$
M' := \big(M \setminus \interior i(S^{k-1} \times D^{m+1-k})\big) \cup_{i'} \big( D^k \times S^{m-k}\big) \quad \hbox{where } i':=i\vert_{S^{k-1} \times S^{m-k}}
$$
is said to be obtained from $M$ by the \emph{surgery} of index $k$ along $i$.
Observe that, reversely, $M$ is  obtained from $M'$ by a surgery of index $(m+1-k)$.

\begin{example}\label{ex:dim3}   In dimension $m:=3$, we get the following operations $M\leadsto M'$:
\begin{enumerate}
\item \emph{Index $k=1$:} we consider the disjoint union $S^0 \times D^3$ of two balls in $M$ and replace it by $D^1 \times S^2$;
thus the two balls are deleted and their boundaries are identified one to the other in an orientation-preserving way.
\item \emph{Index $k=2$:} we consider a solid torus $S^1 \times D^2$  in $M$  and replace it by another one $D^2\times S^1$; 
``meridians'' and ``parallels'' of solid tori are exchanged during this process.
\item \emph{Index $k=3$:} we consider a thickened sphere $S^2 \times D^1$ in $M$ and we fill each of the two spheres $S^2 \times S^0$ with a ball.
\end{enumerate}

Thus, a surgery of index $1$ can be of two types in dimension $3$: if the two balls $S^0 \times D^3$ belong to the same connected component of $M$,
then  $M' \cong M\sharp (S^1 \times S^2)$ which  can also be obtained by surgery of index $2$ along a solid torus $S^1 \times D^2 \subset M$ 
such that   $S^1 \times \{0\}$ bounds a disk; otherwise,  $M'$ is obtained from $M$
by taking the connected sum of two of its connected components.

Similarly, a surgery of index $3$ can be of two types: if the thickened sphere  $S^2 \times D^1$ is separating,
then $M$ is reversely obtained from $M'$ by taking the connected sum of two of its connected  components; 
otherwise, we have  $M\cong M'\sharp  (S^1\times S^2)$.

We conclude that, in dimension $3$, it is enough to consider surgeries of index~$2$. 
For  later use, we reformulate them in knot-theoretical terms.
Let $K\subset \interior(M)$ be a knot;  a \emph{parallel} of $K$
is a simple closed curve  in the boundary  $\partial \hbox{N}(K)$ 
of the regular neighborhood $\hbox{N}(K)$ of $K$, that is isotopic to $K$ inside $ \hbox{N}(K)$;
the \emph{meridian} of $K$ is the simple closed curve  $\mu(K)$ in $\partial \hbox{N}(K)$ that bounds a disk in $\hbox{N}(K)$
 but not in $\partial \hbox{N}(K)$; 
 up to isotopy in $\partial \hbox{N}(K)$, the meridian is unique but there are infinitely many possibilities for a parallel.
 See Figure \ref{fig:knot}.
 
\begin{figure}[h!]
\includegraphics[scale=0.8]{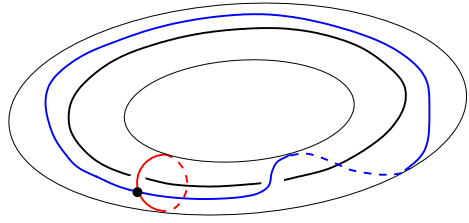}
\caption{A knot $K$ (black) in its regular neighborhood $\hbox{N}(K)$,
together with the meridian  (red) and a parallel (blue)} \label{fig:knot}
\end{figure}

We now assume that $K$ is \emph{framed} in the sense that a parallel $\rho(K)$ has been specified;
 then the $3$-manifold obtained from $M$ by \emph{surgery} along  $K$ is
 $$
 M_K := \big( M \setminus \interior \hbox{N}(K) \big) \cup_\phi (D^2 \times S^1)
 $$ 
 where $\phi: S^1 \times S^1 \to \partial  \hbox{N}(K)$ is a diffeomorphism mapping $\{1\} \times S^1$ to $\mu(K)$
 and $S^1 \times \{1\}$ to $\rho(K)$. 
 The manifold $M_K$ is well-defined only up to orientation-preserving diffeomorphisms,
 and the surgery $M\leadsto M_K$  is the same as a surgery $M\leadsto M'$ of index~$2$, 
 where the embedding $i: S^1 \times D^2 \hookrightarrow \interior(M)$ has image $\hbox{N}(K)$ and
 maps $S^1 \times \{0\}$ (resp$.$  $S^1 \times \{1\}$) to $K$ (resp$.$ to $\rho(K)$).   
 
Very often, a framed knot $K$ in a $3$-manifold $M$ is given by drawing on the blackboard a \emph{knot diagram}, 
which represents the image of a generic projection
of the knot on a planar surface $B\subset M$ onto which (part of) $M$ deformation retracts:
we keep track of the ``over/under''  crossing information  at each double point
and the parallel of $K$ is given  by lifting the curve parallel to the projection of $K$ in $B$.
This is called the ``\emph{blackboard framing convention}''.
For instance, here are three diagrams of the trivial knot $U\subset S^3$ showing three different framings:
$$
\centre{
\includegraphics[scale=0.3]{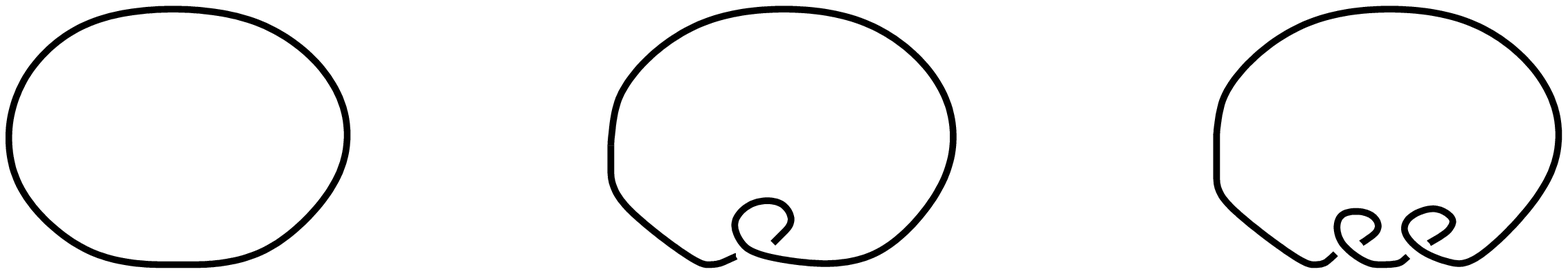}
}
$$
 then the resulting manifold $S^3_U$ is  $S^1 \times S^2$, $S^3$ and $\R P^3$, respectively. 
  (To be specific, the knots are given in $\R^3 \subset  S^3 $ and the planar surface $B$ onto which we project  is an affine plane of $\R^3$.)  
\hfill $\blacksquare$
\end{example}

A surgery of index $k$ is only the tip of the iceberg of a higher-dimensional operation. 
Let $n\in \N$ and $k\in \{0,\dots,n\}$. A \emph{$k$-handle} in dimension $n$ is a copy of $D^k\times D^{n-k}$; 
its boundary can be decomposed into two parts:
$$
\partial (D^k \times D^{n-k}) = \big( S^{k-1} \times D^{n-k} \big) \cup \big( D^k \times S^{n-k-1} \big)
$$
Let $W$ be an $n$-manifold with boundary. \emph{Attaching} a $k$-handle to $W$ means to specify an embedding
$i: S^{k-1} \times D^{n-k} \hookrightarrow \partial W$ to construct the new $n$-manifold 
$$
W' = W \cup_i \big(D^k \times D^{n-k}\big).
$$
Then $\partial W'$ is obtained from $\partial W$ by a surgery of index $k$.

\begin{remark}
Technically speaking, the new manifold $W'$ has ``corners'' but there exists a standard procedure to round those ``corners''.
Alternatively, one can give a  smooth model of the attachment of a $k$-handle that arises from Morse theory (see below).
For instance, here are schematic images  (with or without corners) of  a $1$-handle attached in dimension $2$:\\[0.2cm]
$$
\centre{
\labellist \small \hair 2pt
\pinlabel {corners} [r] at 10 500
\endlabellist
\includegraphics[scale=0.16]{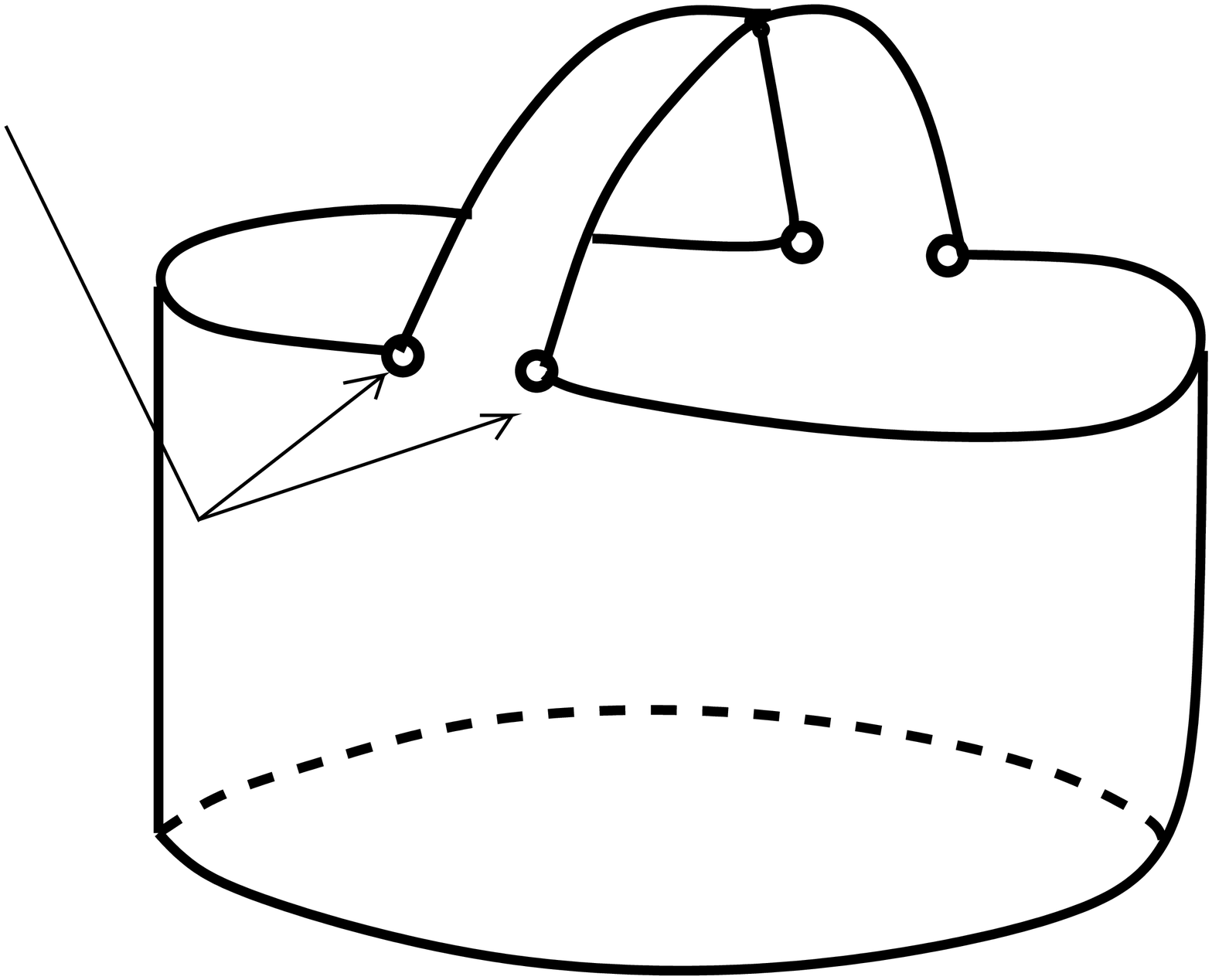}
}
\quad \hbox{vs} \quad
\centre{
\includegraphics[scale=0.16]{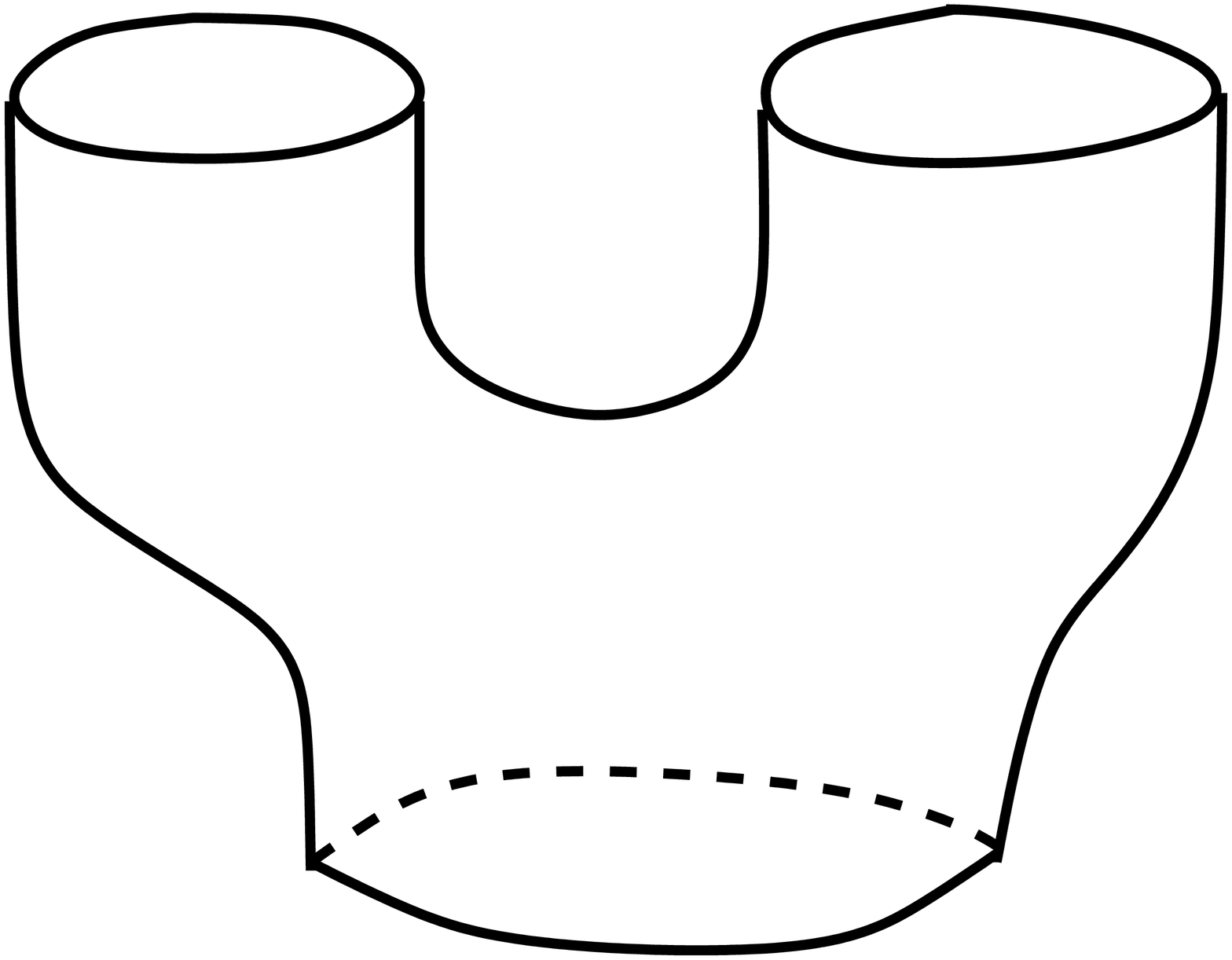}
}
$$
\hfill $\blacksquare$
\end{remark}

Two closed $m$-manifolds $M$ and $M'$ are \emph{cobordant} if there exists a compact $(m+1)$-manifold $W$ such that 
$\partial W \cong (-M) \sqcup M' $. Then, $W$ is called a \emph{cobordism} from $M$ to $M'$. 
  Of course, any compact $n$-manifold $W$ with boundary  can be viewed as a cobordism from $\varnothing$ to $\partial W$
and, in particular, any closed $n$-manifold $W$ can be viewed as a cobordism from $\varnothing$ to $\varnothing$.  

\begin{definition}
The \emph{$m$-th cobordism group} is the quotient set
$$
\Omega_m :=  \frac{\{\hbox{\small closed $m$-manifolds}\}}{\hbox{\small cobordism}}
$$
equipped with the disjoint union $\sqcup$ operation.  \hfill $\blacksquare$
\end{definition}

Thom \cite{Thom54} studied those abelian groups   for all integers $m\geq 1$:
he described them as kinds of stable homotopy groups, he showed that they constitute the coefficient modules 
of a generalized homology theory, he computed $\Omega_m$ up to degree $m=7$
and   gave, among other things, an explicit computation of the ring $\Omega_* \otimes \Q$ \dots 
For this pioneering work, Thom received the Fields Medal in 1958.

\begin{example}
As soon as one knows the classification of closed $k$-manifolds for $k\in \{0,1,2\}$, it is pretty clear that
$$
\Omega_0  \simeq \Z, \quad \hbox{and} \quad 
\Omega_1= \Omega_2=\{0\}.
$$
However, it is much less obvious that $\Omega_3=\{0\}$ as well: we shall prove it in \S \ref{subsec:Omega_3}. \hfill $\blacksquare$
\end{example}

Let $W$ be an $n$-dimensional cobordism from $M$ to $M'$.
A \emph{handle decomposition} of $W$ is an increasing sequence
$$
W_{-1} \subset  W_0 \subset W_1 \subset \cdots \subset W_{n} = W 
$$
where $W_{-1} \cong M \times [-1-\epsilon,-1+\epsilon]$  and $W_i$ is obtained from $W_{i-1}$ by attaching finitely many $i$-handles.
  Note that $-W$ is a cobordism from $M'$ to $M$ and has a \emph{dual} handle decomposition, 
consisting of  one handle of index $n-i$ for every handle of index $i$ in $W$.\\[-0.2cm]  
 
\noindent
\textbf{Fact.} \emph{Morse theory tells us that any cobordism $W$ has a handle decomposition. 
Specifically, any  Morse function $f:W \to [-1-\epsilon,n + \epsilon]$ such that 
\begin{itemize}[label=$\diamond$]
\item for each $i\in\{0,1,\dots,n\}$, all critical points of $f$ of index $i$ are in $f^{-1}(i)$,
\item  $(-1-\epsilon)$ and $(n+\epsilon)$ are regular values of $f$,
\item $f^{-1}(-1-\epsilon)=M$ and $f^{-1}(n+\epsilon)=M'$,
\end{itemize}
defines a handle decomposition of $W$ by setting $W_i:=f^{-1}([-1-\epsilon,i+\epsilon])$.
Furthermore, there is one handle of index $i$ for every critical point of $f$ of index $i$.}  \hfill $\blacksquare$\\

\noindent
  We recommend Milnor's textbooks \cite{Milnor1,Milnor2}  for an introduction to Morse theory. As a complement to this, 
Cerf theory can also tell us how any two handle decompositions of the same cobordism
are related one to the other by some  operations (namely, \emph{creation/annihilation} of two handles of consecutive indices, and \emph{handle slidings}).
But we shall not need that in these lectures. 

It follows from the above fact that, in particular, any closed $n$-manifold $W$ has a handle decomposition $W_0 \subset W_1 \subset \cdots \subset W_n=W$
where $W_0$ consists of  $0$-handles, 
$W_1$ is obtained from $W_0$ by attaching $1$-handles, and so on, to finish by gluing $n$-handles to get $W$. As is easily checked, we can assume that

\begin{itemize}[label=$\diamond$]
\item $W_0$ consists of a single $0$-handle  $D^0\times D^n$,
\item dually, $W_n$ is obtained from $W_{n-1}$ by attaching a single $n$-handle $D^n \times D^0$.
\end{itemize} 

\begin{example} \label{ex:Heegaard}
Let $M$ be a closed $3$-manifold. According to what has been recalled above, $M$ has a handle decomposition
$$
M_0 \subset M_1 \subset M_2 \subset M_3 =M
$$
with a single $0$-handle and a single $3$-handle. Thus, there is an integer $g\geq 0$ such that $M_1$ is diffeomorphic to 
\begin{center}
\labellist \small \hair 2pt
\pinlabel {$H_g:=$} [r] at 0 100
\pinlabel {$1$} [br] at 48 212
\pinlabel {$g$} [bl] at 383 210
\pinlabel {$\cdots$} at 215 193
\endlabellist
\centering
\includegraphics[scale=0.3]{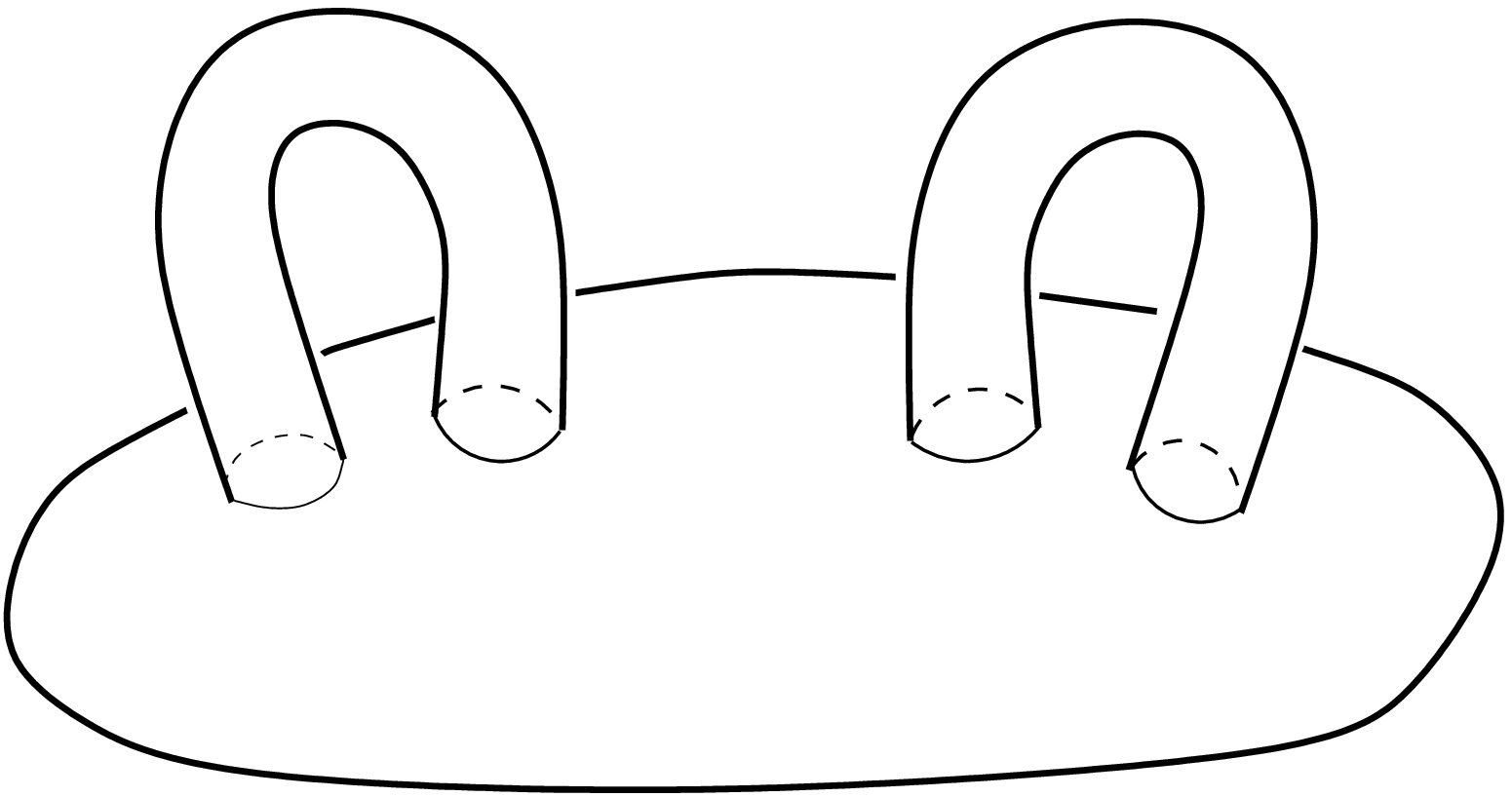}
\end{center}
which is a called the standard \emph{handlebody} of genus $g$, and whose boundary $$\Sigma_g:= \partial H_g$$
is the standard closed (oriented) surface of genus $g$.
Dually, there is an integer $g'$ such that $M'_1:= M \setminus \interior(M_1)$ is diffeomorphic to $H_{g'}$. 
Since $M_1$ and $M'_1$ share the same boundary, we must have $\Sigma_g=\Sigma_{g'}$: hence $g=g'$.
We conclude that any closed $3$-manifold  $M$ can be decomposed as
$$
M \cong H_g \cup_f  (-H_g)
$$
where $f: \Sigma_g \to \Sigma_g$ is an orientation-preserving diffeomorphism. 
Such a decomposition is called a \emph{Heegaard splitting} of $M$ of genus $g$. \hfill $\blacksquare$
\end{example}

In the rest of these notes, we restrict our attention to $3$-manifolds.

\subsection{Mapping class groups of surfaces}

The Heegaard splittings, which have been described in Example \ref{ex:Heegaard},
reveal that all closed  $3$-manifolds can be efficiently presented in terms of  diffeomorphisms of surfaces.
The following lemma adds that, being only interested in $3$-manifolds \emph{up to diffeomorphisms}, 
we only have to consider diffeomorphisms of surfaces \emph{up to isotopy}.

\begin{lemma}
Let $g\in \N$.
The (oriented) diffeomorphism type of $M_f := {H_g \cup_f  (-H_g)}$ only depends on the isotopy class of $f$.
\end{lemma}

\begin{proof}  
For any orientation-preserving diffeomorphisms $E:H_g \to H_g$ and $f:{\Sigma_g \to \Sigma_g}$, 
we clearly have
$$
M_{f \circ E|_{\Sigma_g}} \cong  M_f \cong M_{E|_{\Sigma_g} \circ f}.
$$
Assume that $f': \Sigma_g \to \Sigma_g$ is another  orientation-preserving diffeomorphism which is isotopic to $f$.
Then $e= f^{-1} \circ f'$ is isotopic to the identity, 
and we can use a collar neighborhood of $\Sigma_g$ in $H_g$ to construct a diffeomorphism $E:H_g \to H_g$
such that $E|_{\Sigma_g}=e$. We conclude that 
$M_{f'}= M_{f\circ e}\cong  M_f.$  
\end{proof}

Thus we are led to consider the \emph{mapping class group} of the surface $\Sigma_g$, which is defined by
\begin{equation} \label{eq:MCG}
\calM(\Sigma_g) := \frac{\{\hbox{orientation-preserving diffeomorphisms $\Sigma_g \to \Sigma_g$}\} }{\hbox{ isotopy}}.
\end{equation}
  We refer to the textbooks \cite{Birman,FM} for an  exposition of mapping class groups.
For the moment, we just need to review the simplest examples 
and     give explicit generating systems for those groups.

\begin{example} \label{ex:low_genus}
  The group $\calM(\Sigma_0)$ is trivial. Besides,
through its action on the abelian group  $H_1(\Sigma_1;\Z)\simeq \Z^2$, 
the group $\calM(\Sigma_1)$ is isomorphic to $\hbox{SL}(2;\Z)$.
See the above-mentioned  textbooks, or \cite[\S 2]{Mas11} for a direct treatment of these examples.   \hfill $\blacksquare$ 
\end{example}

Let $\alpha$ be a simple closed curve  in $\Sigma_g$. 
We identify  a regular neighborhood $\hbox{N}(\alpha)$ of $\alpha$ 
 with the annulus  $S^1 \times [0,1]$, in such a way that orientations are preserved.
The \emph{Dehn twist} along $\alpha$ is the diffeomorphism 
$T_\alpha: \Sigma_g \to \Sigma_g$ defined by 
$$
T_\alpha(x) = \left\{\begin{array}{ll}
x & \hbox{ if } x \notin \operatorname{N}(\alpha)\\
\left(e^{2i\pi(\theta +r)},r\right) & \hbox{ if } x=\left(e^{2i\pi \theta},r\right) \in N(\alpha)=S^1 \times [0,1].
\end{array}\right.
$$
  Because of the choice of $\operatorname{N}(\alpha)$ and its ``parametrization'' by $S^1 \times [0,1]$,
the diffeomorphism $T_\alpha$ is only defined up to isotopy.
But the isotopy class  $[T_\alpha] \in \calM(\Sigma_g)$
only depends on the isotopy class  of the curve~$\alpha$.   
Here is the effect of $T_\alpha$ on a curve $\rho$
which crosses transversely $\alpha$ in a single point:
\begin{center}
\labellist \small \hair 2pt
\pinlabel {\textcolor{red}{$\alpha$}} [r] at 40 155
\pinlabel {\textcolor{blue}{$\rho$}} [r] at 144 24
\pinlabel {$\operatorname{N}(\alpha)$} [bl] at 278 262
\pinlabel {$\stackrel{T_{\alpha}}{\longrightarrow}$}  at 418 155
\endlabellist
\centering
\includegraphics[scale=0.3]{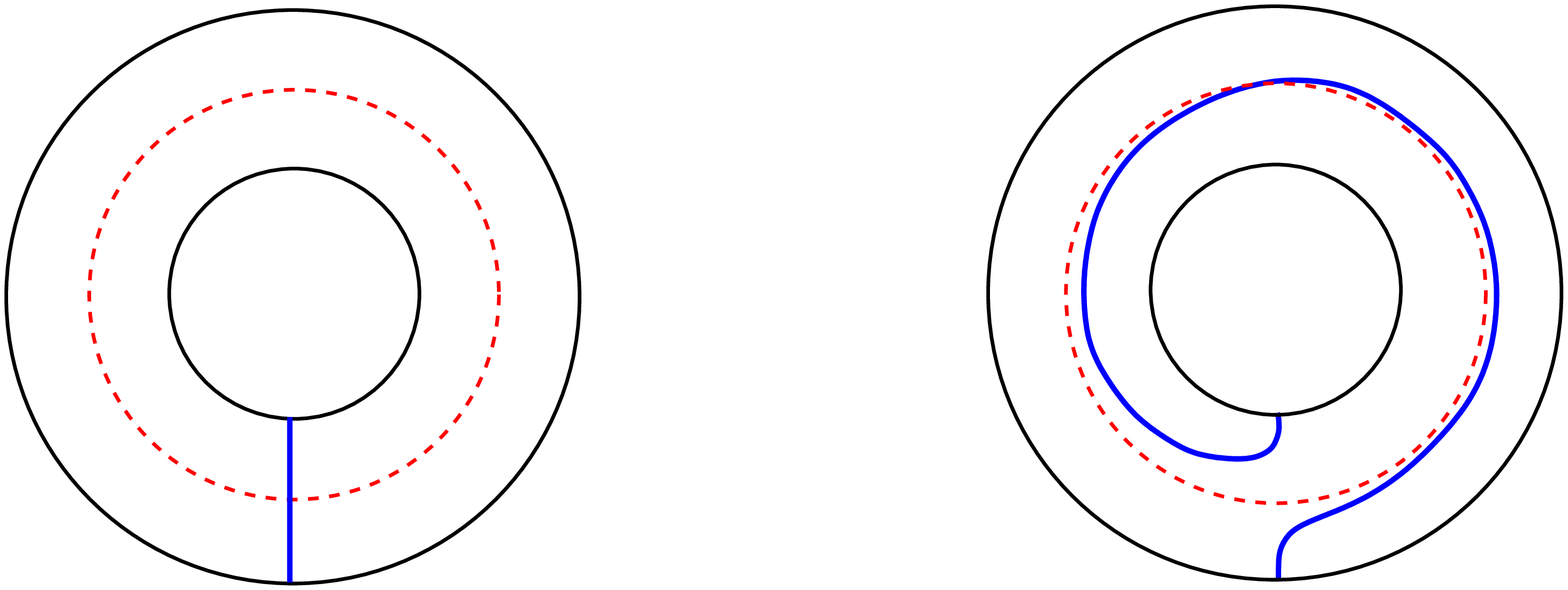}
\end{center}
\vspace{0.1cm}

\begin{theorem}[Dehn 1938]
\label{th:Dehn}
In any genus $g\geq 1$, the group  $\calM(\Sigma_g)$ is generated by finitely many Dehn twists.
\end{theorem}

Dehn's generating system \cite{Dehn} can be written explicitly. 
It consists of $2$ twists in genus $g=1$, and $5$ twists in genus $g=2$: see Figure \ref{fig:Dehn_1_2}. 
In genus $g>2$,  $\calM(\Sigma_g)$ is generated by the Dehn twists  along the $2g(g-1)$ curves shown in Figure~\ref{fig:Dehn}: 
the curves  $\alpha_i$ (blue, for all $i\in \{1,\dots,g\}$),   $\beta_i$ (red, for all $i\in \{1,\dots,g\}$),   $\delta_i$ (purple, for all $i\in \{1,\dots,g\}$),
$\gamma_{ij}$ (green, for any pair $\{i,j\}$ of two elements in $\{1,\dots,2g\}$ that are of distance at least three in the cyclic order).

\begin{figure}[h!]
\includegraphics[scale=0.45]{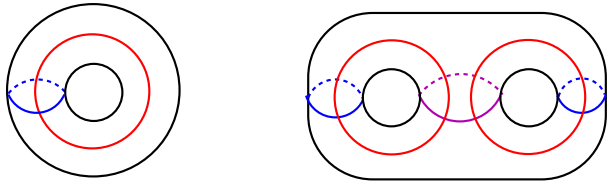} 
\caption{Dehn's generators  in genus $1$ and $2$}
\label{fig:Dehn_1_2}
\end{figure}

\begin{figure}[h!]
\includegraphics[scale=0.25]{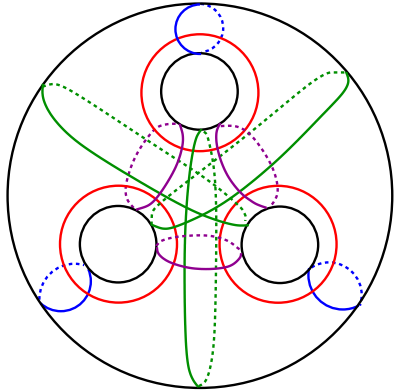} 
\caption{Dehn's generators  in genus $g>2$}
\label{fig:Dehn}
\end{figure}

In the sequel, 
we shall only need the following information about Dehn's generating system of $\calM(\Sigma_g)$:
\begin{equation}
\label{eq:m&p}
\centre{
\hbox{In genus $g>1$, the group $\calM(\Sigma_g)$ is generated by Dehn twists along }\\
\hbox{ simple closed curves, each avoiding a sub-handlebody of genus $1$ of $H_g$.}
}
\end{equation}
Here $\Sigma_g$ is regarded as the boundary of the standard handlebody~$H_g$,
and a \emph{sub-handlebody} of genus $k$ of $H_g$ is the image of $H_k$ under some diffeomorphism
$H_k\, \sharp_\partial\, H_{g-k} \cong H_g$.

\begin{remark}\label{rem:Lickorish}
In the sixties, Lickorish rediscovered and simplified Dehn's generating system of the mapping class group \cite{Lickorish}.
He proved that $\calM(\Sigma_g)$ is actually generated by the Dehn twists 
along the  simple closed curves 
$$
\alpha_1,\dots,\alpha_g,\beta_1,\dots,\beta_g,\gamma_1,\dots,\gamma_{g-1}
$$
shown below:
\vspace{0.cm}
\begin{center}
\labellist \small \hair 2pt
\pinlabel {$\blue{\alpha_1}$} [b] at 97 132
\pinlabel {$\blue{\alpha_2}$} [b] at 260 153 
\pinlabel {$\blue{\alpha_{g-1}}$} [b] at 470 153
\pinlabel {$\blue{\alpha_g}$} [b] at 636 127
\pinlabel {$\red{\beta_1}$} [t] at 90 48
\pinlabel {$\red{\beta_2}$} [t] at 261 48
\pinlabel {$\red{\beta_{g-1}}$} [t] at 462 48
\pinlabel {$\red{\beta_g}$} [t] at 628 50
\pinlabel {$\magenta{\gamma_1}$} [t] at 180 46
\pinlabel {$\magenta{\gamma_{g-1}}$} [t] at 553 43
\pinlabel {{\bf \dots }} at 366 73
\endlabellist
\centering
\includegraphics[scale=0.5]{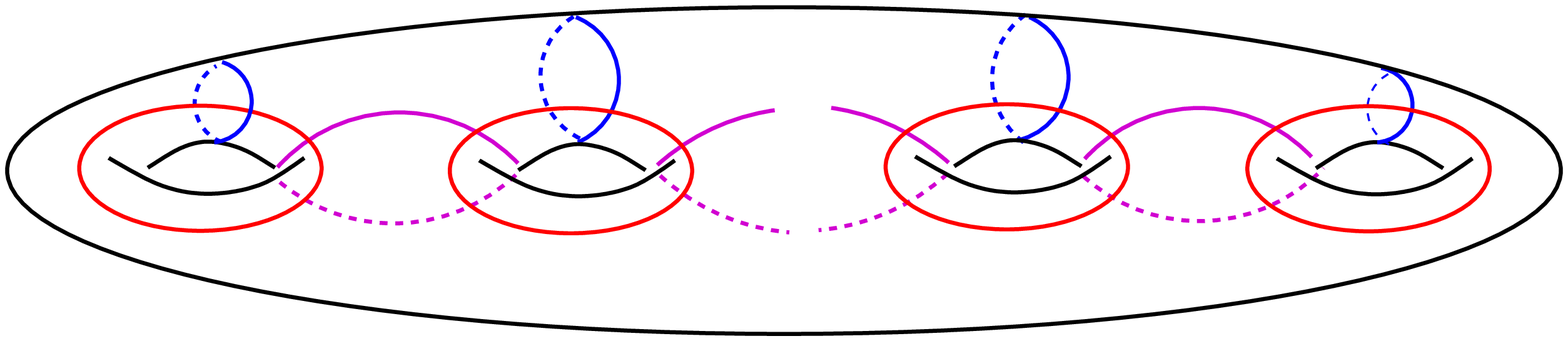}
\end{center}
Afterwards, Humphries  \cite{Humphries} showed that $2g+1$ Dehn twists are enough to generate $\calM(\Sigma_g)$:
specifically, those are the twists along
$
\beta_1,\dots, \beta_g, \gamma_1,\dots, \gamma_{g-1}, \alpha_1,\alpha_2.
$
 \hfill $\blacksquare$
\end{remark}

\subsection{Triviality of $\Omega_3$} \label{subsec:Omega_3}

Let $\calV(\varnothing)$ be the set of diffeomorphism classes of closed  $3$-manifolds.
(Recall that, unless otherwise stated, $3$-manifolds are always oriented.)

\begin{theorem}[Rochlin 1951, Thom 1951,  Wallace 1960, Lickorish 1964]   \label{th:RTWL}  
The following four statements are equivalent, and  hold true:
\begin{itemize}[label=$\diamond$] 
\item[(1\,)] we have $\Omega_3=\{0\}$, i.e$.$ any two $M, M' \in \calV(\varnothing)$ are cobordant;  
\item[(1')] any $M\in \calV(\varnothing)$ is the boundary of a compact $4$-manifold $W$;                            
\item[(2\,)] for any $ M, M' \in \calV(\varnothing)$, there is a sequence of surgeries along framed knots 
$M=M_0 \leadsto M_1 \leadsto \cdots  \leadsto M_r=M'$;
\item[(2')] for any $M  \in \calV(\varnothing)$, there is a framed link $L\subset S^3$
such that $S^3_L \cong M$.
\end{itemize}
\end{theorem}

\begin{proof}[Proof(s)]
The equivalence between (1) and (1') is clear. 
  Assuming (1), let $M  \in \calV(\varnothing)$:
there  is a compact $4$-manifold $W^\circ$ such that $\partial W^\circ \cong (-S^3) \sqcup M$; let $W := W^\circ \cup_\partial D^4$
where $D^4$ is glued along the $S^3$ boundary component of $W^\circ$; then $\partial W \cong M$.
Assuming (1'), let $M, M' \in \calV(\varnothing)$: then $(-M)\sqcup M'  \in \calV(\varnothing)$ and 
there is a compact $4$-manifold $W$ with boundary $(-M)\sqcup M'$.   

The equivalence between (2) and (2') is also easy.  
Assuming (2), let $M\in \calV(\varnothing)$; 
there is a sequence of surgeries along framed knots 
$S^3=M_0  \leadsto \cdots  \leadsto M_r=M$; for each $i$, we can assume that the framed
knot  $K_i \subset M_i$ along which we do the surgery to get $M_{i+1}$ is disjoint from the glued solid tori that correspond
to the previous surgeries, hence we can view $K_i$ as a knot in the initial manifold $S^3$;
then the framed link $L:= K_0 \sqcup \dots \sqcup K_{r-1}$ is such that $S^3_L \cong M$.
Assuming (2'), let $M,M'\in \calV(\varnothing)$;
there is a framed link $L \subset S^3$ such that $S^3_L \cong M$; by doing the surgeries along the components of $L$ stepwisely,
we obtain a first sequence of surgeries $S^3=M_0 \leadsto  \cdots  \leadsto M_r=M$;
similarly, we find a second sequence of surgeries $S^3=M'_0 \leadsto  \cdots  \leadsto M'_{r'}=M'$;
thus, by reversing the first sequence, we get a sequence of surgeries producing $M'$ from $M$.  

The equivalence between (1) and (2) is a result of Wallace \cite{Wallace}.
Indeed Wallace proved that, in any dimension $m\geq 1$, two closed $m$-manifolds $M$ and $M'$ are cobordant 
if and only if there is a sequence
$$
M=M_0 \leadsto M_1 \leadsto \cdots  \leadsto M_r=M'
$$ 
where $M_i \leadsto M_{i+1}$ stands for a surgery of index $k_i$ and the sequence $(k_i)_i$ is not decreasing.
(In \cite{Wallace}, surgeries are called \emph{spherical modifications}.) This equivalence follows from
the existence of handle decompositions for cobordisms and the relation between surgery and attachement of handles.
Observing that, in dimension $m=3$, only surgeries of index $2$ do  matter (see Example \ref{ex:dim3}),
Wallace assumes (1) to deduce (2') thus answering a question of Bing \cite{Bing}.

Indeed, statement (1) had been proved independently by Rochlin \cite{Rochlin} and Thom \cite{Thom51,Thom52,Thom54}.
Actually, Thom gave three proofs of very different natures: let us expose the proof that came chronologically first and is sketched in \cite{Thom51}.
It uses Heegaard splittings of $3$-manifolds,
the key idea being that the subset
$$
B_g := \big\{\,  [f] \in \calM(\Sigma_g) : M_f= H_g \cup_f  (-H_g) \hbox{ bounds a compact $4$-manifold}\, \big\} 
$$
is a subgroup of the mapping class group, for every $g\in \N$:
\begin{itemize}[label=$\diamond$]
\item $1\in B_g$ because $M_{\operatorname{id}}$ is diffeomorphic to $\sharp^g (S^1 \times S^2)$
which, for instance, is the boundary of   $\sharp^g_\partial (D^2 \times S^2)$; 
\item if $f\in B_g$, then $f^{-1} \in B_g$ because $M_{f^{-1}} \cong -M_f$;
\item if $f,f'\in B_g$, then $f'f\in B_g$ because, given a compact $4$-manifold $W$ bounded by $M_f$ 
and a compact $4$-manifold $W'$ bounded by $M_{f'}$,
the $3$-manifold $M_{f'f}$ is the boundary of the $4$-manifold obtained by gluing $W$ and $W'$ along the ``left side'' handlebody $H_g$ of $M_f$ 
and the ``right side'' handlebody $-H_g$ of $M_{f'}$. 
\end{itemize}
Since any $3$-manifold has a Heegaard splitting, the triviality of $\Omega_3$ will follow from the fact that, for any $g\geq 1$,  
we have $B_g=\calM(\Sigma_g)$ or, equivalently, that each of Dehn's generators of $\calM(\Sigma_g)$ belongs to $B_g$.
This is proved by induction on $g$. In genus $g=1$, there are two generators $\tau$  (see Figure~\ref{fig:Dehn_1_2}):
the corresponding $3$-manifold $M_\tau$ is either $S^3=\partial B^4$ or $S^1\times S^2=\partial(D^2 \times S^2)$;
hence $B_1= \calM(\Sigma_1)$. Assume that $B_{g-1}= \calM(\Sigma_{g-1})$.
According to \eqref{eq:m&p}, 
each Dehn generator of $\calM(\Sigma_g)$ is a Dehn twist $\tau$ along a simple closed curve  avoiding a sub-handlebody of genus $1$ of~$H_g$;
therefore $M_\tau$ is diffeomorphic to $(S^1 \times S^2)\sharp M_h$ for some $h\in \calM(\Sigma_{g-1})$;
hence $M_\tau$ is related to $M_h$ by a surgery of index $1$, so that $M_\tau$ and $M_h$ are cobordant; 
by the induction hypothesis, $M_h$ bounds, and so does $M_\tau$; hence $\tau \in B_g$.

Being not aware of Dehn's work \cite{Dehn}, 
Lickorish re-proves in \cite{Lickorish} that $\mathcal{M}(\Sigma_g)$ is generated by finitely many Dehn twists (see Remark \ref{rem:Lickorish}),
and he shows statement (2) in a direct way. The key idea in his argument is the following:
\begin{quote}\textbf{Lickorish's trick.}
\emph{Let $U$ and $V$ be compact $3$-manifolds whose boundaries are identified.
Let  $\gamma \subset \partial V$ be a simple closed curve, and let $K\subset \interior(V)$ be the knot obtained
by slightly ``pushing'' $\gamma$. Then we have
$$
U  \cup_{\tau} (-V) \cong U \cup_{\operatorname{id}} (-V_K) 
$$
where $\tau:=T_\gamma$ is the Dehn twist along $\gamma$, 
and $V_K$ is obtained from $V$ by surgery along $K$ framed with the parallel differing
from $\gamma$ by a meridian of $K$.}
\end{quote}
This trick is easily verified using the definitions of a surgery and a Dehn twist.
Let $g\in \N$ and $f\in \calM(\Sigma_g)$. Decomposing $f$ as a product of Dehn twists (or their inverses), 
Lickorish's trick implies that  $M_f=  H_g \cup_f  (-H_g)$ can be transformed into $M_{\operatorname{id}}=\sharp^g (S^1 \times S^2)$ 
by finitely many surgeries along framed knots. The same is true about $S^3$,  
since we have  $S^3 = M_\iota$ for some $\iota \in \calM(\Sigma_g)$ and whatever $g$ is.
Hence, $M_f$ can be transformed into $S^3$ by finitely many surgeries.
\end{proof}

\begin{remark}  
 Rourke gave in \cite{Rourke} yet another proof of statement (2) of Theorem~\ref{th:RTWL},
 which is also based on the presentations of $3$-manifolds by their Heegaard splittings.
But, in contrast with Thom's and Lickorish's arguments, his proof does not need 
knowledge about the generation of the mapping class group. 
It is both tricky and elementary. \hfill $\blacksquare$  
\end{remark}

We can be more general and consider $3$-manifolds with boundary.
Let $R$ be a closed  surface, which may be disconnected.
A compact  $3$-manifold $M$ has boundary  \emph{para\-metrized} 
by $R$, if $M$ comes with a map $m:R \to M$  which is an orientation-preserving diffeomorphism onto $\partial M$.
Our convention will always be to denote the boundary parametrization with the lower-case letter.

Two  manifolds with parametrized boundary $M$ and $M'$ are considered \emph{diffeomorphic}
if there is an orientation-preserving diffeomorphism $f:M\to M'$ such that $f \circ m =m'$.
We denote by 
$
\calV(R)
$
the set of diffeomorphism classes of compact   $3$-manifolds with boundary parametrized by the surface $R$.

\begin{cor}
For any $M, M' \in \calV(R)$, there is a sequence of surgeries along framed  knots 
$M=M_0 \leadsto M_1 \leadsto \cdots  \leadsto M_r=M'$.
\end{cor}

\begin{proof} 
Denote by $(R_i)_i$ the family of connected components of $R$ and, for each~$i$,
fix an identification of $R_i$ with the standard surface $\Sigma_{g_i}$
  where $g_i$ is the genus of~$R_i$.
Fix in $S^3$ a copy $H$ of the disjoint union $-\sqcup_i  H_{g_i}$ of standard handlebodies. 
Then $S^3 \setminus \interior(H)$ with the obvious boundary parametrization
defines a ``preferred'' element of $\calV(R)$. 

We shall prove that $M$ can be transformed into $S^3 \setminus \interior(H)$    by surgery along a framed link $L$.
To this purpose, we consider the closed $3$-manifold
$$
\overline{M}:= M \cup_m \big( - \sqcup_{i} H_{g_i} \big).
$$
By Theorem \ref{th:RTWL}, there is a framed link $L \subset \overline{M}$ and an orientation-preserving diffeomorphism
$\phi: \overline{M}_L  \to S^3$; furthermore, we can assume that $L$ is contained in $M$ after an isotopy.
The image $H' \subset S^3$ of  $\sqcup_{i} H_{g_i} \subset \overline{M}$ by $\phi$ is a disjoint union of handlebodies.
  Of course, we have a priori $H\neq H'$.
Then, we think of $H$ and  $H'$  as regular neighborhoods in $S^3$ of some 
 knotted framed graphs $G$ and $G'$, respectively,
 of the same topological type. After finitely many ``crossing changes'' and ``framing changes", $G'$ can be transformed to $G$ 
 since they have the same
 topological type. Each of these ``crossing changes'' and  ``framing changes" can be realized by surgery along a framed trivial knot
and, after an isotopy, we can assume that each such knot does not meet the part of $S^3= \phi(\overline{M}_L)$ where the surgery along $L$ took place.
 Therefore, after addition of some components to the framed link $L$, we can assume that $H=H'$ as subsets of $S^3$. 
 Hence $\phi$ restricts to an orientation-preserving diffeomorphism $M_L \to S^3 \setminus \hbox{int}(H)$.
 This diffeomorphism may not be compatible with the boundary parametrizations of $M$ and $S^3 \setminus \hbox{int}(H)$.
 However, since $\calM(R_i)$ is generated by Dehn twists and since every Dehn twist can be realized by a surgery along a knot
 (using Lickorish's trick), 
 we can assume this compatibility at the price of adding to $L$ yet other components. We conclude
 that $M_L$ and $S^3 \setminus \hbox{int}(H)$ represent the same class in~$\calV(R)$.   
\end{proof}

\section{Surgery equivalence relations: definitions and first properties}

We have seen in \S \ref{sec:basics} that the surgery operations arising directly from differential topology 
are too general in dimension three: any two compact $3$-manifolds 
(with the same parametrized boundary, if any) can be related one to the other by such operations.
Thus, to relate $3$-manifolds in an interesting way, 
we need to consider  more restrictive modifications 
and one reasonable restriction is to require that they preserve the homology type of  $3$-manifolds. 
So, we are led to consider the subgroup of the mapping class group that acts trivially in homology.

\subsection{Torelli groups of surfaces}

Let $S$ be a compact surface with, at most, one boundary component. 
As a generalization of \eqref{eq:MCG},
the \emph{mapping class group} of $S$ is defined by
$$
\calM(S) = \left\{\begin{array}{l}
\frac{\big\{\hbox{ orientation-preserving diffeomorphisms $S \to S$}\big\} }{\hbox{ isotopy}} \quad \hbox{ if } \partial S =\varnothing, \\
\frac{\big\{\hbox{  diffeomorphisms $S \to S$ that are the id on $\partial S$}\big\} }{\hbox{ isotopy rel $\partial S$}} \quad \hbox{ if } \partial S \neq \varnothing.
\end{array}\right.
$$

\begin{definition}
The \emph{Torelli group} of $S$ is the subgroup $\calI(S)$ of $\calM(S)$ that acts trivially on $H:= H_1(S;\Z)$.  \hfill $\blacksquare$
\end{definition}

 The study of the Torelli group, from algebraic and topological viewpoints,
was initiated by Birman in her early works, in particular \cite{Bir71,Bir74}. 
Then it was developed considerably by Johnson in the eighties: see his survey \cite{Joh83b}.
Here we shall simply review a  generating system of $\calI(S)$.

\begin{remark}
  According to Example \ref{ex:low_genus}, the Torelli group is not interesting in genus $0$ and $1$:
hence we shall  assume that the genus of $S$ is at least $2$.  $\hfill$ $\blacksquare$  
\end{remark}

First of all, let us determine the action of a Dehn twist in homology. 
For this, we need the \emph{(homology) intersection form} of the surface $S$
$$
\omega: H_1(S;\mathbb{Z}) \times H_1(S;\mathbb{Z}) \longrightarrow \Z
$$ 
which is defined as follows:
if $a=[\alpha]\in H_1(S;\mathbb{Z})$ and $b=[\beta] \in H_1(S;\mathbb{Z})$ are represented 
by smooth oriented closed curves $\alpha$ and $\beta$, in transverse position, then
$$
\omega([\alpha],[\beta]) := \sum_{x\in \alpha\cap \beta} 
\left\{\begin{array}{ll} +1, & \hbox{if $(\vec{\alpha}_x, \vec{\beta}_x)$ is direct}\\ -1, & \hbox{otherwise}  \end{array} \right\}.
$$ 
  Note that  the pairing $\omega$ is bilinear, skew-symmetric and non-singular:
thus, $\omega$ is a symplectic form on $H= H_1(S;\Z)$.   

\begin{lemma}
Let $\alpha\subset S$ be a simple closed curve. The action of the Dehn twist $T_\alpha$ in homology is given by the following  formula:
\begin{equation} \label{eq:twist_homology}
\forall x\in H, \quad (T_\alpha)_*(x) = x + \omega([\alpha],x)\cdot [\alpha].
\end{equation}
\end{lemma}

\noindent  
In other words, $(T_\alpha)_*$ is the transvection defined by the vector $[\alpha]$ and the linear form  $\omega([\alpha],-)$.
Formula \eqref{eq:twist_homology} is easily deduced from the definition of a Dehn twist. 

Here are two immediate  consequences of the transvection formula \eqref{eq:twist_homology}:
\begin{itemize}[label=$\diamond$] 
\item[(i)] for a simple closed curve $\alpha \subset S$,  we have $T_\alpha\in \calI(S)$
if and only if we have $[\alpha]=0 \in H$ (i.e$.$ $\alpha$ is separating in $S$);  
\item[(ii)] for any  simple closed curves $\alpha, \beta$ in $S$ such that $\alpha \cap \beta=\varnothing$ and $[\alpha] = [\beta]\in H$
(i.e$.$ $\alpha$ and $\beta$ cobound a subsurface of $S$)
we have $T_{\alpha}^{-1} T_\beta  \in \calI(S)$.
\end{itemize}
Following Johnson, we call an element  $T_\alpha$ of type (i) a \emph{BSCC map} (for ``Bounding Simple Closed Curve''),
and its \emph{genus} is the genus of the subsurface of $S$ bounded by $\alpha$. (If $\partial S=\varnothing$, then
there are two such subsurfaces and we take the minimal genus of those two.).
Besides, we call  an element $T_{\alpha}^{-1} T_\beta$ of type (ii) a \emph{BP map} (for ``Bounding Pair''),
and its \emph{genus}  is  the genus of the subsurface of $S$ with boundary $\alpha \sqcup \beta$.
(If $\partial S=\varnothing$ and $[\alpha] \neq 0$, then
there are two such subsurfaces and we take the minimal genus of those two.).

The following is a combination of several works, namely \cite{Bir71,Pow78,Joh79}.

\begin{theorem}[Birman 1971, Powell 78, Johnson 1979] \label{th:BPJ}
The Torelli group $\calI(S)$ has the following generating sets, 
whose nature depends on the genus $g$ and the number $n$ of boundary component  of $S$:
\begin{center} \begin{tabular}{c|c|c| }
& $ n=0 $  & $ n=1 $  \\ \hline 
$g=2$ &  BSCC maps of genus $1$ & BSCC maps of genus 1 \&  BP maps of genus $1$ \\ \hline
$g\geq 3 $&  BP maps of genus $1$  & BP maps of genus $1$ \\ \hline
\end{tabular} \end{center}
\end{theorem}

One of the major accomplishments from Johnson's works in the 80's is the fact that the group $\calI(S)$ is finitely generated
 in genus at least $3$ \cite{Joh83a}, but we will not need this fact in these lectures.
 Note that $\calI(S)$ is not finitely generated in genus~$2$~\cite{MM}.

\subsection{Torelli twists in $3$-manifolds} \label{subsec:Torelli_twists}

We fix a closed surface $R$, which may be disconnected.

\begin{definition}
Let $M \in \calV(R)$, let  $S \subset \interior(M)$  be a  compact surface with one boundary component
and let  $s\in \calI(S)$. The $3$-manifold obtained from $M$ by a \emph{Torelli twist} along $S$ with $s$ is  
\begin{equation}\label{eq:Torelli_surgery}
M_{s} := \big(M \setminus  \interior \operatorname{N}(S) \big) \cup_{\tilde s }  \operatorname{N}(S) 
\end{equation}
where  $\operatorname{N}(S)$ is  a regular neighborhood of $S$ in $M$ identified to $S\times [-1,1]$,
and $\tilde s$ is the self-diffeomorphism of $\partial(S\times [-1,1])$
 given by $s$ on $S \times \{1\}$ and  the identity elsewhere. 
With the  obvious boundary parametrization $m_s:R \to M_s$ induced by~$m$, we get $M_s \in \calV(R)$.
\hfill $\blacksquare$
\end{definition}

\noindent
Equivalently, $M_s$ is obtained by cutting open $M$ along $S$ and gluing back with~$s$:\\[0.cm]
\begin{center}
\labellist \small \hair 2pt
\pinlabel {$M$}  at 30 30
\pinlabel {$M_s$}  at 540 30
\pinlabel {$S$} [u] at 144 60
\pinlabel {$\leadsto$}  at 400 100
\pinlabel {$s$}  at 648 100
\endlabellist
\centering
\includegraphics[scale=0.35]{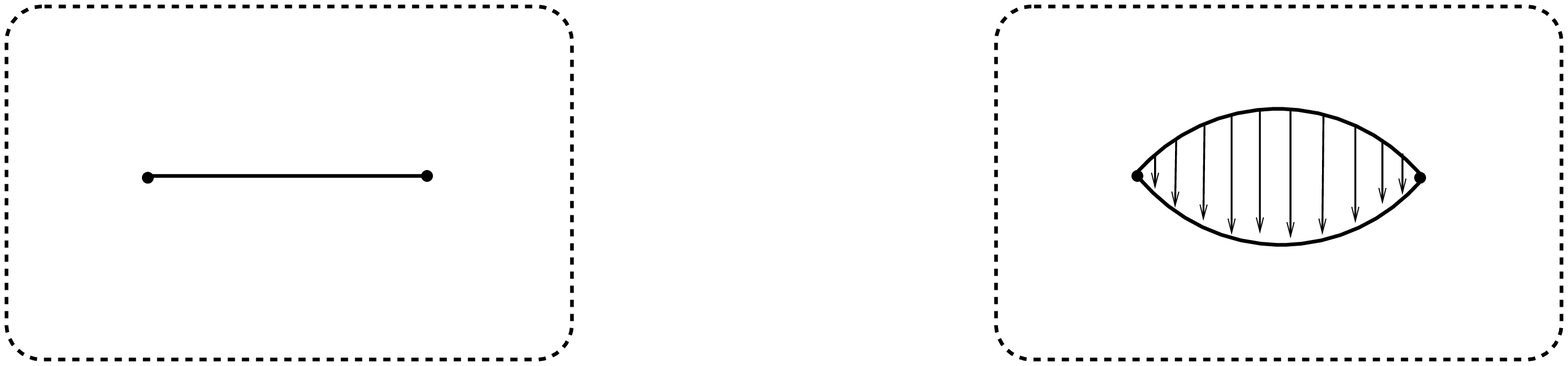}
\end{center}

\begin{definition}
Let $M, M' \in \calV(R)$. We say that $M$ and $M'$ are \emph{Torelli--equivalent}
if there is a compact surface $S\subset \interior(M)$ and an $s\in \calI(S)$ such that $M_s\cong M'$. \hfill $\blacksquare$ 
\end{definition}

\begin{lemma} \label{lem:equivalence}
The Torelli--equivalence is a non-trivial equivalence relation on $\calV(R)$.
\end{lemma}

\begin{proof} 
The Torelli-equivalence  is clearly reflexive and symetric as a relation in $\calV(R)$.
We verify the transivity by considering a first Torelli twist $M \leadsto M_s=M'$ along $S\subset M$
and a second one $M' \leadsto M'_{s'}$ along $S' \subset M'$. 
Since $S'$ deformation retracts onto a $1$-dimensional subcomplex 
and since the part $\operatorname{N}(S)  \subset M_s$ of the decomposition \eqref{eq:Torelli_surgery} is a handlebody
which also retracts to a $1$-dimensional subcomplex, we can isotope $S'$ in $M'$ 
so that it lies in the part $M \setminus \interior \operatorname{N}(S) \subset M_s$ of the decomposition \eqref{eq:Torelli_surgery}.
Hence we can view $S'$ as a  subsurface of $M$, disjoint from $S$. We attach to $S \sqcup S'$ a $1$-handle, inside $M$, to get 
a larger subsurface $T := S \sharp_\partial S'$ of $M$. We have $t := s \sharp_\partial s' \in \calI(T)$ and $M''\cong M_t$.
Hence $M''$ is Torelli-equivalent to $M$.

To prove that the Torelli--equivalence is a non-trivial relation, 
we observe that a Torelli twist $M \leadsto M_s$ induces a unique isomorphism in homology
such that the following diagram is commutative:
\begin{equation} \label{eq:psi_s}
\xymatrix{
 H_1(M; \Z) \ar@{-->}[rr]_-\simeq^-{\psi_s} &&  H_1(M_s;\Z) \\
 & \ar@{->>}[lu]^-{\operatorname{incl}_*} H_1 \big(M \setminus \interior \hbox{N}(S);\Z\big)  \ar@{->>}[ru]_-{\operatorname{incl}_*}   & 
}
\end{equation}
(The unicity follows from the surjectivity of the homomorphism $\operatorname{incl}_*$ 
induced by the inclusion $M \setminus \interior \hbox{N}(S) \hookrightarrow M$, 
and the existence is justified using the Mayer--Vietoris theorem.)
Hence two manifolds in $\calV(R)$ with different homology types can not be Torelli--equivalent.
\end{proof}

We now  give another description of the Torelli--equivalence.
Let $M\in \calV(R)$. A \emph{$Y$-graph} in $M$ is a surface $G \subset \interior (M)$ consisting of one ``node'', three ``edges'' 
and three ``leaves''  as shown on the left side of  Figure~\ref{fig:Y-graph}. The regular neighborhood of $G$ is a handlebody of genus $3$, 
inside which $G$ can be replaced by the 6-component  framed link  shown on the right side of  Figure \ref{fig:Y-graph}
(using the blackboard framing convention):  
to get this link, the node of $G$ is replaced by one copy
of the borromean rings, and each leaf of $G$ becomes a knot ``clasping'' one of those three rings.
We define
$
M_G 
$
to be the $3$-manifold obtained from $M$ by surgery along this framed link, and we call the move $$M \leadsto M_G$$ a \emph{$Y$-surgery}.
  This operation is equivalent to the ``\emph{borromean surgery}'' move that Matveev considered in \cite{Matveev}. 
Under this form, this operation was introduced by Goussarov~\cite{Go99} and Habiro \cite{Hab00} 
as part of a much larger package which is now known as ``\emph{clasper  calculus}'': see \S \ref{subsec:claspers} below.   
 
\begin{figure}[h!]
{\labellist \small \hair 0pt 
\pinlabel {$\longrightarrow$} at 460 130
\pinlabel {node} at 217 122
\pinlabel {edge} at 125 170
\pinlabel {leaf} at 110 290
\endlabellist}
\includegraphics[scale=0.3]{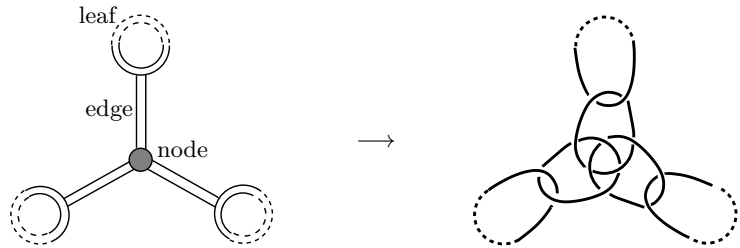}
\caption{A $Y$-graph and the associated framed link} \label{fig:Y-graph}
\end{figure}

\begin{proposition} \label{prop:Y_Y}
Two manifolds $M,M' \in \calV(R)$ are Torelli--equivalent if, and only if, 
there is a sequence of $Y$-surgeries $M=M_0 \leadsto M_1 \leadsto \cdots  \leadsto M_r=M'$.
\end{proposition}

\begin{proof}[Sketch of the proof]
In the definition of a Torelli twist $M \leadsto M_s$ along $S \subset M$, we can
assume that the genus of $S$ is  arbitrary high: indeed,  we can always take the boundary-connected sum of $S$ 
with another subsurface $U$  of $M$ (with $\partial U \cong S^1$) and extend $s$  by the identity to a diffeomorphism of $S \sharp_\partial U$.
Besides, we know from Theorem \ref{th:BPJ} that $\calI(S)$ is generated by BP maps of genus $1$ 
if the genus of $S$ is at least $3$.
Hence it is enough to show that a $Y$-surgery is equivalent to a Torelli twist $M \leadsto M_s$ 
defined by a BP map  $s$  of genus~$1$.

Using Lickorish's twist, we see that $M_s \cong M_{A \sqcup B}$ where $A \sqcup B$ is the $2$-component link in $M$
 given by the two curves $\alpha \sqcup \beta \subset S$ that define the BP map $s$, with the appropriate framings.
 Then the rest of the argument consists in showing that surgery along $A \sqcup B$ is equivalent to the surgery along a $6$-component
 framed link defining a $Y$-surgery: this is explained in \cite[Lemma 5.1]{GGP} or \cite[Fig$.$ 6.2]{Mas07}.
\end{proof}

\begin{remark} \label{rem:blink}
 A \emph{blink}  of genus $h$ in a compact $3$-manifold $M$ is a compact surface $B\subset \interior(M)$ of genus $h$ with two boundary components
$\partial B = B^+ \sqcup (-B^-)$: the knot $B^\pm$ is framed with the parallel given by the curve $\partial \hbox{N}(B^\pm) \cap B$ 
and corrected by the meridian $\pm \mu(B^\pm)$.
Surgeries along blinks have been considered in \cite{Hilden,Matveev} and \cite{GL97}, where the term ``blink'' was coined.
As in the proof of  Proposition~\ref{prop:Y_Y}, we deduce from Lickorish's trick  that  surgery along a blink
is equivalent to a Torelli twist with a BP map of  the same genus.
 Thus two manifolds $M,M' \in \calV(R)$ are Torelli--equivalent if, and only if, one can find
 a disjoint union $B= \sqcup_i B_i$ of blinks in $M$ such that $M_B \cong M'$.  
\hfill $\blacksquare$
\end{remark}

Finally, we give another description of the Torelli--equivalence in terms of Heegaard splittings.
However, we only formulate this description for the two instances of a surface $R$ that we shall consider later: 
\begin{itemize}[label=$\diamond$] \label{page:R}
\item[(i)] $R=\varnothing$: then $\calV(R)$ consists of closed $3$-manifolds;
\item[(ii)] $R= \partial(\Sigma \times [-1,1])$ where $\Sigma$ is a compact surface with $\partial \Sigma \cong S^1$:
then $\calV(R)$ consists of cobordisms (with ``vertical'' boundary) from $\Sigma$ to $\Sigma$.
\end{itemize}

The notion of ``Heegaard splitting''  in the case (i) has been seen in Example~\ref{ex:Heegaard},
and it can be reformulated  as follows. A \emph{Heegaard splitting} of genus $g$  of a closed $3$-manifold $M$ is a decomposition $M = M_- \cup M_+$
where $M_-$ and $M_+$ are two copies of the  handlebody $H_g$ in $M$ such that $M_-\cap M_+= \partial M_\pm$
(which is called the \emph{Heegaard surface}).

Likely, the notion of ``Heegaard splitting'' in the case (ii) is defined as follows. Let $M$ be a cobordism from $\Sigma$ to $\Sigma$.
We set $\partial_\pm M := m(\Sigma\times \{\pm 1\} )$, and we denote a collar neighborhood of $\partial_-M$ 
(resp$.$ $\partial_+ M$) simply by $\partial_- M \times [-1,0]$ (resp. $\partial_+ M \times [0,1]$).
A \emph{Heegaard splitting} of $M$  of \emph{genus} $g$ is a decomposition 
$
M = M_- \cup M_+,
$
where $M_-$ is obtained from $\partial_- M \times [-1,0]$ by adding $g$ $1$-handles along $\partial_- M \times \{0\}$, 
$M_+$ is obtained from $\partial_+ M \times [0,1]$ by adding $g$ $1$-handles along $\partial_+ M \times \{0\}$,
and we have $M_-\cap M_+ = \partial M_- \cap \partial M_+$ (which is called the \emph{Heegaard surface}).
The existence of Heegaard splittings in this situation (cobordisms with ``vertical'' boundary) is again an application of Morse theory.

\begin{proposition} \label{prop:Heegaard}
Assume that $R$ is of one of the above  types (i) and (ii).
 Two manifolds $M,M' \in \calV(R)$  are Torelli--equivalent if, and only if, there is a Heegaard splitting $M=M_- \cup M_+$
 with Heegaard surface $S$ and an $s\in \calI(S)$ such that $M' \cong M_- \cup_s M_+$.
\end{proposition}

\begin{proof}
We only prove the proposition in the case (i), the case (ii) being similar and a little bit more technical (see \cite[Lemma 2.1]{MM13} for instance).
It is enough to show that, given a closed $3$-manifold~$M$ and a surface $E \subset M$ with one boundary component,
we can always find a Heegaard splitting $M=M_-\cup M_+$ whose Heegaard surface contains  a subsurface that is isotopic to $E$ in $M$.

Let $\hbox{N}(E)$ be a regular neighborhood of $E$ in $M$ and set $\tilde M := M \setminus \interior \hbox{N}(E)$.
Viewing $\tilde M$ as a cobordism from $\varnothing$ to $ \partial   \hbox{N}(E)$, we can find a handle decomposition
$$
\tilde M_0 \subset \tilde M_1  \subset \tilde M_2 = \tilde M
$$
where $\tilde M_0$ consists of a single $0$-handle, $\tilde M_1$ is obtained from $\tilde M_0$ by attaching $1$-handles
and $\tilde M_2$ is obtained from $\tilde M_1$ by attaching $2$-handles. The latter can be viewed, dually, as $1$-handles attached
to $\hbox{N}(E)$ inside $M$. Hence there is a Heegaard splitting $M=M_- \cup M_+$ where
$$
M_- := \tilde M_1 \quad \hbox{and} \quad M_+ := \big(\tilde M_2 \setminus \interior(\tilde M_1) \big) \cup \hbox{N}(E).
$$
Observe that $E$ can be isotoped  in $\hbox{N}(E)$ onto $\partial \hbox{N}(E)$; furthermore, since $E$ deformation retracts onto
a $1$-dimensional subcomplex, we can next isotope it in $\partial \hbox{N}(E)$ to make it disjoint 
from the attaching locus of the $1$-handles attached to $ \hbox{N}(E)$. Thus we have isotoped $E$ to 
a subsurface of the Heegaard surface.
\end{proof}

\subsection{Filtrations on the Torelli groups} \label{subsec:filtrations}

We will define surgery equivalence relations for $3$-manifolds which are much stronger than the Torelli--equivalence
and arise from certain filtrations of the Torelli group.

To define these filtrations, we first recall that the \emph{lower central series} of a group~$G$ 
is the decreasing sequence of subgroups 
\begin{equation} \label{eq:LCS}
G=\Gamma_1 G \supset \Gamma_2 G \supset \Gamma_3 G \supset \cdots
\end{equation}
that are defined inductively by $\Gamma_{i+1}G:= \left[\Gamma_i G,G\right]$ for all $i\geq 1$.
Let $S$ be a compact surface with one boundary component, and fix a base-point $\star\in \partial S$.
The canonical action of $\calI(S)$ on the fundamental group $\pi:= \pi_1(S,\star)$ induces, 
for every integer $k\geq 1$, a group homomomorphism
\begin{equation} \label{eq:rho_k}
\rho_k : \calI(S) \longrightarrow \Aut(\pi/\Gamma_{k+1} \pi)
\end{equation}
since $\Gamma_{k+1} \pi$ is a characteristic subgroup of $\pi$.
Defining $J_k  \calI(S)  := \ker(\rho_k)$ for every~$k\geq 1$, we get a filtration of the Torelli group
$$
\calI(S) = J_1 \calI(S)  \supset J_2 \calI(S)  \supset J_3 \calI(S) \supset \cdots 
$$ 
which is nowdays  refered to as the \emph{Johnson filtration} of $\calI(S)$. 
  The study of the Johnson filtration on its whole started in Morita's seminal work \cite{Mor93},
and it is still  an active field of research. (See \cite{Satoh} for a survey.)  

\begin{example} \label{ex:Johnson}
Johnson made a deep study of  the second term of the filtration
$$
\calK(S) :=  J_2 \calI(S) 
$$
in \cite{Joh85a,Joh85b}, so much that this group is called the \emph{Johnson subgroup} (or the \emph{Johnson kernel}).
In particular, Johnson proved that $\calK(S)$ is generated by BSCC maps. \hfill $\blacksquare$
\end{example}

One of the main reasons to be interested in this filtration is that it has a trivial intersection
$$
\bigcap_{k\geq 1} J_k \calI(S) = \{1\}
$$
as can be easily checked from the following two classical facts: 
\begin{itemize}[label=$\diamond$]
\item[(i)] (Baer 1928) the canonical action of $\calI(S)$ on $\pi$ is faithful \cite{Baer}; \label{page:Baer}
\item[(ii)] (Magnus 1937) the lower central series of  $\pi$  has a trivial intersection, because $\pi$ is free  \cite{Magnus}.
\end{itemize}
  Thus, one  of the main objectives of the study of the Johnson filtration is to fully understand its associated graded, namely
$$
\hbox{Gr}^J\, \calI(S) = \bigoplus_{k\geq 1} \frac{J_k \calI(S)}{J_{k+1} \calI(S)}.
$$  
Another interesting feature of the Johnson filtration is that it is \emph{strongly central} in the sense that
\begin{equation} \label{eq:N-series}
\forall k,l \in \N^*, \quad \big[J_k \calI(S), J_l \calI(S) \big] \subset J_{k+l} \calI(S)
\end{equation}
  (see \cite[Prop. 4.1]{Mor93}).  Consequently, the commutator operation in the group $\calI(S)$ 
induces a Lie ring structure on $\hbox{Gr}^J \calI(S)$, which opens the door to Lie-theoretical methods in the study of $\calI(S)$. 
 (Again, see \cite{Satoh} for a survey.) 

The Johnson filtration has also been much studied in relation with the lower central series 
$ \calI(S) =\Gamma_1  \calI(S)  \supset \Gamma_2  \calI(S) \supset \Gamma_3  \calI(S) \supset \cdots$ of the Torelli group.
Indeed, \eqref{eq:N-series} implies that the latter is contained in the former:
\begin{equation} \label{eq:inclusion}
\forall k \in \N^*, \quad \Gamma_k \calI(S) \subset J_k \calI(S).
\end{equation}  
The associated graded of the lower central series of the Torelli group
\begin{equation} \label{eq:Gr_Gamma}
\hbox{Gr}^\Gamma \, \calI(S) = \bigoplus_{k\geq 1} \frac{\Gamma_k \calI(S)}{\Gamma_{k+1} \calI(S)}
\end{equation}
has been determined  with rational coefficients by Hain \cite{Hain}, as part of the stronger result
of identifying the Malcev Lie algebra of $\calI(S)$. 
For a comparison between $\hbox{Gr}^\Gamma \, \calI(S)\otimes \Q$ and $\hbox{Gr}^J \, \calI(S) \otimes \Q$ in low degrees,
see \cite{Mor98,MSS}.  

\begin{remark}
 \label{rem:Hain}
  Hain also obtained in  \cite{Hain} that the inclusion reciprocal to \eqref{eq:inclusion} is not true: 
specifically, there is no $d\in \N^*$ such that $J_d \calI(S) \subset \Gamma_3 \calI(S)$.   \hfill $\blacksquare$
\end{remark}

The above paragraphs only give a brief and limited overview of what is known about the Johnson filtration and the lower central series
of the Torelli group. We conclude this subsection with an  informal “comparison table” between those two filtrations:

\begin{center}
{\small
\begin{tabular}{c|c|c}
& lower central series  $\big(\Gamma_k \calI(S)\big)_k$ &  {Johnson filtration} $\big(J_k \calI(S)\big)_k$   \\ \hline
trivial intersection? & yes & yes \\ \hline
testing elements ? &  given $h\in \calI(S)$ and $k\in \N^*$, & given $h\in \calI(S)$ and $k\in \N^*$, \\
& it is hard to decide & it is easy to decide \\
& whether  $h \in \Gamma_k \calI(S)$ &    whether  $h \in J_k \calI(S)$ \\ 
& unless $k$ is small (say $k\leq 3$)  & using  ``Johnson homomorph.''\\ \hline
explicit generators ? & it is easy to deduce an explicit &  it seems difficult  to construct   \\
& generating syst.  in any degree $k$ &an explicit generating  syst. \\
& from a generating syst. of  $\calI(S)$  & in  a given degree $k$ \\ \hline
finitely generated? & yes, in any degree $k$:   & yes, in any degree $k$:  \\
& if $g\geq 3$ for $k=1$ \cite{Joh83a}  &  if $g\geq 3$ for $k=1$ \cite{Joh83a}  \\
& if $g \geq 4$ for $k = 2$ \cite{EH,CEP} &  if $g \geq 4$ for $k = 2$ \cite{EH,CEP}   \\
&  if $g \geq 2k-1 $ for $k\geq 3$ \cite{CEP}  &   if $g \geq 2k-1 $ for $k\geq 3$ \cite{CEP}
\end{tabular}
}
\end{center}

\subsection{Stronger surgeries in $3$-manifolds} \label{subsec:Y_and_J}

We are now in position to introduce two families of surgery equivalence relations that refine the Torelli--equivalence.
We fix a closed surface $R$, which may be disconnected.

\begin{definition}
Let $k\in \N^*$.
Two $3$-manifolds $M,M' \in \calV(R)$ are  \emph{$J_k$-equivalent} (resp$.$  \emph{$Y_k$-equivalent})
if $M'$ can be obtained from $M$ by a Torelli twist $M\leadsto M_{s}$ along a surface $S\subset \hbox{int}(M)$
with an $s\in J_{k} \calI(S)$ (resp$.$ an  $s\in \Gamma_{k} \calI(S)$).  \hfill $\blacksquare$ 
\end{definition}

Of course, the $J_1$-equivalence and $Y_1$-equivalence are just the same as the Torelli-equivalence.

\begin{lemma}
For every $k\in \N^*$, the {$J_k$-equivalence} (resp.  the {$Y_k$-equivalence}) is an equivalence relation in $\calV(R)$.
\end{lemma}

\begin{proof}
  We come back to the proof of Lemma \ref{lem:equivalence}, using the same notations.

If we have $s\in J_k \calI(S)$ and $s'\in J_k \calI(S')$, then $s\sharp_\partial s'$ belongs to $J_k \calI(S \sharp_\partial S')$
as can be checked from the fact that $\pi_1(S \sharp_\partial S')$ 
is the free product of $\pi_1(S )$ and $\pi_1(S')$. This proves the transitivity of the $J_k$-equivalence.

If we have $s\in \Gamma_k \calI(S)$ and $s'\in \Gamma_k \calI(S')$, then $s\sharp_\partial s'$ belongs to $\Gamma_k \calI(S \sharp_\partial S')$
as follows from the fact that $s\sharp_\partial s'  = (s\sharp_\partial \operatorname{id}) \circ (  \operatorname{id} \sharp_\partial s' )$.
This proves the transitivity of the $Y_k$-equivalence.  
\end{proof}

\begin{remark}
  Proposition \ref{prop:Heegaard} can also be refined to reformulate the {$J_k$-equivalence} (resp$.$  the {$Y_k$-equivalence}) 
in terms  of Heegaard splittings.   \hfill $\blacksquare$ 
\end{remark}

We deduce from \eqref{eq:inclusion} the following ``ladder'' of equivalence relations:
$$
\begin{array}{cccccccccccc}
Y_1 & \Longleftarrow & Y_2 & \Longleftarrow & Y_3 & \Longleftarrow & \cdots & Y_k &  \Longleftarrow & Y_{k+1} & \Longleftarrow & \cdots  \\
  \parallel & & \Downarrow & & \Downarrow && &  \Downarrow & & \Downarrow && \\
J_1 & \Longleftarrow & J_2 & \Longleftarrow & J_3 & \Longleftarrow & \cdots & J_k &  \Longleftarrow & J_{k+1} & \Longleftarrow & \cdots 
\end{array}
$$
  Note that the converse of the implication ``$Y_k \Rightarrow J_k$'' is not true.
Specifically, there is no $d \in \N^*$ such that ``$J_d \Rightarrow  Y_3$'': 
this can be easily deduced from Hain's result mentioned in Remark \ref{rem:Hain}.

After $Y_1=J_1$, the next equivalence relation to consider is the $J_2$-equivalence. Let us give an alternative description
in terms of surgeries along knots. Given $M \in \calV(R)$ and a null-homologous knot $K\subset \interior(M)$, 
there is a unique parallel $\rho_0(K) \subset \partial \hbox{N}(K) $ that is null-homologous in $M\setminus K$: for any $n\in \Z$, 
the knot $K$ is said to be \emph{$n$-framed} if it is equipped with the unique parallel $\rho_n(K)$ 
that represents the homology class $n[\mu(K)] + [\rho_0(K)] \in H_1\big(  \partial \hbox{N}(K);\Z\big)$.
(Here, we fix an orientation of $K$, we orient $\rho_0(K)$ compatibly with $K$ and orient $\mu(K)$ with the right-hand rule 
using the orientation of $M$.) Following Cochran, Gerges and Orr \cite{CGO}, we say that 
an $M\in \calV(R)$ is \emph{$2$-surgery equivalent} to an $M'\in \calV(R)$ if there is a finite sequence 
$$
M=M_0 \leadsto M_1 \leadsto \cdots  \leadsto M_r=M' \label{page:2-surgery}
$$
of surgeries along null-homologous $(\pm 1)$-framed knots.

\begin{proposition} \label{prop:J_2}
The $J_2$-equivalence is the same as the $2$-surgery equivalence.
In particular, the $2$-surgery equivalence is an equivalence relation in $\calV(R)$.
\end{proposition}

\begin{proof}
Assume that $M,M'$ are $J_2$-equivalent: then there is a surface $S \subset \interior(M)$ and an $s\in J_2 \calI(S)$
such that $M'\cong M_s$. According to what has been mentioned in Example \ref{ex:Johnson}, $s$ decomposes as a product of BSCC maps (or their inverses). 
Thus, by considering parallel copies of $S$, we find a  finite sequence
$$
M=M_0 \leadsto M_1 \leadsto \cdots  \leadsto M_r=M'
$$
where each move $M_i \leadsto M_{i+1}$ is a Torelli twist defined by a BSCC map (or its inverse). By Lickorish's trick,
such a move can be interpreted as a surgery along a null-homologous $(\pm 1)$-framed knot.
So $M$ is $2$-surgery equivalent to $M'$.

Assume now that $M$ is $2$-surgery equivalent to $M'$. We wish to prove that $M$ and $M'$ are $J_2$-equivalent.
By transitivity of $J_2$, we can assume that $M'$ is obtained from $M$ 
by a single surgery along a null-homologous $(\pm 1)$-framed knot $K\subset M$.
There is a \emph{Seifert surface} for $K$ in $M$, i.e. a compact surface $\Sigma$ such that $\partial \Sigma =K$.
The regular neighborhood $\hbox{N}(\Sigma)$ is a handlebody, in which $K$ can be viewed 
as a push-off of a bouding simple closed curve $\gamma \subset \partial \hbox{N}(\Sigma)$. Then, by Lickorish's trick,
$M'=M_K$ is diffeomorphic to $\big(M\setminus \interior \hbox{N}(\Sigma)\big) \cup_\tau  \hbox{N}(\Sigma)$
where $\tau:=T_\gamma$. Hence $M'$ is the result of the Torelli twist $M \leadsto M_s$ along the surface $S$ 
obtained from $ \partial \hbox{N}(\Sigma)$ by cutting a small open disk,  with $s:= T_\gamma\in J_2 \calI(S)$.
\end{proof}

\begin{remark} \label{rem:boundary_links}
  A \emph{boundary link}  in a compact $3$-manifold $M$ is a framed link $L=\sqcup_i L_i$ 
for which there exists a compact surface $S= \sqcup_i S_i \subset \interior(M)$ with as many connected components 
as $L$, such that $\partial S_i =L_i$ and the parallel of $L_i$ differs from the curve $\partial \hbox{N}(L_i) \cap S_i$ by $\pm \mu(L_i)$.
Surgeries along boundary links have been considered in \cite{Matveev,GL97,CGO}, for instance.
The argument used in the proof of Proposition~\ref{prop:J_2}  shows that  surgery along a boundary link
is equivalent to the simultaneous realization of Torelli twists by BSCC maps on pairwise-disjoint surfaces.
 Thus two manifolds $M,M' \in \calV(R)$ are $J_2$-equivalent if, and only if, one can find
 a  boundary link $L$ in $M$ such that $M_L \cong M'$.  
\hfill $\blacksquare$
\end{remark}

In general, Cochran, Gerges and Orr make  in \cite{CGO} the following definition for any integer $k\geq 2$.

\begin{definition} \label{def:k-surg}
A manifold $M\in \calV(R)$ is \emph{$k$-surgery equivalent} to an $M'\in \calV(R)$ if there is a finite sequence 
$$
M=M_0 \leadsto M_1 \leadsto \cdots  \leadsto M_r=M'
$$
where each move $M_i \leadsto M_{i+1}$ is the surgery along a $(\pm 1)$-framed knot $K_i$ that is trivial in $\Gamma_k \pi_1(M_i)$.
\hfill  $\blacksquare$
\end{definition}

It turns out that the $k$-surgery equivalence is indeed an equivalence relation \cite[Cor. 2.2 \& Prop. 2.3]{CGO}.
But  $k$-surgery equivalence is very different from $J_k$-equivalence in higher degree $k$:
while the former is rather well understood, the latter still remains unexplored (see \S \ref{subsec:higher}). 
In fact, since one does not know explicit generating systems for the Johnson filtration,
it seems that one does not know generators for the $J_k$-equivalence relation for $k>2$.

\subsection{Clasper calculus} \label{subsec:claspers}

In contrast with the $J_k$-equivalence, explicit generators are known for the $Y_k$-equivalence:
these are defined  in terms of ``surgeries'' along 
certain framed graphs, and generalize in degree $k>1$ the $Y$-surgeries that have been recalled in \S \ref{subsec:Torelli_twists}.
These surgery techniques were developed independently by Goussarov \cite{Go99,Go00} and Habiro \cite{Hab00}.

We give a very brief overview of those techniques, using Habiro's terminology and conventions.
Let $M\in \calV(R)$. A {\em graph clasper}  in $M$ is a (possibly disconnected) compact surface $G \subset \interior(M)$, 
which is decomposed into {\em leaves}, {\em nodes} and {\em edges}.  
Leaves are copies of the annulus $S^1 \times D^1$  and nodes are copies of the disc $D^2$. 
Edges are $1$-handles (i.e$.$ copies of $D^1 \times D^1$)  connecting those leaves and nodes;
the \emph{ends} of an edge constitute the attaching locus of the $1$-handle (i.e$.$ $S^0 \times D^1$).
There are two rules to respect in the attachment: each leaf receives exactly one end of an edge,
and each node receives exactly three ends of edges. The \emph{degree} of $G$ is the number of  its nodes.
The \emph{shape} of $G$ is the abstract graph, whose vertices have valency 1 or 3, 
onto which $G$ deformation retracts after deletion of all of its leaves.

\begin{example}
Graph claspers of degree $0$ (and shape $\textsf{I}$) are called \emph{basic claspers} and consist of only one edge and two leaves:
$$
\includegraphics[scale=0.3]{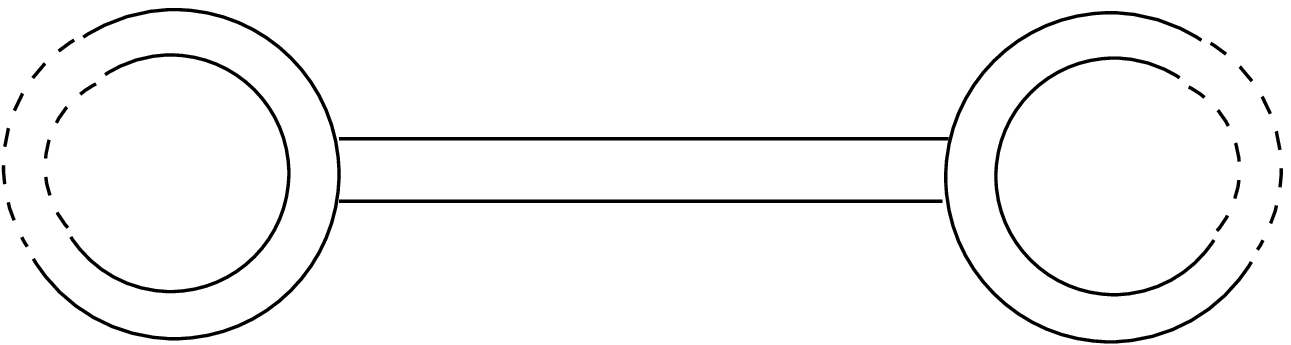}
$$
A connected graph clasper of degree $1$  (and shape $\textsf{Y}$) is a $Y$-graph, as shown in Figure \ref{fig:Y-graph}.
Here is an example of a connected graph clasper of degree $3$:
$$
\includegraphics[scale=0.25]{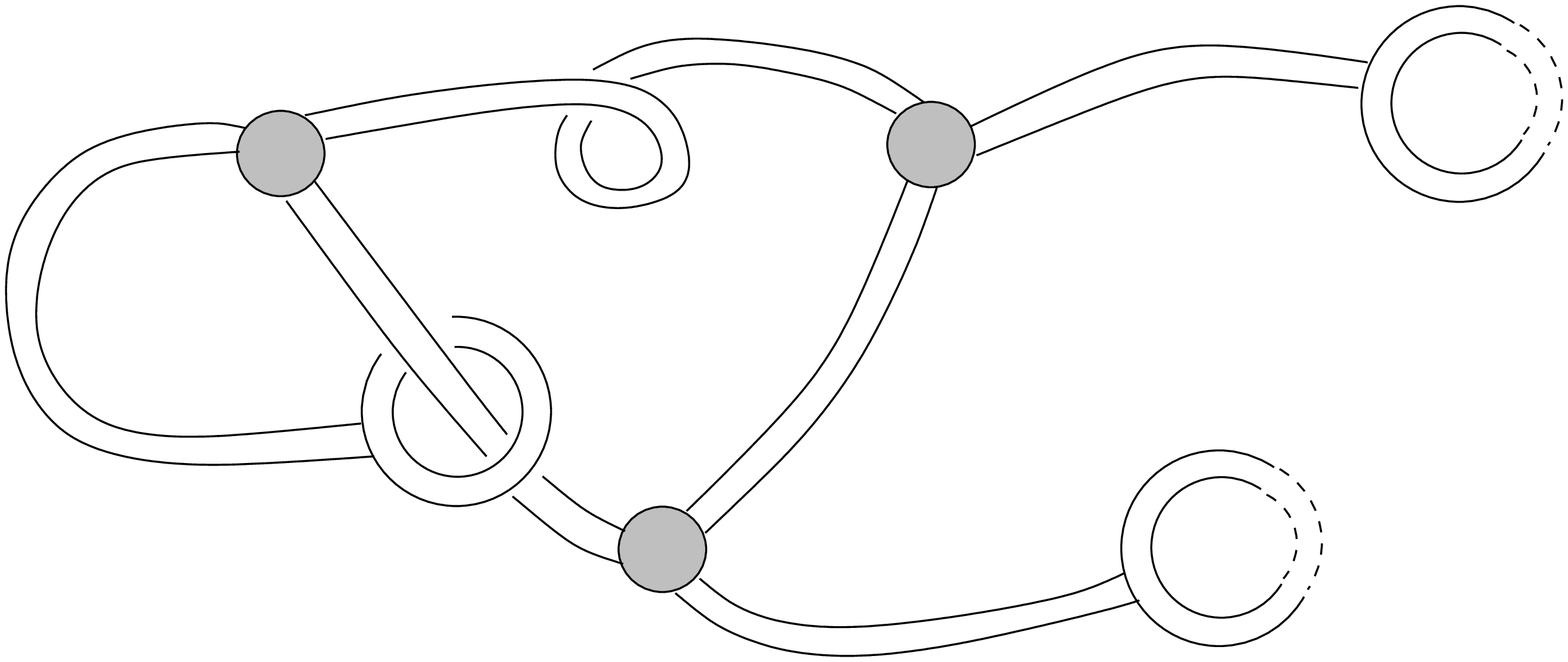}
$$
\hfill $\blacksquare$
\end{example}

Surgery along a graph clasper $G\subset \interior(M)$ is defined as follows.  
We first replace each node with three leaves in a ``Borromean rings'' fashion: 
\begin{center}
{\labellist \small \hair 0pt 
\pinlabel {$\longrightarrow$} at 409 143
\endlabellist}
\includegraphics[scale=0.3]{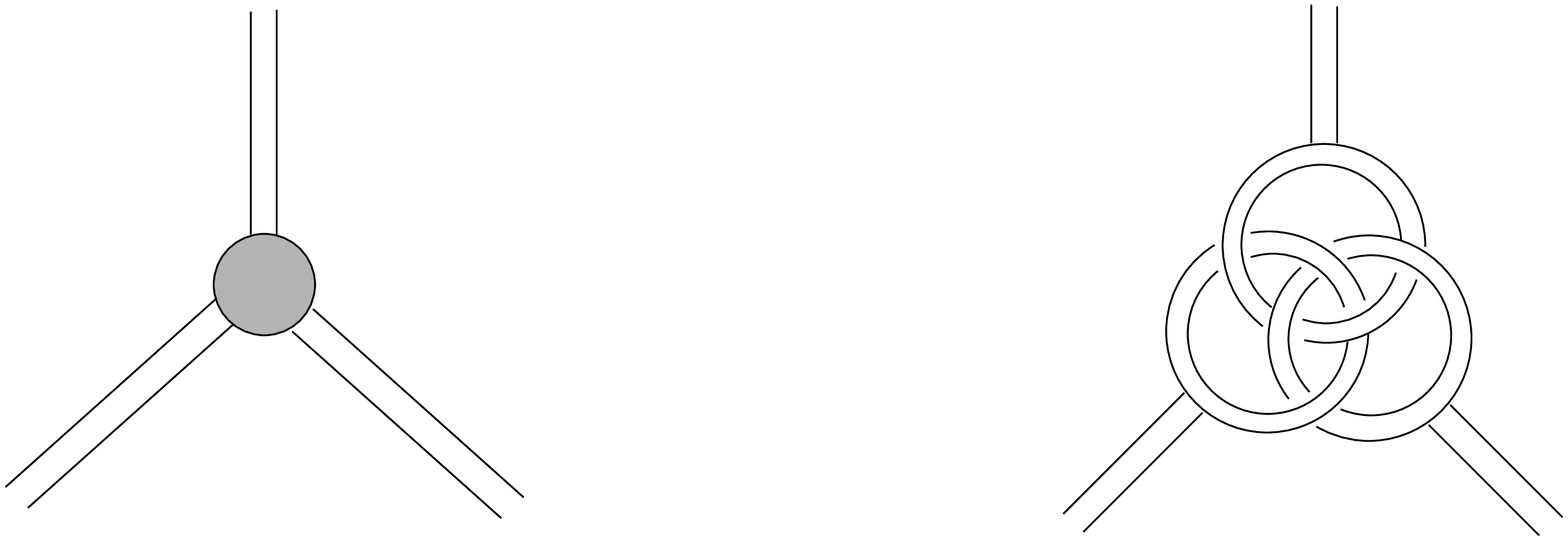}
\end{center}
This results in a disjoint union of basic claspers,   
which we replace by $2$-component framed links  as follows:
\begin{center}
{\labellist \small \hair 0pt 
\pinlabel {$\longrightarrow$} at 460 50
\endlabellist}
\includegraphics[scale=0.35]{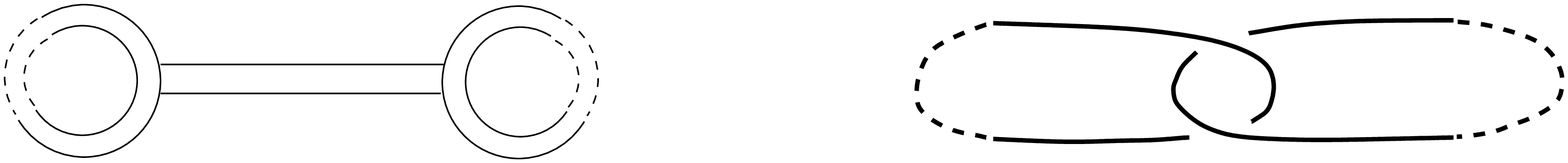}
\end{center}
(For instance, if we start from a $Y$-graph $G$, then we recover the $6$-component framed link shown in Figure  \ref{fig:Y-graph}.)
Then, the \emph{surgery} $M \leadsto M_G$ along  $G$ is defined as the surgery along the resulting framed link in $M$,
and we have the following generalization  of Proposition \ref{prop:Y_Y}:

\begin{proposition}[Habiro 2000] \label{prop:Y_clasper}
For any integer $k\geq 1$, the $Y_k$-equivalence relation is generated by surgeries along  connected graph claspers of degree $k$.
\end{proposition}

\noindent
  See  \cite{Hab00}, and the appendix of \cite{Mas07} for a proof. Note that the $Y_k$-equivalence appears in
the works of Goussarov and Habiro under different names: it is named  ``$(k-1)$-equivalence'' in \cite{Go99}
and ``$A_k$-equivalence'' in \cite{Hab00}.

There exists a \emph{clasper calculus}, which has been developed in \cite{Go00,Hab00,GGP}.
This calculus can be regarded as a braided version of  the commutator calculus in groups or,
to be more accurate, an instance of the braided Hopf-algebraic calculus. In the setting of \cite{Hab00},
there is a notion of ``clasper'', which is more general than the above notion of ``graph clasper'', 
and there are 12 ``moves'' which can be applied to claspers without changing the diffeomorphism types of the resulting manifolds.

Thanks to Proposition \ref{prop:Y_clasper}, this clasper calculus
can be used to show that certain operations $G \leadsto G'$ on  graph claspers will not change the $Y_\ell$-equivalence class
of the resulting manifold, for  $\ell$ large enough depending on the degrees of the components of $G$ and the nature of the operation.
Thus, these operations  are very useful tools to study
sets of $Y_k$-equivalence classes up to $Y_\ell$-equivalence for some~${\ell>k}$.

  Here are some instances of such operations on graph claspers, taking place in a manifold  $M \in \calV(R)$ which we fix from now on:\\
\begin{itemize}[label=$\diamond$]
\item[$(\mathcal{O}_0)$] \textbf{Cutting an edge.} 
 Any graph clasper $G$ can be transformed to a graph clasper $G'$ (of the same degree, but not the same shape)
 by insertion of a Hopf link of two leaves at the middle of an edge:
\begin{center}
{\labellist \small \hair 0pt 
\pinlabel {$G$} at 260 70
\pinlabel {$G'$} at 740 70
\pinlabel {$\cong$} at 430 40
\endlabellist}
\includegraphics[scale=0.35]{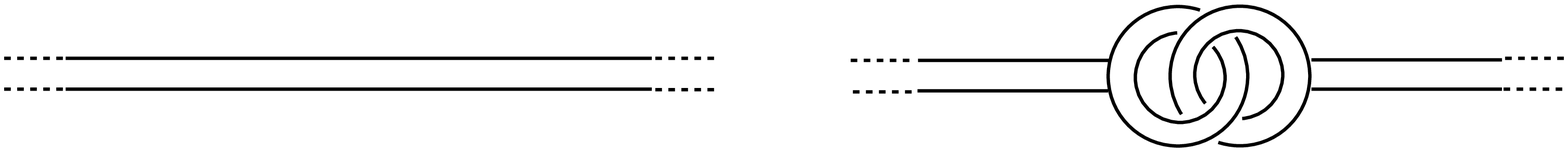}
\end{center}
(In fact,  this operation is Habiro's ``Move 2'' \cite{Hab00}.)\\
\item[$(\mathcal{O}_1)$] \textbf{Developing a node.} \label{page:dev_node}
 Any graph clasper $G$ of degree $k+1$, showing one node incident to two leaves, 
 can be transformed to a graph clasper $G'$ of degree~$k$ by the following transformation:
\begin{center}
{\labellist \small \hair 0pt 
\pinlabel {$G$} at 200 70
\pinlabel {$G'$} at 620 70
\pinlabel {$\cong$} at 370 110
\endlabellist}
\includegraphics[scale=0.42]{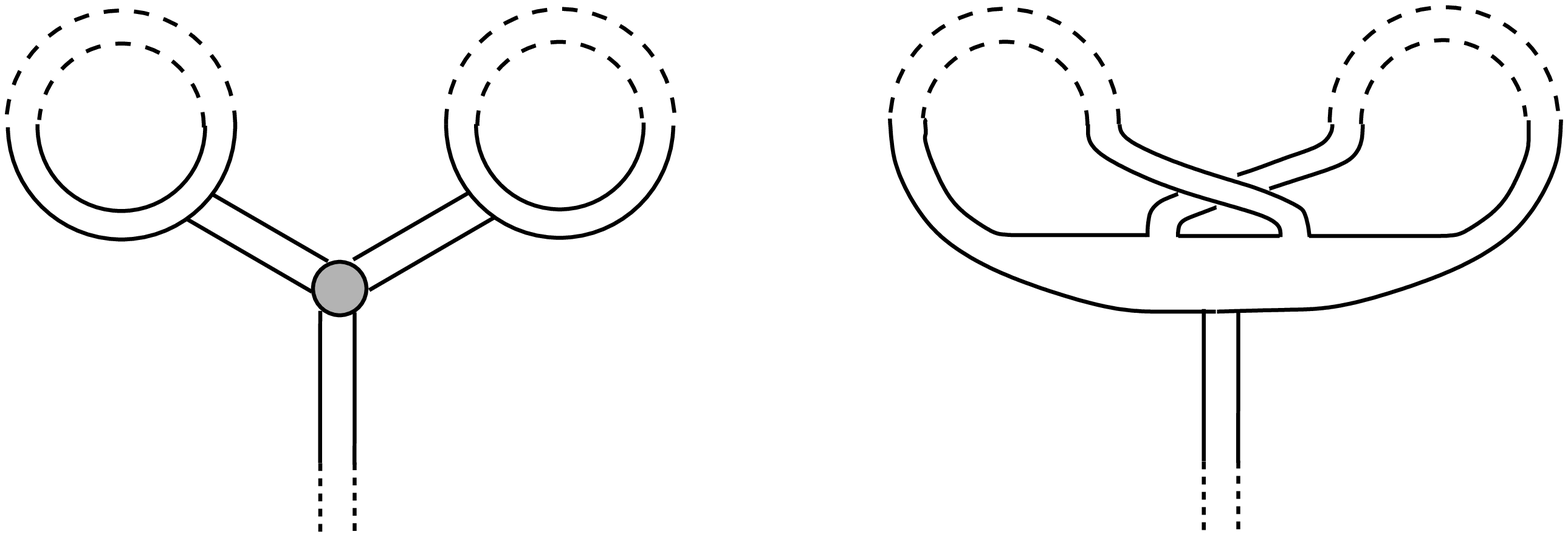}
\end{center}
(In fact,  this operation is essentially Habiro's ``Move 9'' \cite{Hab00}.)\\
\item[$(\mathcal{O}_2)$] \textbf{Sliding an edge.} If $G$ is a connected graph clasper of degree $k$ in $M$ 
and if $G'$ is obtained from $G$ by sliding one of its edges along
a disjoint framed knot $K$, then we have $M_G \sim_{Y_{k+1}} M_{G'}$:
\begin{center}
{\labellist \small \hair 0pt 
\pinlabel {$G$} at 320 140
\pinlabel {$K$} at 200 90
\pinlabel {$G'$} at 900 140
\pinlabel {$\sim_{Y_{k+1}}$} at 510 90
\endlabellist}
\includegraphics[scale=0.3]{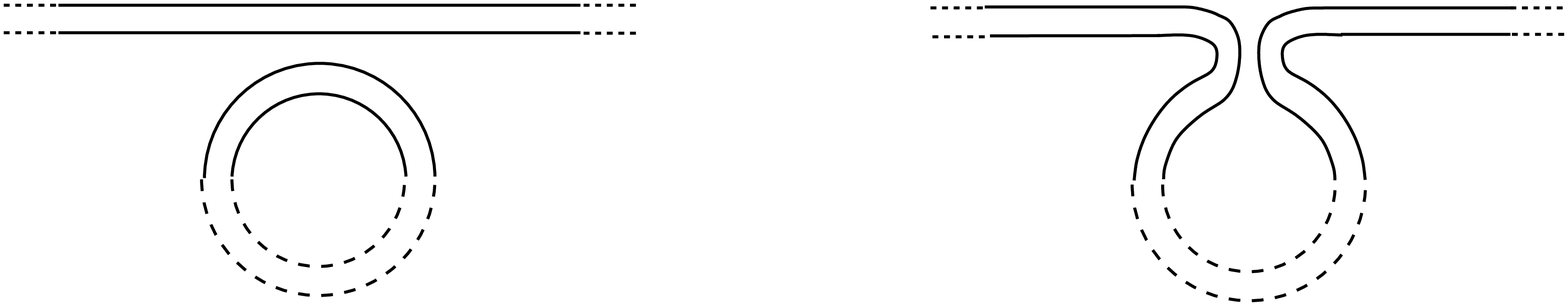}
\end{center}
\item[$(\mathcal{O}_3)$] \textbf{Cutting a leaf.} If $G$ is a connected graph clasper of degree $k$ in $M$ with a leaf $L$ decomposed as $L=L_1 \sharp L_2$,
then  we have $M_G \sim_{Y_{k+1}} M_{G_1 \sqcup G_2}$ where $G_i$ is $G$ with the leaf $L$ replaced by the ``half-leaf'' $L_i$
and $G_1 \sqcup G_2$ is a  disjoint union of $G_1$ and $G_2$:
\begin{center}
{\labellist \small \hair 0pt 
\pinlabel {$G$} at 190 70
\pinlabel {$L$} at 160 144
\pinlabel {$G_1$} at 545 60
\pinlabel {$G_2$} at 720 60
\pinlabel {$L_1$} at 525 185
\pinlabel {$L_2$} at 740 185
\pinlabel {$\sim_{Y_{k+1}}$} at 400 110
\endlabellist}
\includegraphics[scale=0.3]{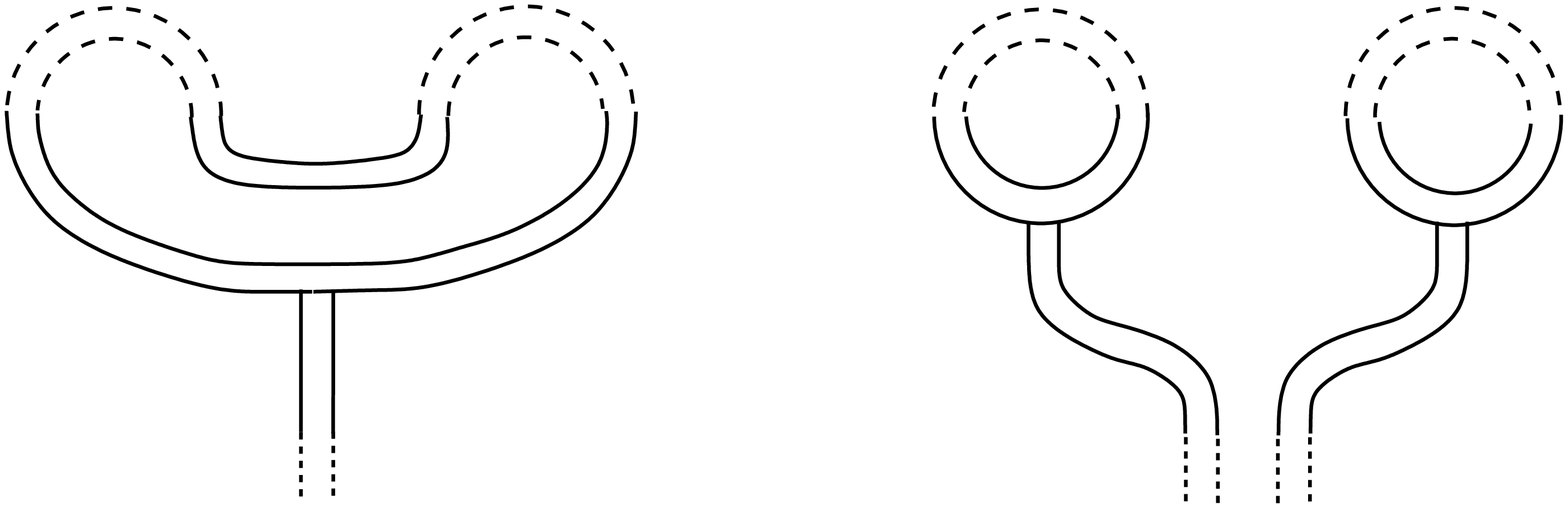}
\end{center}
\item[$(\mathcal{O}_4)$] \textbf{Crossing a leaf with a leaf.}
If $G_1 \sqcup G_2$ is the disjoint union of two connected graph claspers in $M$ of degrees $k_1$ and $k_2$, respectively,
and if $G'_1 \sqcup G'_2$ is obtained from $G_1 \sqcup G_2$ by crossing a leaf of $G_1$ with a leaf of~$G_2$,
then we have  $M_{G_1 \sqcup G_2} \sim_{Y_{k_1+k_2}} M_{G'_1 \sqcup G'_2}$.\\
\item[$(\mathcal{O}_5)$] \textbf{Half-twisting an edge.}
If $G$ is a connected graph clasper of degree $k$ in $M$ and if $G^-$ is obtained from $G$ by adding a half-twist to an edge,
then there is a disjoint union $G \sqcup G^-$ of $G$ and $G^-$ in $M$ such that $M_{G \sqcup G^-} \sim_{Y_{k+1}} M$.
\end{itemize}

\begin{remark}
References for the above operations on graph claspers include \cite{Hab00} (in the case of links instead of $3$-manifolds),
\cite{Go00}, \cite{GGP}, \cite{Gar02}, \cite[\S E]{Oht02} and \cite{Mas05}. \hfill $\blacksquare$
\end{remark}

In the rest of this subsection, we outline the \emph{general strategy} to  study the $Y_\ell$-equivalence relations
using the above techniques of clasper calculus.
So, let us assume that we have been able  to classify the $Y_{k}$-equivalence relation 
on $\calV(R)$ for some $k\geq 1$, and that we now wish to classify
the $Y_{k+1}$-equivalence on  a specific  $Y_{k}$-equivalence class $$ \calV_0 \subset \calV(R).$$ 
For this, we fix a $3$-manifold $V \in \calV_0$ and we consider the free abelian group $\Z\!\cdot\! C_k$ generated by the set
$$
C_k := \big\{ \hbox{connected graph claspers in $V$ of degree $k$}\big\}\big/ \hbox{isotopy}.
$$ 
Then we consider the map \label{page:method}
$$
\psi_k: \Z\!\cdot\! C_k \longrightarrow \frac{\calV_0}{Y_{k+1}} ,
\quad \sum_i \varepsilon_i G_i \longmapsto \big[V_{(\sqcup_i G_i^{\varepsilon_i})}\big].  
$$
where, for a family $(G_i)_i$ of connected graph claspers of degree $k$ in $V$ weighted by a family of signs $(\varepsilon_i)_i$,
we choose an \emph{arbitrary} disjoint union  $\sqcup_i G_i^{\varepsilon_i}$  of the graph claspers $G^{\varepsilon_i}_i$ 
using the convention that  $G_i^-:=(G_i \hbox{\small\, with an half-twist on a edge})$ and $G_i^+:=G_i$.
That $\psi_k$ is well-defined follows from the operations $(\mathcal{O}_2),(\mathcal{O}_4),(\mathcal{O}_5)$. 

Let us show that  $\psi_k$ is surjective.
Any $M\in \calV_0$ is $Y_k$-equivalent to $V$ and, so, by Proposition \ref{prop:Y_clasper},
there is a sequence $V=M_0 \leadsto M_1  \leadsto \cdots \leadsto M_r =M$ where each move 
$M_i \leadsto M_{i+1}$ is either a surgery along a connected graph clasper of degree~$k$, or the inverse of such a surgery;
furthermore, thanks to  $(\mathcal{O}_0)$, we can assume that each graph clasper involved in the sequence is tree-shaped. 
Now, any surgery $W \leadsto W_T$  along a tree-shaped graph clasper $T$ in a $3$-manifold $W$ has the following properties:
\begin{itemize}[label=$\diamond$]
\item it is reversible, in the sense that  there is a  graph clasper $I$ (of the same shape as $T$) in $W_T$, such that $(W_T)_I \cong W$;
\item there is a $t\in \calM(\partial \hbox{N}(T))$ such that $W_T \cong \left(W \setminus \interior  \hbox{N}(T)\right) \cup_t \hbox{N}(T)$,
hence any graph  clasper in $W_T$ can be isotoped into the subset $W \setminus \interior  \hbox{N}(T)$ of~$W_T$.
\end{itemize}
It follows that there exists a disjoint union $G=\sqcup_i G_i$ of (tree-shaped) connected graph claspers of degree $k$ in $V$ such that $V_G \cong M$.
We deduce that $\psi_k(\sum_i G_i) = M$.

Thus, we would like to understand the equivalence relation $\sim$
on $ \Z\!\cdot\! C_k$ such that the map  $\psi_k$ factorizes to a bijection on the quotient set:
$$
\xymatrix{
 \Z\!\cdot\! C_k  \ar@{->>} [d] \ar@{->>}[rr]^-{\psi_k} && {\displaystyle \frac{\calV_0}{Y_{k+1}} }\\
{\displaystyle \frac{ \Z\!\cdot\! C_k}{\sim}  }  \ar@{-->>}[rru]_-{\overline{\psi_k}}^-\simeq   & &
}
$$
For instance, it follows from $(\mathcal{O}_2)$ that we must have $G' \sim G^\pm$ for any graph claspers
$G$ and $G'$ in $M$ which have the same shape and the same leaves.
Besides, there are other instances of the relation $\sim$ that deal with  leaves 
and result from $(\mathcal{O}_0), (\mathcal{O}_1)$ and $(\mathcal{O}_3)$. 
Finally, using other operations  on graph claspers (not in the above list), we obtain other instances of the relation $\sim$
that do not affect the leaves but change the shape: one such example is the so-called ``\emph{IHX relation}''. 
Once we have a candidate for the relation $\sim$, 
the difficulty is then to show the injectivity of the resulting map $\overline{\psi_k}: {\Z\!\cdot\! C_k}/{\sim} \to {\calV_0}/{Y_{k+1}}$.
This is proved by finding sufficiently enough topological invariants on $\calV(R)$  --- or, at least, on  its subset $\calV_0$ --- 
that are unchanged by $Y_{k+1}$-surgery and constitute a left-inverse $Z_k$ to $\overline{\psi_k}$
when they are conveniently assembled all together:
$$
\xymatrix{
{\displaystyle \frac{ \Z\!\cdot\! C_k}{\sim}  }  \ar@{->>}[r]^-{\overline{\psi_k}}  \ar@/_0.8cm/[rr]^-{\id} &  {\displaystyle \frac{\calV_0}{Y_{k+1}} } \ar[r]^-{Z_k}
& {\displaystyle \frac{ \Z\!\cdot\! C_k}{\sim}  } 
}
$$  

At the end of this process, we  conclude that $\overline{\psi_k}$ is injective and, so, bijective, thus obtaining a combinatorial description
of the quotient set ${\calV_0}/{Y_{k+1}}$, and concluding that the invariant $Z_k$ classifies the $Y_{k+1}$-equivalence relation on $\calV_0$.

\begin{remark}  
In all the few situations that the author knows,  the relation $\sim$ on  $\Z\!\cdot\! C_k$ happens to be always defined by a subgroup of $\Z\!\cdot\! C_k$:
hence ${\calV_0}/{Y_{k+1}}$ has a structure of abelian group, although ${\calV_0}$ may not have (a priori)  a natural operation. 
If  ${\calV_0}$  does have a natural operation compatible with the ${Y_{k+1}}$-equivalence and if we know  that ${\calV_0}/{Y_{k+1}}$
 inherits a structure of abelian group, it is often much easier to carry on the above process with rational coefficents
 in order to get a combinatorial description of the vector space $\big({\calV_0}/{Y_{k+1}}\big) \otimes \Q$.  
\hfill $\blacksquare$
\end{remark}  

The above ``general strategy'', to study inductively  the $Y_\ell$-equivalence relations by clasper calculus, will be mentioned in the next sections in a few examples.

\subsection{Other kinds of surgeries}

To conclude this section, we mention yet other  surgery equivalence relations.
Some of them are just alternative descriptions of the relations that have been introduced in the previous subsections,
but  other ones are quite different. We fix a closed surface $R$, which may be disconnected.\\

\begin{enumerate}
\item \textbf{LP surgeries.}  
A \emph{homology handlebody} of \emph{genus} $g$ is a compact $3$-manifold $C'$ with the same homology  type as  $H_g$;
 the \emph{Lagrangian} of $C'$ is  the kernel $L_{C'}$ of the homomorphism 
$H_1(\partial C';\Z) \to H_1(C';\Z)$ induced by the inclusion $\partial C'\hookrightarrow C'$:
this is a Lagrangian subgroup of $H_1(\partial C';\Z)$ with respect to the intersection form.
Following Auclair and Lescop \cite{AL}, we call  \emph{LP-pair}  a couple $\mathsf{C}=(C',C'')$  of
two homology handlebodies whose boundaries are identified $\partial C' = \partial C''$  in such a way that $L_{C'}=L_{C''}$.
(The acronym ``LP'' is for ``Lagrangian-Preserving''.)
Given an $M \in \calV(R)$ and an LP pair $\mathsf{C}=(C',C'')$ such that $C'\subset M$,
one can replace in $M$ the submanifold $C'$ by $C''$ to obtain a new $3$-manifold
$$
M_{\mathsf{C}} := \left( M\setminus \hbox{int} (C') \right) \cup_\partial C''.
$$
The move $M\leadsto M_{\mathsf{C}}$ in $\calV(R)$ is called an \emph{LP-surgery}. 

A Torelli twist $M \leadsto M_s$ can be interpreted as an LP-surgery since a regular neighborhood
of the surface $S\subset M$  is a handlebody. Conversely, an LP-surgery can be realized by finitely many $Y$-surgeries
because, for any LP pair $\mathsf{C}$, the  homology handlebodies $C'$ and $C''$ are Torelli--equivalent.  
(See Remark \ref{rem:HH} below.) 
Therefore, LP-surgery equivalence is the same as Torelli--equivalence.

There is also a rational version of the LP-surgery using $H_1(-;\Q)$ instead  of $H_1(-;\Z)$, which has been considered by Moussard \cite{Mous12}.
However, rational  LP-surgery equivalence is coarser than Torelli--equivalence as a relation.\\

\item \textbf{Torelli surgeries.} Let $M\in \calV(R)$, let  $C \subset M$ be  a handlebody and  let $c\in \calI(\partial C)$.
 Following  Kuperberg and Thurston \cite{KT}, we say that
 $$
 M_c := \big( M \setminus \interior(C) \big) \cup_c C
 $$
 is obtained from $M$ by a \emph{Torelli surgery} along $C$. Clearly, a Torelli surgery can be realized by a Torelli twist
 (by choosing a small open disk $D\subset \partial C$ 
 and isotoping $c$ so that it fixes $D$ pointwisely); conversely,
 a Torelli twist can be realized by a Torelli surgery (because a regular neighborhood
of a  surface with non-empty boundary is a handlebody). 
Thus, the $Y_k$-equivalence and $J_k$-equivalence relations can be reformulated in terms of Torelli surgeries.\\

\item \textbf{Lagrangian Torelli surgeries.}  Let $C$ be a handlebody. The \emph{Lagrangian Torelli group} of 
$S:= \partial C \setminus  \hbox{(small open disk)}$ is defined by
$$
\qquad \qquad \calI^L(S) := \Big\{ f \in \calM(S) : f_*(L_C) \subset L_C \hbox{ and $f_*$ is the $\id$ on $\frac{H_1(S;\Z)}{L_C}$}  \Big\}
$$
where $L_C$ is the Lagrangian  of $C$. A \emph{Lagrangian Torelli surgery} is defined in a way similar to a Torelli surgery
 using the Lagrangian Torelli group instead of the Torelli group. Clearly, a Lagrangian Torelli surgery is a special case of an LP surgery:
 therefore, the equivalence relation defined by Lagrangian Torelli surgeries is again the Torelli--equivalence.
 
 Nevertheless, following Faes \cite[\S A]{Faes_thesis}, 
 we can define a new family of equivalence relations on $\calV(R)$ by considering the following  filtration 
 on the Lagrangian Torelli group of a handlebody $C$. Let $\mathbb{L}_C$ denote the kernel
 of the homomorphism $p:\pi_1(S) \to \pi_1(C)$ induced by the inclusion $S \hookrightarrow C$ and consider, for any integer $k\geq 1$, the subset
 $$
 L_k \calI^L(S)  :=  \big\{ f \in \calI^L(S) : p f_*(\mathbb{L}_C) \subset \Gamma_{k+1} \pi_1(C) \big\}
 $$
 of the mapping class group of $S$.
 According to Levine \cite{Lev01,Lev06}, the filtration 
 $$
  \calI^L(S)  = L_1 \calI^L(S)  \supset   L_2 \calI^L(S) \supset   L_3 \calI^L(S)  \supset \cdots
 $$
 is a  decreasing sequence of subgroups of the Lagrangian Torelli group,  which contains the Johnson filtration of the Torelli group $\calI(S)$.
 But, in contrast with the latter, the  intersection of the former is not trivial: 
its intersection is the subgroup  of $\calI^L(S)$ consisting of all diffeomorphisms that extend to the full handlebody $C$;
hence this intersection is irrelevant for Lagrangian Torelli surgeries. 

Thus, it is interesting to consider the following relation for any $k\in \N^*$:
we say that $M,M' \in \calV(R)$ are  \emph{$L_k$-equivalent}  if $M'$ can be obtained from $M$ by a Lagrangian Torelli surgery $M\leadsto M_{c}$ 
along a handlebody $C\subset \hbox{int}(M)$ with a $c\in L_{k} \calI(S)$. 
Clearly, we have ``$J_k \Rightarrow L_k$'' for any $k\geq 1$.
We have already mentioned  the equality of relations $L_1=J_1$,
and it follows essentially from Levine's results that $L_2=J_2$. 
However, the $L_3$-equivalence is strictly weaker than the $J_3$-equivalence as a relation \cite[\S A]{Faes_thesis}.
\end{enumerate}

\section{Surgery equivalence relations: their characterization} \label{sec:charac}

In this section, we review several results from the middle 1970's to nowadays,
which provide characterizations of  the $J_k$-equivalence,
the $Y_k$-equivalence, and
 the $k$-surgery equivalence
relations  in terms of topological invariants (for some or all values of $k\in \N^*$).

\subsection{Two case studies to consider} \label{subsec:two_cases}

Let $R$ be a compact surface. The problem of characterizing  surgery equivalence relations in $\calV(R)$ 
is very much dependent on the choice of $R$. 
So we shall restrict ourselves to the two cases that we have already mentioned on page~\pageref{page:R}:
\begin{center}
(i) $R=\varnothing$; \quad (ii) $R= \partial(\Sigma \times [-1,1])$ where $\Sigma$ is a compact surface with $\partial \Sigma \cong S^1$.
\end{center}

Actually, in case (i), our interest in the set of closed $3$-manifolds $\calV(\varnothing)$ will quickly specialize to the class 
$$
\mathcal{S} := \big\{ M \in  \calV(\varnothing) : H_*(M;\Z)\simeq H_*(S^3;\Z) \big\}
$$ 
of \emph{homology $3$-spheres}.
Of course, this is a strong restriction but, as we shall see, $\mathcal{S}$ is still a very rich set-up for studying surgery equivalence relations.

\begin{remark}
The set $\mathcal{S}$ with the connected sum operation $\sharp$ is a monoid, whose neutral element is $S^3$. \hfill $\blacksquare$
\end{remark}

Similarly, in case (ii), 
our interest in the set  $\calV\big(\partial(\Sigma \times [-1,1])\big)$ of  cobordisms  will be restricted to the subset $\mathcal{IC}(\Sigma)$
of \emph{homology cylinders} over $\Sigma$.
Those are cobordisms $(C,c)$ from $\Sigma$ to $\Sigma$ such that the boundary parametrizations
$$
c_+ := c\vert_{\Sigma \times \{+1\}}: \Sigma \longrightarrow C \quad \hbox{and} \quad c_- := c\vert_{\Sigma \times \{-1\}}: \Sigma  \longrightarrow C
$$
 induce isomorphisms in homology and satisfy $c_{+,*}=c_{-,*}: H_1(\Sigma;\Z) \to H_1(C;\Z)$:
$$
\labellist
\scriptsize\hair 2pt
 \pinlabel {$C$}  at 92 128
 \pinlabel {$c_+$} [r] at 91 222
 \pinlabel {$c_-$} [r] at 90 38
 \pinlabel {$\Sigma$} [r] at 2 255
 \pinlabel {$\Sigma$} [r] at 3 4
\endlabellist
\centering
\includegraphics[scale=0.28]{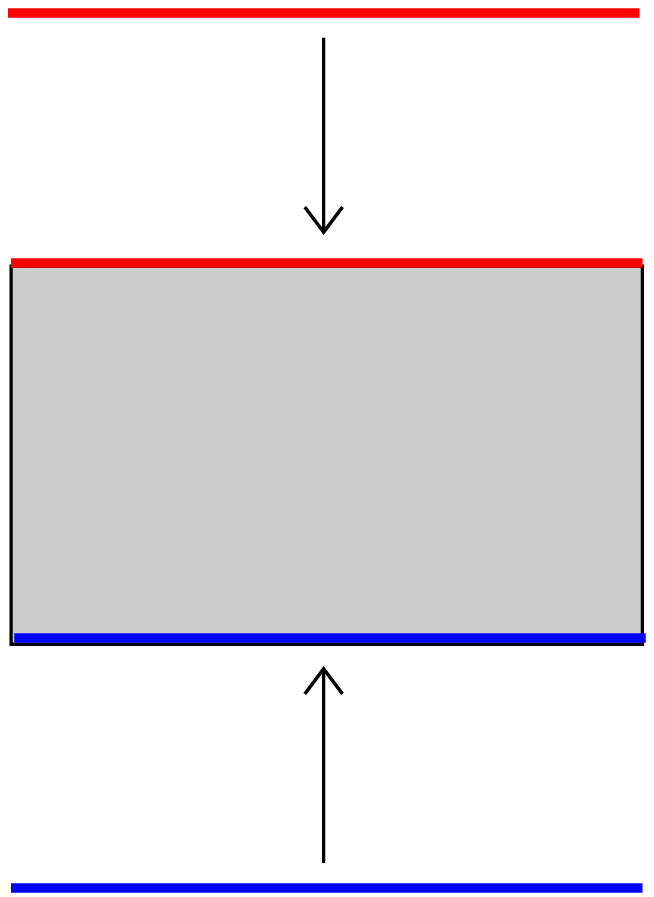}
$$
Two  cobordisms $(C,c)$ and $(D,d)$ from $\Sigma$ to $\Sigma$
can be \emph{multiplied} by gluing $D$ \lq\lq on the top of\rq\rq{} $C$, 
using the boundary parametrizations $d_-$ and $c_+$ to identify $d_-(\Sigma)$ with $c_+(\Sigma)$:
$$
\centre{\labellist
\scriptsize\hair 2pt
 \pinlabel {$C$}  at 92 128
 \pinlabel {$c_+$} [r] at 91 222
 \pinlabel {$c_-$} [r] at 90 38
 \pinlabel {$\Sigma$} [r] at 2 255
 \pinlabel {$\Sigma$} [r] at 3 4
\endlabellist
\centering
\includegraphics[scale=0.28]{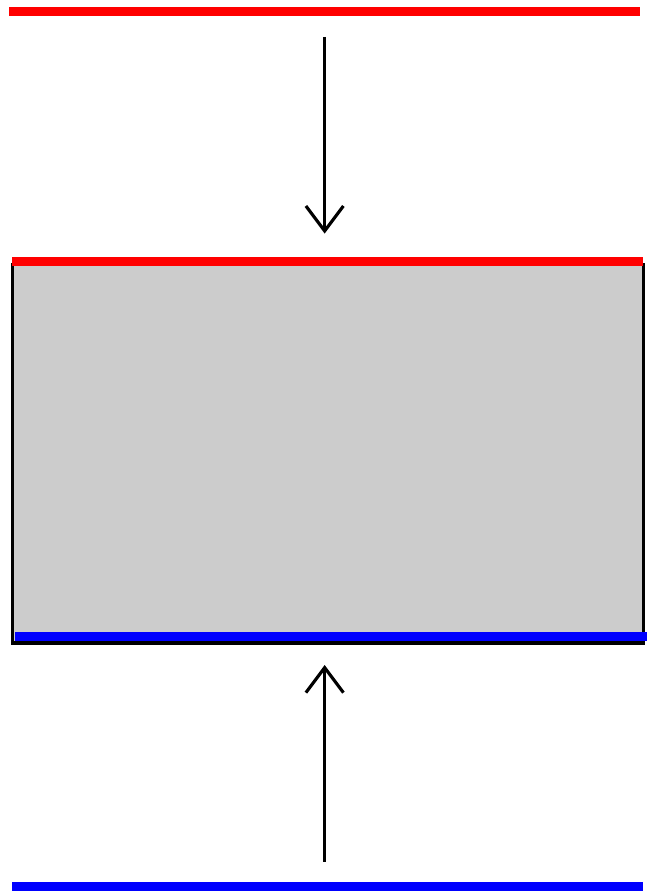}} \quad \circ \quad  
\centre{\labellist
\scriptsize\hair 2pt
 \pinlabel {$D$}  at 92 128
 \pinlabel {$d_+$} [r] at 91 222
 \pinlabel {$d_-$} [r] at 90 38
 \pinlabel {$\Sigma$} [r] at 2 255
 \pinlabel {$\Sigma$} [r] at 3 4
\endlabellist
\includegraphics[scale=0.28]{cobordism}} 
\quad := \quad 
\centre{\centering
\labellist
\scriptsize\hair 2pt
 \pinlabel {$\Sigma$} [r]  at 3 362
 \pinlabel {$\Sigma$} [r] at 2 4
 \pinlabel {$D$}  at 93 232
 \pinlabel {$C$}  at 92 126
 \pinlabel {$c_{-}$} [r] at 92 37
 \pinlabel {$d_+$} [r] at 91 331
\endlabellist
\centering
\includegraphics[scale=0.28]{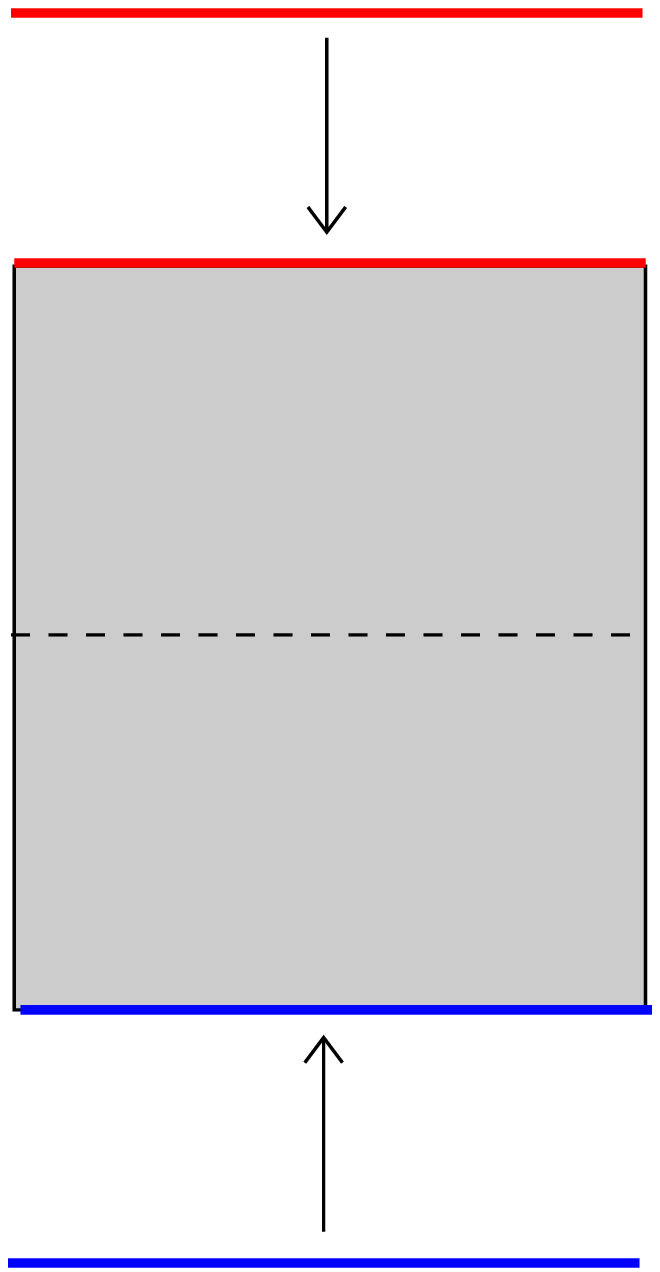}}
$$
It is easily checked that $C \circ D \in \mathcal{IC}(\Sigma)$ if $C,D \in \mathcal{IC}(\Sigma)$.
Hence  the set $\mathcal{IC}(\Sigma)$  with this operation $\circ$ is a monoid,
whose neutral element is  the trivial cylinder $U:= \Sigma \times [-1,1]$ (with the obvious boundary parametrization).

\begin{proposition}
The ``mapping cylinder'' construction defines a monoid homomorphism 
$\operatorname{cyl}: \calI(\Sigma) \to \mathcal{IC}(\Sigma)$, 
which is injective and surjective onto the group  of invertible elements of $\mathcal{IC}(\Sigma)$.
\end{proposition}

\begin{proof}[About the proof]
  A diffeomorphism $f\colon \Sigma \to \Sigma$ defines a cobordism $\operatorname{cyl}(f)$ 
from $\Sigma$ to $\Sigma$
whose underlying $3$-manifold is the trivial cylinder $U$
and whose boundary parametrization  ${\partial (\Sigma \times [-1,1])} \to {\partial U }$ is
given by $f$ on the top surface $\Sigma \times \{+1\}$ and by the identity elsewhere. 
Clearly, the diffeomorphism class of $\operatorname{cyl}(f)$ only 
depends on the isotopy class of $f$ and, obviously, $\operatorname{cyl}(f)$ is a homology cylinder
if $f$ induces the identity in homology. 

Thus we obtain a map $\operatorname{cyl}: \calI(\Sigma) \to \mathcal{IC}(\Sigma)$.
Clearly it is multiplicative,
and it is injective for the following reason: two diffeomorphisms $\Sigma \to \Sigma$ are isotopic rel $\partial \Sigma$
if and only if they are homotopic rel $\partial \Sigma$, by the classical result of Baer  \cite{Baer} 
that we have already alluded to at page \pageref{page:Baer}.
The image of  $\operatorname{cyl}$ is determined  in \cite[Prop$.$ 2.4]{HM12}, for instance. 
\end{proof}

The following is easily checked.

\begin{proposition}
 The map 
$\iota:\mathcal{S} \to \mathcal{IC}(\Sigma)$  defined by $\iota( M):=  M \sharp U$ 
is an injection of monoids, and it is an isomorphism for $\Sigma=D^2$.
\end{proposition}

Thus, the monoid of homology cylinders $ \mathcal{IC}(\Sigma)$ can be viewed as a simultaneous generalization
of the Torelli group $ \mathcal{I}(\Sigma)$ and the monoid $\mathcal{S}$.

\subsection{Characterization of the Torelli--equivalence}

The most fundamental result is the characterization of the Torelli--equivalence, 
which has been obtained for closed $3$-manifolds  by Matveev \cite{Matveev}. To state his result, we recall that 
the \emph{linking number} 
\begin{equation} 
\operatorname{Lk}(K,L) \in \Q
\end{equation}
of two disjoint oriented knots $K,L$ in a closed $3$-manifold $M$ is defined  when $K$ and $L$ are rationally null-homologous:
let $n\in \N^*$  be such that $n[K]= 0 \in H_1(M;\Z)$ and let $\Sigma \subset M$ be a surface transverse to $L$ 
such that $\partial \Sigma$ consists of $n$ parallel copies of the knot $K$; then 
\begin{equation}  \label{eq:Lk}
\operatorname{Lk}(K,L) := \frac{1}{n} \Sigma \bullet L
\end{equation}
where $\Sigma \bullet L \in \Z$ denotes the algebraic intersection number.
It can be verified that  the class of $\operatorname{Lk}(K,L)$ modulo $1$ only depends on the integral homology classes of $K$ and $L$. 
Hence we get a map
$$
\lambda_M: \hbox{Tors}\, H_1(M;\Z) \times  \hbox{Tors}\, H_1(M;\Z) \longrightarrow \Q/\Z, \  ([K], [L]) \longmapsto \big(\operatorname{Lk}(K,L)\!\!\!\!  \mod 1\big)
$$
which is called the \emph{(torsion) linking pairing} of $M$ and is one of the eldest
invariants of closed $3$-manifolds \cite[\S 77]{ST}.  
The map $\lambda_M$ is bilinear, symmetric and non-singular (see \cite[Lemma 6.7]{Mas11}, for instance).

\begin{theorem}[Matveev 1987] \label{th:Matveev}
Two manifolds $M,M' \in \calV(\varnothing)$ are Torelli--equivalent if, and only if, there is an isomorphism 
$\psi: H_1(M;\Z) \to H_1(M';\Z)$ such that the following diagram is commutative:
\begin{equation} \label{eq:triangle}
\xymatrix{
\operatorname{Tors}\, H_1(M;\Z) \times  \operatorname{Tors}\, H_1(M;\Z)  \ar[rr]^-{\lambda_M}  \ar[d]_-{\psi \times \psi}^-\simeq&& \Q/\Z \\
\operatorname{Tors}\, H_1(M';\Z) \times  \operatorname{Tors}\, H_1(M';\Z)  \ar[rru]_-{\lambda_{M'} } && }
\end{equation}
\end{theorem}

\begin{proof}[Sketch of proof]
Assume that $M$ and $M'$ are Torelli--equivalent. 
Hence there is a Torelli twist $M \leadsto M_s$ along a surface $S \subset M$ such that $M_s\cong M'$.
This surgery induces an isomorphism $\psi:=\psi_s$ in homology, as described by \eqref{eq:psi_s}. 
Using the notations of \eqref{eq:Lk} and setting $x:=[K]$ and $y:=[L]$,  we have
$$
\lambda_M(x,y ) =  \frac{1}{n} \Sigma \bullet L \!\!\!\mod 1.
$$
Since the handlebody $\hbox{N}(S) = S \times [-1,1]$ deformation retracts onto a $1$-dimensional subcomplex,
we can isotope $K$ and $L$ in $M$ to make them disjoint from $\hbox{N}(S)$: hence, as subsets of $M \setminus \interior \hbox{N}(S)$,
$K$ and $L$ can also be regarded as knots in $M_s=M'$; so we have $\psi(x) =[K]$ and $\psi(y) =[L]$ in $H_1(M';\Z)$.
Furthermore, we can isotope $\Sigma$ so that it cuts  the handlebody $\hbox{N}(S)$ tranversely along meridional disks  of $\hbox{N}(S)$:
in particular, the boundary $\partial \Sigma^\circ \subset \partial \hbox{N}(S)$ of  $\Sigma^\circ:= \Sigma \cap \big( M \setminus \interior \hbox{N}(S)\big)$ is null-homologous in~$\hbox{N}(S)$.
Recall that $M'$ is obtained from $M \setminus  \interior \hbox{N}(S)$ by re-gluing $\hbox{N}(S)$ using a diffeomorphism
$\tilde s:\partial \hbox{N}(S) \to \partial \hbox{N}(S)$ that acts trivially in homology: 
hence $\partial \Sigma^\circ$ is still null-homologous  in the re-glued handlebody of $M'$,
so that $\Sigma^\circ \subset  M \setminus  \interior \hbox{N}(S)$ can be completed inside the re-glued handlebody 
to get a  surface $\Sigma' \subset M'$ satisfying $\partial \Sigma' = n K$. We conclude that
$$
\lambda_{M'}(\psi(x),\psi(y))= \Big(\frac{1}{n} \Sigma' \bullet L \!\!\!\mod 1\Big) =  \Big(\frac{1}{n} \Sigma \bullet L \!\!\!\mod 1\Big) = \lambda_M(x,y).
$$

Assume now that there is an isomorphism $\psi$ in homology satisfying \eqref{eq:triangle}. 
According to Theorem \ref{th:RTWL}, $M$ has a \emph{surgery presentation} in $S^3$:
i$.$e$.$, there is an $n$-component framed link $L\subset S^3$ such that $M=S^3_L$.
We now recall the way of computing $\lambda_M$ from the \emph{linking matrix} of $L$, 
which is the $n\times n$ matrix
$$
\operatorname{Lk}(L) := \big(\operatorname{Lk}(L_i,L_j)\big)_{i,j}.
$$
(Here we have choosen an orientation for each component $L_i$ of $L$, and the linking number $\operatorname{Lk}(L_i,L_j)$
is an integer because $H_1(S^3;\Z)$ is trivial; by convention, $\operatorname{Lk}(L_i,L_i):= \operatorname{Lk}(L_i,\rho(L_i))$
is the linking number of $L_i$ and its parallel $\rho(L_i)$.) 

Let $H:= \Z^n$, let $f:H \times H \to \Z$ be the symmetric bilinear map
whose matrix in the canonical basis $(e_i)_i$ is $\operatorname{Lk}(L)$, and let $\widehat{f}: H \to \Hom(H,\Z)$ be the adjoint of~$f$.
We consider the symmetric bilinear form
$$
\lambda_f: G_f \times G_f \longrightarrow \Q/\Z
$$
defined on the finite abelian group $G_f := \hbox{Tors} \big(\operatorname{Coker} \widehat{f}\,\big)$ by
$$
\forall \{u\}, \{v\} \in G_f \subset \frac{\Hom(H,\Z)}{\widehat{f}(H)}, \quad
\lambda_f\left(\{u\}, \{v\}\right) := \big(  f_\Q \left(\widehat{u},\widehat{v} \right)\!\!\!  \mod 1 \big)
$$
where $f_\Q$ is the extension of $f$ to rational coefficients
and where $\widehat{u},\widehat{v}$ are antecedents of $u_\Q,v_\Q: H\otimes \Q \to \Q$ by the adjoint
$\widehat{f}_\Q: H\otimes \Q \to \Hom(H\otimes \Q,\Q)$. It is easily verified that $\lambda_f$ is non-singular. 

This algebraic construction from the matrix $\hbox{Lk}(L)$
has the following topological interpretation in terms of the $4$-manifold $W_L$ obtained from $D^4$ by attaching $2$-handles along $L$:
\begin{itemize}[label=$\diamond$]
\item  $H\simeq H_2(W_L;\Z)$  and $-f$ then corresponds to  the  intersection form of $W_L$;
\item hence $ \hbox{Coker} \widehat f \simeq H_1(M;\Z)$ and  $-\lambda_f$ then corresponds to $\lambda_M$.
\end{itemize}
We proceed similarly with $M'$ to get a  symmetric bilinear form $f'$ on a finitely-generated free abelian group  $H'$. 
By assumption, we have $(G_f,\lambda_f) \simeq (G_{f'},\lambda_{f'})$ 
and it follows from  early works  in knot theory \cite{KP,Kyle} and algebra \cite{Wall72,Durfee}
that the pairs $(H,f)$ and $(H',f')$ are \emph{stably equivalent},
meaning that there exist integers $n_\pm,n'_\pm\geq 0$ such that 
$$
(H,f) \oplus (\Z,+1)^{\oplus n_+}   \oplus (\Z,-1)^{\oplus n_-} \simeq (H',f') \oplus (\Z,+1)^{\oplus n'_+}   \oplus (\Z,-1)^{\oplus n'_-}.
$$
The direct sum with $(\Z,\pm 1)$ can be realized, at the level of surgery presentations, by the disjoint union with the $(\pm 1)$-framed unknot,
and this does not change the $3$-manifold after surgery.  Besides, an automorphism of $H$ 
can be decomposed into finitely many ``elementary'' automorphisms which, in terms of the basis $(e_i)_i$ of $H$, 
are given by the operations $e_i \leftrightarrow e_j$,
$e_i \mapsto -e_i$ or $(e_i,e_j) \mapsto (e_i+e_j,e_j)$;  these ``elementary'' automorphisms  can be realized, at the level of surgery presentations, by
the renumbering $i\leftrightarrow j$ of the components of $L$, the change of orientation  $L_i \mapsto -L_i$ or the operation $(L_i, L_j)\mapsto (L_i\sharp L_j,L_j)$,
respectively. All these elementary operations on links (which constitute the so-called ``Kirby calculus'' \cite{Kirby})
do not affect the $3$-manifold after surgery : it is obvious for the first two operations and, for the third operation,
 it is justified by sliding the attaching locus of a $2$-handle of $W_L$ along another $2$-handle.  
 
Therefore,  we can assume without restriction of generality  
that $M$ and $M'$ are  presented by surgery in $S^3$ along framed links with the same linking matrix:
$$\operatorname{Lk}(L)= \operatorname{Lk}(L').$$
 Then a result of Murakami \& Nakanishi \cite{MN} asserts that $L$ and $L'$ are related one to the other, 
 by isotopies and finitely many local moves of the following type:
\begin{equation}
\begin{array}{c} {\labellist \small \hair 0pt 
\pinlabel {$\longleftrightarrow$} at 360 140
\endlabellist
\includegraphics[scale=0.25]{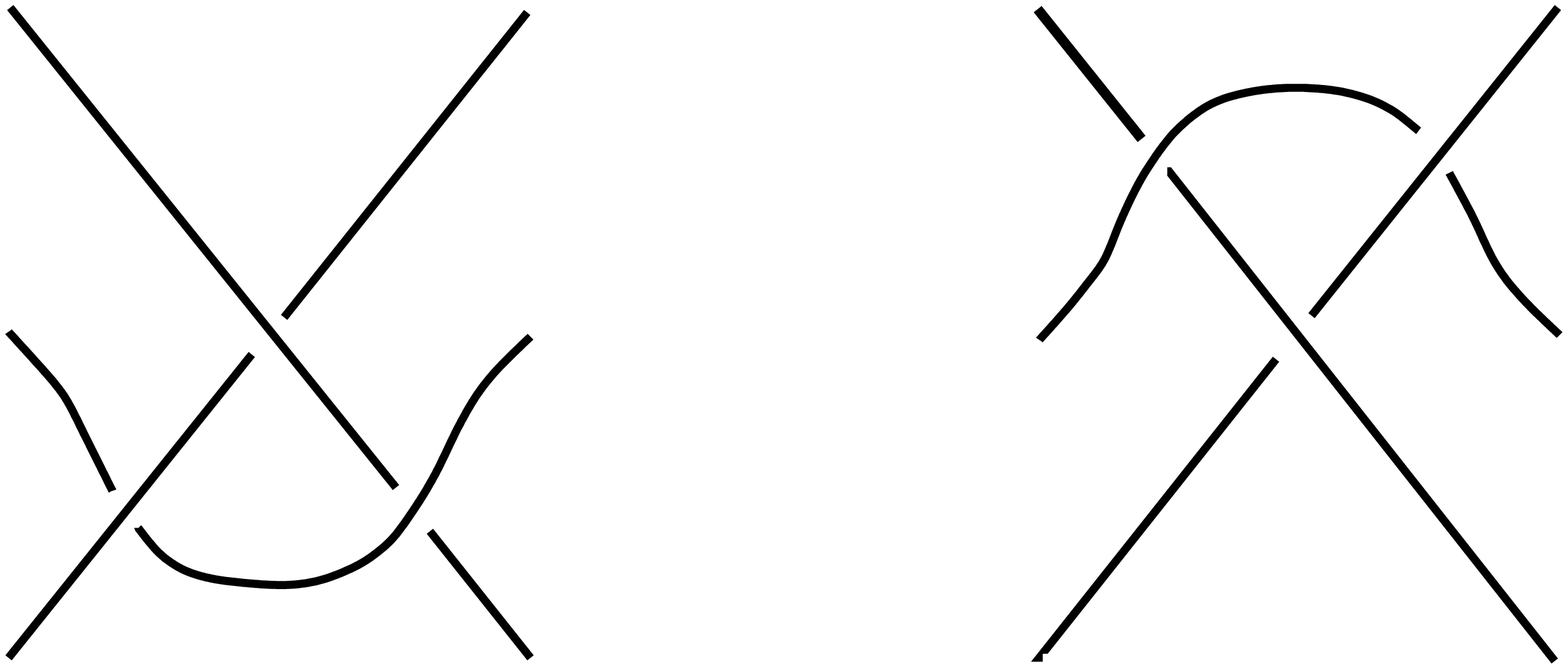}} \end{array}
\end{equation}
Such a local move (called a \emph{$\Delta$-move} in \cite{MN}) can be realized by surgery along a $Y$-graph: see \cite[Fig. 34 (b)]{Hab00}, for instance.
 We conclude that, up to diffeomorphisms, $M$ and $M'$ are related one to the other by finitely many $Y$-surgeries. Hence they are Torelli--equivalent.
\end{proof}

\begin{remark} \label{rem:YY}
  The proof of Theorem \ref{th:Matveev} given in \cite{Matveev} is not  detailed, 
and the knot-theoretical ingredient  in terms of linking matrices  \cite{MN}  is actually posterior to  \cite{Matveev}.
By refining this proof, \cite{Mas03-} and \cite{DM} extend Theorem~\ref{th:Matveev}  
 to the setting of $3$-manifolds  with spin and complex spin structures, respectively:
these extensions  involve quadratic forms which refine the linking pairing and depend on the (complex) spin structures.
See also~\cite{Mous15} for a detailed proof of  Matveev's theorem and additional contents.    \hfill $\blacksquare$
\end{remark}

As an immediate consequence of Theorem \ref{th:Matveev}, we obtain the following result about $\mathcal{S}$
which dates back to \cite{Bir74} and is proved there with Heegaard splittings.
The formulation in terms of blinks (see Remark~\ref{rem:blink}) appears in  \cite{Hilden}.
 
\begin{cor}[Birman 1974] \label{cor:S}
Any homology $3$-sphere is Torelli--equivalent to $S^3$.
\end{cor}

By refining the proof of Theorem \ref{th:Matveev},
we can also prove the following refinement of  Corollary \ref{cor:S} which generalizes \cite{MN}.

\begin{cor} \label{cor:SL}
Let $M,M' \in \mathcal{S}$ 
and let $L\subset M, L'\subset M'$ be framed oriented $n$-component links.
The pairs $(M,L)$ and $(M',L')$ are Torelli--equivalent if, and only~if, we have $\operatorname{Lk}(L) = \operatorname{Lk}(L')$.
\end{cor}

Let $\Sigma$ be a compact surface with $\partial \Sigma \cong S^1$. We now turn to homology cylinders over $\Sigma$
(whose definition has been given in \S \ref{subsec:two_cases}).
The following, which appears in \cite{Hab00}, states that  $\mathcal{IC}(\Sigma)$ constitutes a Torelli--equivalence class.

\begin{proposition}[Habiro 2000] \label{prop:Habiro}
Any homology cylinder  over $\Sigma$ is Torelli--equivalent to the trivial cylinder $U=\Sigma \times [-1,1]$.
\end{proposition}

\begin{proof}[Sketch of the proof]
  Fix a system of \emph{meridians and parallels} in the surface $\Sigma$, i.e$.$ 
a system of  simple oriented closed curves $(\alpha_1,\dots,\alpha_g,\beta_1,\dots,\beta_g)$ 
having the following intersection pattern:
$$
\labellist \small \hair 0pt 
\pinlabel {\textcolor{blue}{$\alpha_1$}} [l] at 310 105
\pinlabel {\textcolor{blue}{$\alpha_g$}} [l] at 568 65
\pinlabel {\textcolor{red}{$\beta_1$}} [b] at 216 80
\pinlabel {\textcolor{red}{$\beta_g$}} [b] at 472 69
\pinlabel {$\circlearrowleft$} at 520 30
\endlabellist
\includegraphics[scale=0.4]{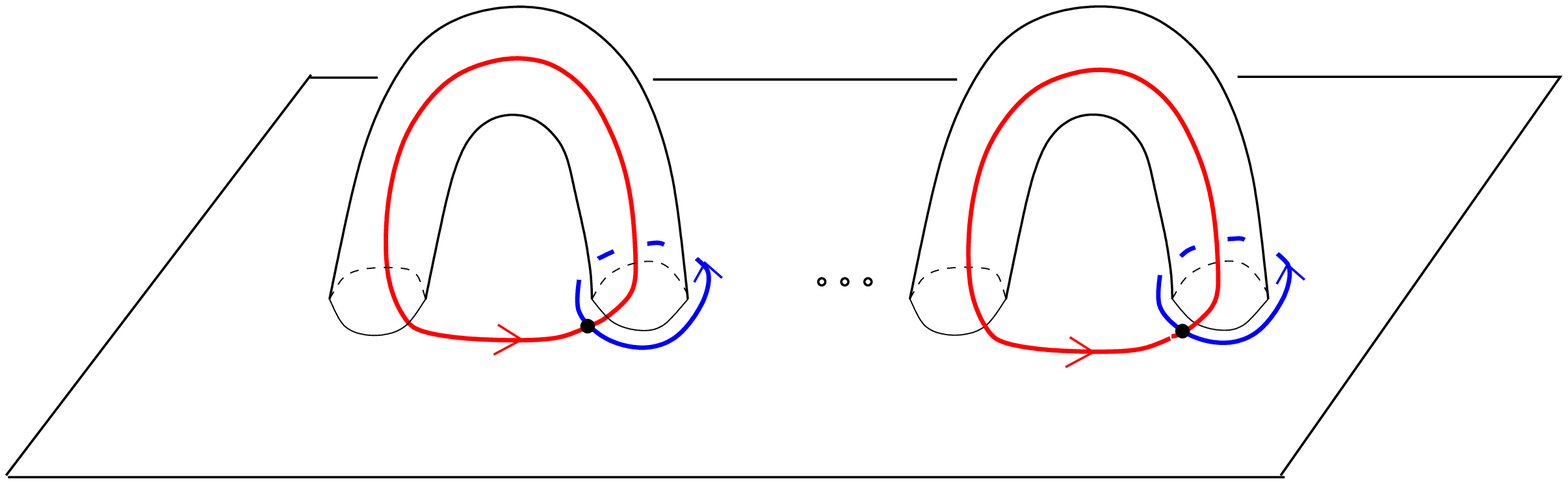}
$$
Let $(C,c)\in \mathcal{IC}(\Sigma)$: recall that $C$ is viewed as a cobordism from the ``top'' surface $\partial_+ C := c_+(\Sigma)$
to the ``bottom'' surface $\partial_- C := c_-(\Sigma)$.
By gluing one $2$-handle along each curve $c_-(\alpha_i)$ on  $\partial_-C$
and one $2$-handle along each curve $c_+(\beta_j)$ on  $\partial_+C$,
the homology cylinder $C$ is turned into a homology $3$-ball~$C'$. 
Next, by adding a $3$-handle to $C'$, we get a homology $3$-sphere $\wideparen{C'}$.
Each $2$-handle $D^2 \times D^1$ has a \emph{co-core}, which is the image of $\{0\} \times D^1$ after attachment of the $2$-handle:
hence the above procedure has also produced a framed oriented $(2g)$-component tangle 
$(\gamma_1^+,\dots, \gamma_g^+,\gamma_1^-,\dots, \gamma_g^-)$ in $C'$, which is called a \emph{bottom-top tangle}.
Now, we can connect the two extremities of each component $\gamma_j^+$ (resp$.$~$\gamma_i^-$)  
by a small arc on the ``top''  (resp$.$ ``bottom'') boundary of  $C'$ 
to get an oriented framed knot $G_j^+$ (resp$.$~$G_i^-$) in $\wideparen{C'}$.
It can be deduced from the equality $c_{+,*}=c_{-,*}: H_1(\Sigma;\Z) \to H_1(C;\Z)$
that the linking matrix of the framed oriented link
$
G:=(G_1^+,\dots,G_g^+, G_1^-,\dots,G_g^-)
$
is 
$$
\operatorname{Lk}(G)= \begin{pmatrix}
0_g & I_g \\ I_g & 0_g
\end{pmatrix}
$$
so that, in particular, it does not depend on $C\in \mathcal{IC}(\Sigma)$.

If we apply the above constructions to the trivial cylinder $U$ instead of $C$, we obtain  $U'\cong D^3$ and, inside  $\wideparen{U'}\cong S^3$,
we obtain a link $T$ with $\operatorname{Lk}(T)=\operatorname{Lk}(G)$.
It follows then from Corollary \ref{cor:SL} that the pair $(\wideparen{C'},G)$ is Torelli--equivalent to  $(\wideparen{U'},T)$
and, therefore, $C$ is Torelli--equivalent to $U$.
We refer to \cite[Cor. 7.7]{CHM} for a more general result and more detailed arguments.  
\end{proof}

\begin{remark} \label{rem:HH}
  Recall that $H_k$ is the standard handlebody of genus $k$, with boundary~$\Sigma_k$.
A manifold $C\in \calV(\Sigma_k)$ is a \emph{homology handlebody} of \emph{genus $k$} if it has the same homology type as $H_k$.
Using the same method of proof as for Proposition~\ref{prop:Habiro}, 
we can show the following characterization due to Habegger \cite{Habegger}:
 two homology handlebodies $C',C''$ of genus $k$ are Torelli--equivalent  if, and only if, 
they have the same Lagrangians:
$$
\ker\big(c'_*: H_1(\Sigma_k;\Z) \longrightarrow  H_1( C';\Z)  \big)  = \ker\big(c''_*: H_1(\Sigma_k;\Z) \longrightarrow  H_1( C'';\Z) \big) 
$$
See also  Auclair \& Lescop \cite[Lemma 4.11]{AL}. \hfill $\blacksquare$  
\end{remark}

\subsection{Characterization of $J_k$ and $Y_k$ at low $k$ for closed manifolds}

The $J_1$-equivalence on $\calV(\varnothing)$ being perfectly understood  thanks to Theorem \ref{th:Matveev},
we now turn to the  $J_2$-equivalence.
Recall from Proposition \ref{prop:J_2} that the  $J_2$-equivalence coincides with
the $2$-surgery equivalence.  The latter has been characterized in \cite{CGO}.

In addition to the linking pairing $\lambda_M$ of a closed $3$-manifold $M$, 
the characterization of the $2$-surgery equivalence involves the cohomology ring of $M$.
It follows from Poincar\'e duality that all the (co)homology groups of $M$ are determined by $H_1(M;\Z)$.
Furthermore,  the cohomology ring $H^*(M;\Z_r)$ is determined for any $r\in \N$ by the \emph{triple-cup product form}
$$
u^{(r)}_M: H^1(M;\Z_r) \times H^1(M;\Z_r) \times H^1(M;\Z_r) \longrightarrow \Z_r,
$$
which is the trilinear and skew-symmetric form defined  by 
$$
\forall x,y,z \in H^1(M;\Z_r), \quad
u^{(r)}_M(x,y,z) :=  \big\langle x \cup y \cup z, [M] \big\rangle \ \in \Z_r .
$$

It turns out that all these forms can be encoded by a single invariant:
the \emph{abelian (oriented)  homotopy type} of $M$, which is defined as the homology class
\begin{equation} \label{eq:mu_ab}
\mu_{1}(M) := f_*([M]) \ \in H_3\big(H_1(M)\big).
\end{equation}
Here, homology groups are taken with $\Z$-coefficients,
$f:M \to K(H_1(M),1)$ is a continuous map in an Eilenberg--MacLane space 
that induces the canonical homomorphism $\pi_1(M) \to H_1(M)$ at the level of $\pi_1$,
and the homology of the space $ K(H_1(M),1)$ is identified to the homology of the (abelian) group $H_1(M)$.

\begin{theorem}[Cochran--Gerges--Orr 2001] \label{th:CGO}
Let $M,M' \in \calV(\varnothing)$. The following three statements are equivalent:
\begin{enumerate}
\item $M$ and $M'$ are $J_2$-equivalent;
\item  there is an isomorphism  $\psi: H_1(M;\Z) \to H_1(M';\Z)$ 
such that $\lambda_M$ corresponds to  $\lambda_{M'}$ through $\psi$, and  
$u^{(r)}_{M'}$ corresponds to $u^{(r)}_{M}$ through $\operatorname{Hom}(\psi, \Z_r)$ for all $r\in \N$;
\item  there is an isomorphism  $\psi: H_1(M;\Z) \to H_1(M';\Z)$ 
such that the induced map $\psi_*: H_3\big(H_1(M;\Z)\big) \to H_3\big(H_1(M';\Z)\big)$
maps $\mu_{1}(M)$ to $\mu_{1}(M')$.
\end{enumerate}
\end{theorem}

\begin{proof}[About the proof]  
In fact, the results of \cite{CGO} give a fourth, equivalent condition: 
\begin{quote}(4){\it \ there is a cobordism $W$ from $M$ to $M'$ 
such that the maps $\hbox{incl}_*: H_1(M;\Z) \to H_1(W;\Z)$ and  $\hbox{incl}_*: H_1(M';\Z) \to H_1(W;\Z)$
induced by the inclusions are isomorphisms.} \end{quote}
Some of the implications are not too difficult to prove, like 
\begin{itemize}[label=$\diamond$]
\item $(1) \Rightarrow (4)$ and $(4) \Rightarrow (1)$ working with the formulation of the $J_2$-equivalence in terms of $2$-surgeries;
\item $(4)  \Rightarrow (3)$ using  the canonical map $\Omega_3\big(K(H_1(M),1)\big) \to H_3\big(H_1(M)\big)$ defined on the 
third cobordism group relative to $K(H_1(M),1)$;
\item $(3)  \Rightarrow (2)$ using that the forms $\lambda_M$ and $u^{(r)}_M$
 are defined by (co)homology operations, which also exist in the category of groups.
\end{itemize}
Some other implications like $(2) \Rightarrow (3)$ and $(3) \Rightarrow (4)$ are much more involved.
We recommend the reading of \cite{CGO} where techniques of low-dimensional topology, differential topology
and algebraic topology intertwine in a rich manner.  
\end{proof}

As an immediate consequence of Theorem \ref{th:CGO}, we obtain the following result about $\mathcal{S}$.
  It appeared priorly in \cite{Matveev}, in its formulation with boundary links (see Remark~\ref{rem:boundary_links}).  

\begin{cor}[Matveev 1987] \label{cor:J_2}
Any homology $3$-sphere is $J_2$-equivalent to $S^3$.
\end{cor}

\noindent
  Although the first publication of Corollary \ref{cor:J_2} seems to be \cite{Matveev},
it appears that the result was known from Johnson as early as 1977 \cite{Joh_to_Bir}.
It has been reproved  (in its formulation with $2$-surgeries)  by Casson in order to give a  surgery description of his invariant  \cite{Casson}.  

Morita \cite{Mor89} gave yet another proof of  Corollary \ref{cor:J_2} using Heegaard splittings.
By extending Morita's techniques and after long computations, Pitsch obtained the following in  \cite{Pi08}:

\begin{theorem}[Pitsch 2008] \label{th:Pitsch}
Any homology $3$-sphere is $J_3$-equivalent to $S^3$.
\end{theorem}

\label{page:JJJ}
In a very recent paper \cite{Faes22}, Faes proved the next step for  $\mathcal{S}$.
But, in contrast with Pitsch's proof  of Theorem \ref{th:Pitsch},
his arguments require the classification of the $Y_k$-equivalence on~$\mathcal{S}$ for $k\in \{2,3,4\}$, 
which was obtained by Habiro \cite{Hab00}.

\begin{theorem}[Faes 2022]  \label{th:Faes}
Any homology $3$-sphere is $J_4$-equivalent to $S^3$.
\end{theorem}

Hence we now return to the family of  $Y_k$-equivalence relations
and, for this purpose, we  review a few $3$-manifold invariants whose nature is very different from the linking pairing or the cohomology ring.
  Recall that the set of \emph{spin structures} on an $n$-manifold $V$ (with $n\geq 2$) 
is defined in terms of its bundle $FV$ of oriented frames  $\hbox{GL}_+(\R;n)\!\! \hookrightarrow E(FV) \stackrel{p}{\longrightarrow} V$ by 
$$
\operatorname{Spin}(V) := \left\{ \sigma \in  H^1(E(FV);\Z_2) : \sigma\vert_{\operatorname{fiber}} \neq  0 \in  H^1(\hbox{GL}_+(\R;n);\Z_2) \right\}.
$$
When it is non-empty (i.e$.$ when the second Stiefel--Whitney class $w_2(V)\in H^2(V; \Z_2)$ vanishes),
the set $\operatorname{Spin}(V)$ is  an affine space over $H^1(M;\Z_2)$, 
the action being given by $x\cdot \sigma = \sigma +p^*(x)$ for any $x\in H^1(M;\Z_2)$ and $\sigma \in \hbox{Spin}(M)$.  

Any closed 3-manifold $M$ has a trivial tangent bundle and, so, it admits spin structures.
Given $\sigma \in \operatorname{Spin}(M)$,  the \emph{Rochlin invariant} of $(M,\sigma)$ is defined by
$$
R_M (\sigma) := \hbox{sgn}(W ) \mod 16
$$
where $W$ is  a compact 4-manifold bounded by $M$  to which $\sigma$  extends,
and  $\hbox{sgn}(W)$ denotes the signature of its intersection form on $H_2(W;\Z)$.  
That $R_M (\sigma)$ is well-defined follows from the vanishing of $\Omega_3^{\operatorname{Spin}}$  (a refinement of Theorem \ref{th:RTWL}),
the fact (due to Rochlin) that the signature of a spinable closed $4$-manifold is divisible by~$16$,
and the fact (due to Novikov) that the signature is additive under full-boundary gluing. 
(See \cite{Kirby} for these classical results on $4$-dimensional topology.) 
Hence there is a map $R_M:\operatorname{Spin}(M) \to \Z_{16}$ attached to any closed $3$-manifold $M$.  

Besides, according to \cite{LL,MS}, we can associate to any $\sigma\in \hbox{Spin}(M)$
a \emph{quadratic form} over the linking pairing $\lambda_M$, which means a map 
$q_{M,\sigma}: \hbox{Tors}\, H_1(M;\Z) \to \Q/\Z$ satisfying
$$
\forall x,y \in  \hbox{Tors}\, H_1(M;\Z), \quad q_{M,\sigma}(x+y) = q_{M,\sigma}(x) + q_{M,\sigma}(y) +\lambda_M(x,y).
$$  
Hence there is also a map $q_M:\hbox{Spin}(M) \to \hbox{Quad}(\lambda_M)$ whose target is the set of quadratic forms over $\lambda_M$. 
(This is the refinement of the linking pairing that has been evoked in Remark \ref{rem:YY}.)  

We can now state the characterization of $Y_2$ on $\calV(\varnothing)$ given in \cite{Mas03}.

\begin{theorem}[Massuyeau 2003] \label{th:Y_2}
Two manifolds $M,M' \in \calV(\varnothing)$ are $Y_2$-equivalent if, and only if, there is an isomorphism 
$\psi: H_1(M;\Z) \to H_1(M';\Z)$ and a bijection $\Psi: \operatorname{Spin}(M') \to \operatorname{Spin}(M)$ satisfying the following:
\begin{enumerate}
\item $\lambda_M$ corresponds to  $\lambda_{M'}$ through $\psi$, and  
$u^{(r)}_{M'}$ corresponds to $u^{(r)}_{M}$ through $\operatorname{Hom}(\psi, \Z_r)$ for any $r\in \N$;
\item $R_{M'}$ corresponds to $R_M$ through $\Psi$;
\item $\psi$ and $\Psi$ are compatible in the sense that $\Psi$ is affine over $\operatorname{Hom}(\psi,\Z_2)$
and we have the commutative diagram:
$$
\xymatrix{
\operatorname{Spin}(M) \ar[r]^-{q_M} & \operatorname{Quad}\left(\lambda_M\right)\\
\operatorname{Spin}(M') \ar[u]^-{\Psi}_-\simeq \ar[r]_-{q_{M'}} & 
\operatorname{Quad}\left(\lambda_{M'}\right). \ar[u]_{\psi^*}^-\simeq
}
$$
\end{enumerate}
\end{theorem}

\begin{proof}[About the proof]  
Assume  a Torelli twist $M \leadsto M_s$ along a surface $S \subset M$ such that $M_s\cong M'$.
This surgery induces an isomorphism $\psi_s$ in homology as we have seen at \eqref{eq:psi_s}. 
Furthermore, as shown in~\cite{Mas03-}, the surgery  $M \leadsto M_s$ induces a canonical bijection
$\Psi_s:\operatorname{Spin}(M_s) \to \operatorname{Spin}(M)$, which is affine over 
$$
\Hom(\Psi_s,\Z_2):  \Hom(H_1(M_s),\Z_2)\simeq H^1(M_s;\Z_2) \to H^1(M;\Z_2) \simeq \Hom(H_1(M),\Z_2)
$$
Specifically, it is the unique map that fits into the following commutative diagram:
$$
\xymatrix@!0 @R=1cm @C=4cm  {
\operatorname{Spin}(M)\ \ \ \ar@{>->}[rd]^{\hbox{\footnotesize incl}^*} &\\
& \operatorname{Spin}\big(M\setminus \interior \hbox{N}(S) \big)\\
\operatorname{Spin}(M_s) \ \ \ \ \ar@{>->}[ru]_{\hbox{\footnotesize incl}^*}   \ar@{-->}[uu]^{\Psi_s}_-\simeq &
}
$$
 We have seen in the proof of Theorem \ref{th:Matveev} that the linking pairing is preserved by the Torelli twist $M \leadsto M_s$,
  but this is not true anymore neither for the cohomology ring or for the Rochlin function.
 Nonetheless, we can explicitly compute how those two invariants change after a single $Y$-surgery,
 and thus observe that there is no variation if the $Y$-graph has a $0$-framed null-homologous leaf: 
 hence, using the operation $(\mathcal{O}_1)$ at page \pageref{page:dev_node}, we see that there is no variation
 by surgery along a connected graph clasper of degree $2$. Using Proposition \ref{prop:Y_clasper}, we deduce
 that the isomorphism class of the triplet
 (linking pairing, cohomology rings, Rochlin function) is invariant under $Y_2$-equivalence.

To prove the converse,  we apply  the ``general strategy''  by clasper calculus, 
which has been  sketched on page \pageref{page:method} (with $k:=1$).
Thus, although Theorem~\ref{th:CGO} and Theorem~\ref{th:Y_2} show  similarities in their statements,
their proofs are very different and logically independent.   
\end{proof}

As an immediate consequence of Theorem \ref{th:Y_2}, 
we obtain the following result for homology $3$-spheres which appeared priorly in \cite{Hab00}.
Note that an $M\in \mathcal{S}$ has a unique spin structure $\sigma_0$, and it turns out that $R(M,\sigma_0)$ can only be 0 or $8$ modulo~16:
in this case, the \emph{Rochlin invariant} of $M$ refers to  $R(M,\sigma_0)/8 \in \Z_2$.

\begin{cor}[Habiro 2000] \label{cor:Y_2}
Two homology $3$-spheres are $Y_2$-equivalent if, and only if, they have the same Rochlin invariant.
\end{cor}

The paper \cite{Hab00} also contains the characterization of $Y_3$ and $Y_4$. To state this, 
let us recall that the \emph{Casson invariant} $$\lambda(M) \in \Z$$ of an $M \in \mathcal{S}$
is an integral lift of the Rochlin invariant $R(M,\sigma_0)/8 \in \Z_2$.
  In some sense, $\lambda(M)$ is defined to count the number of conjugacy classes of irreducible representations of $\pi_1(M)$
in the Lie group $\hbox{SU}(2)$ using a Heegaard splitting of~$M$ \cite{Casson}.
Casson also provided  a surgery formula for  $\lambda$ in terms of the Alexander polynomial of knots, 
which makes this invariant very computable: see, for instance, the textbook \cite{Saveliev}.
By means of this surgery formula, Morita could prove that $\lambda$ behaves like a ``quadratic'' function on the Torelli group \cite{Mor89,Mor91},
and Lescop generalized Morita's result in a broader situation  \cite{Lescop}
(namely, Walker's extension of the Casson invariant to \emph{rational} homology $3$-spheres). 
This quadraticity of $\lambda$ is an expression of its property to be  a finite-type invariant of degree $2$ (see \S \ref{subsec:higher} below),
and this is  precisely the property of $\lambda$ that is needed for the following result.

\begin{theorem}[Habiro 2000] \label{th:Y_34}
Two homology $3$-spheres are $Y_3$-equivalent (resp., $Y_4$-equivalent) 
if, and only if, they have the same Casson invariant.
\end{theorem}

The characterization of $Y_3$ (and, a fortiori, $Y_4$) in the general case of closed $3$-manifolds 
does not seem to appear in the literature.
Neither is the characterization of $J_3$ (and, a fortiori, $J_4$).

\begin{remark} \label{rem:two_problems}  
At this point of our discussion, it is important to focus on the nature of the results that we have presented so far for closed $3$-manifolds.
Each of them is concerned with a certain  surgery equivalence relation $\sim$  and states that
$$
\forall M, M'\in \calV(\varnothing), \quad M \sim M' \ \Longleftrightarrow \  I(M) \simeq I(M')
$$
where $I:\calV(\varnothing) \to A$ is a certain ``package'' of algebraic invariants 
with values in an appropriately-defined set where there is a notion of isomorphism $\simeq$.
But such a \emph{characterization} of $\sim$ is not yet a \emph{classification} result, since it continues with two other problems:
\begin{itemize}[label=$\diamond$]
\item \emph{Realization:} Does one know what is the image of $I$ in $A$?
\item \emph{Isomorphism:} Is the isomorphism problem solved in $A$?
\end{itemize}
So,  let us reconsider the above characterizations of surgery  equivalence relations under this new angle:\\
\begin{tabular}{c|c|c|c|}
& Torelli--equivalence  & $J_2$-equivalence & $Y_2$-equivalence \\ \hline
Characterization & Theorem \ref{th:Matveev} & Theorem \ref{th:CGO}  & Theorem \ref{th:Y_2} \\ \hline
Realization problem & solved \cite{Wall63} & solved \cite{Sul75,Tur84} & solved \cite{Tur84} \\ \hline
Isomorphism problem &  solved \cite{Wall63,KK}  & unknown?&  unknown?  \\ \hline
\end{tabular}\vspace{0.2cm}

Wall  showed that any non-singular bilinear pairing on a finite abelian group can be realized
as the linking pairing of a closed $3$-manifold \cite{Wall63}. 
He also gave a partial description (by generators and relations) of the abelian monoid
of isomorphism classes of such pairings  (where the operation is the direct sum $\oplus$).
His work has been completed later on by Kawauchi \& Kojima \cite{KK}.

Sullivan proved in \cite{Sul75} that any trilinear  alternate form on a finitely-generated free abelian group
can be realized as the triple-cup product form of a closed $3$-manifold: 
it is interesting to note that, in the middle of the 70's and in order to prove this result, Sullivan was already using a surgery operation
equivalent to the $Y$-surgery.

There exist several kinds of relations between the linking pairing, the triple-cup product forms and the Rochlin function.
For instance, the triple-cup product forms $u_M^{(r)}$ and $u_M^{(s)}$ with coefficients in $\Z_r$
and $\Z_s$, respectively, are related in an obvious way if $r$ divides $s$.
But there are also other, more delicate, relations: for instance, 
the third ``discrete'' differential of the Rochlin function $R_M$ is given by~$u_M^{(2)}$.
In fact, Turaev  described in \cite{Tur84} all such possible relations, 
and he thus solved the realization problem for the triplet  (linking pairing, cohomology rings, Rochlin function).
However, since the isomorphism problem for trilinear skew-symmetric forms does not seem to be solved 
(even for coefficients in  $\Q$), there is currently no procedure to decide (in general)
whether two closed $3$-manifolds are $J_2$-equivalent. 
Consequently, the same applies to the $Y_2$-equivalence relation.    \hfill $\blacksquare$
\end{remark}

\subsection{Characterization of $J_k$ and $Y_k$ at low $k$ for homology cylinders} 

We now consider the case of homology cylinders
over a compact  surface $\Sigma$ (with one boundary component).

We start with some generalities about the structure added 
by the sequence of $Y_k$-equivalence relations on the monoid $\mathcal{IC}(\Sigma)$.
For every $k\in \N^*$, denote by $Y_k \mathcal{IC}(\Sigma)$ the subset of homology cylinders
that are $Y_k$-equivalent to the trivial cylinder $U$. Hence, we get a decreasing sequence
$$
 \mathcal{IC}(\Sigma) = Y_1 \mathcal{IC}(\Sigma) \supset Y_2 \mathcal{IC}(\Sigma) 
 \supset Y_3 \mathcal{IC}(\Sigma) \supset \cdots 
$$
of submonoids, which is called the \emph{$Y$-filtration}. 
Goussarov \cite{Go00} and Habiro \cite{Hab00} proved that, for any integer $k\geq 0$,
the quotient monoid 
$$
\frac{\mathcal{IC}(\Sigma)}{Y_{k+1}} 
$$
is a group and, that, for any  integers $i,j \geq 1$, the inclusion
$$
\left[\frac{Y_i \mathcal{IC}(\Sigma)}{Y_{k+1}} , \frac{Y_{j} \mathcal{IC}(\Sigma)}{Y_{k+1}} \right] \subset \frac{Y_{i+j} \mathcal{IC}(\Sigma)}{Y_{k+1}}
$$
holds true in that group.
In particular, ${Y_k\mathcal{IC}(\Sigma)}/{Y_{k+1}}$  is an abelian group for all $k\geq 1$, and the direct sum of abelian groups
$$
\hbox{Gr}^Y \mathcal{IC}(\Sigma) := \bigoplus_{k\geq 1} \frac{Y_k\mathcal{IC}(\Sigma)}{Y_{k+1}}
$$
has the structure of a graded Lie ring. The following is easily checked.

\begin{proposition}
The ``mapping cylinder'' construction $\operatorname{cyl}: \mathcal{I}(\Sigma) \to \mathcal{IC}(\Sigma)$ induces
a morphism of graded Lie rings 
$
\operatorname{Gr}( \operatorname{cyl}): \operatorname{Gr}^\Gamma \mathcal{I}(\Sigma) \to  \operatorname{Gr}^Y \mathcal{IC}(\Sigma).
$
\end{proposition}

\noindent  
Thus the  ``Lie algebra of homology cylinders''  $\hbox{Gr}^Y \mathcal{IC}(\Sigma)$ is highly related to
the ``Torelli Lie algebra''  $ \operatorname{Gr}^\Gamma \mathcal{I}(\Sigma)$, which has been reviewed at \eqref{eq:Gr_Gamma}.
We refer to the works \cite{Hab00,GL05,Habegger,CHM,HM09,MM13,NSS20,NSS21};
see also the end  of \S \ref{subsec:higher} in this connection.

In this subsection, we only deal with the low-degree parts of $\operatorname{Gr}^Y \mathcal{IC}(\Sigma)$.
We start with the characterization of the $Y_2$-equivalence, which needs two invariants of homology cylinders.
The first invariant is the action of $\mathcal{IC}(\Sigma)$ on the second nilpotent quotient $\pi/ \Gamma_3 \pi$
of $\pi=\pi_1(\Sigma,\star)$. Indeed, as observed in \cite{GL05},
the group homomorphism \eqref{eq:rho_k} can be extended (for any $k\in \N^*$) to a monoid homomorphism:
$$
\xymatrix{
\calI(\Sigma)\ar[d]_{\operatorname{cyl}}  \ar[r]^-{\rho_k}& \operatorname{Aut}\big(\pi/ \Gamma_{k+1} \pi\big) \\
\mathcal{IC}(\Sigma)  \ar@{-->}[ru]_-{\rho_k} & 
}
$$
The second invariant of homology cylinders that we need is the \emph{Birman--Craggs homomorphism},
which originates from constructions of Birman \& Craggs \cite{BC} on the Torelli group and was studied by Johnson \cite{Joh80}.
In our setting, the most efficient way to define it is as follows:
$$
\beta:  \mathcal{IC}(\Sigma) \longrightarrow \operatorname{Map}(\operatorname{Spin}(\Sigma), \Z_2), \
M \longmapsto \frac{1}{8} {R_{\wideparen{M}}}
$$
Here, we associate to any $M\in \mathcal{IC}(\Sigma)$ the closed $3$-manifold
\begin{equation} \label{eq:closure}
\wideparen{M} := M \cup_{m} (- \Sigma \times [-1,1]),
\end{equation}
we identify $\operatorname{Spin}(\Sigma)$  to $\operatorname{Spin}(\wideparen{M})$
via the map $m_\pm:  \Sigma \to M \hookrightarrow \wideparen{M}$,
and we use the fact that the Rochlin function $R_{\wideparen{M}}$ takes values in $\{0,8\}$ (because $H_1(\wideparen{M};\Z)$ is torsion-free).
The following  is a generalization of Corollary~\ref{cor:Y_2} in genus $g>0$.

\begin{theorem}[Habiro 2000, Massuyeau--Meilhan 2002] \label{th:Y2_cyl}
Two homology cylinders $M,M'$ are $Y_2$-equivalent if, and only if, 
$\beta(M)=\beta(M')$ and $\rho_2(M)=\rho_2(M')$.
\end{theorem}

\begin{proof}[About the proof]
  This characterization is announced in \cite{Hab00} and proved in \cite{MM03}.
It preceded Theorem~\ref{th:Y_2} and uses the same techniques for its proof.
Note that the situation of homology cylinders is simpler than the situation of closed manifolds for two reasons:
the first homology groups of homology cylinders are torsion-free (hence there is no linking pairing to deal with),
 and  they come with a natural parametrization by an abelian group independent of the manifold (namely $H_1(\Sigma;\Z)$).  
\end{proof}

\begin{remark}
Actually, the results in \cite{MM03} give an explicit computation of the abelian group $ \mathcal{IC}(\Sigma)/Y_2$
and, thanks to Johnson's computation of the abelianized Torelli group \cite{Joh85b},
 this  implies that the degree $1$ part
$$
\operatorname{Gr}_1( \operatorname{cyl}): \mathcal{I}(\Sigma)/[\mathcal{I}(\Sigma),\mathcal{I}(\Sigma)] 
\to \mathcal{IC}(\Sigma)/Y_2
$$
of the ``mapping cylinder'' construction is an isomorphism. \hfill $\blacksquare$
\end{remark}

To state now the characterization of the $Y_3$-equivalence, we need still more invariants.
On the one hand, we  fix an embedding $j:\Sigma \to S^3$ such that $j(\Sigma)$ is a Heegaard surface of $S^3$
(deprived of a small open disk), and we identify $ \operatorname{N}(j(\Sigma))$ with $\Sigma \times [-1,1]$ via $j$. 
Then the Casson invariant induces  a map
$$
\lambda_j: \mathcal{IC}(\Sigma) \longrightarrow \Z, \ 
M \longmapsto  \lambda\big((S^3 \setminus \interior(\Sigma \times [-1,1])) \cup_m M \big),
$$
which constitutes an invariant of homology cylinders. It depends on the choice of $j$, of course,
but this dependency can be  managed as Morita did in the case of the Torelli group \cite{Mor91}.
On the other hand, we can consider the homology $H_1(M,\partial_-M; \Z[H])$
of $M$ relative to its ``bottom'' boundary'' $\partial_-M= m_-(\Sigma)$,  with coefficients
twisted by $m_{\pm,*}^{-1}: H_1(M;\Z) \to H:=H_1(\Sigma;\Z)$; the order of this $\Z[H]$-module
$$
\Delta(M,\partial_-M) := \hbox{ord} \ H_1(M,\partial_-M; \Z[H]) \ \in \Z[H]
$$
is a relative version of the Alexander polynomial.
With this definition, $\Delta(M,\partial_-M)$ is only defined up to multiplication
by a unit of the ring $\Z[H]$, i.e$.$ an element of $\pm H$;
however, by using Turaev's refinement of the Reidemeister torsion, this indeterminacy can be fixed.
Next, it is possible to ``expand'' $\Delta(M,\partial_-M)$ 
as an element of (the degree completion of) the symmetric algebra $S(H)$, 
and to keep only the degree $2$ part of that expansion:
$$
\alpha(M) \in S^2(H).
$$
The following is a generalization of Theorem \ref{th:Y_34} in genus $g>0$.

\begin{theorem}[Massuyeau--Meilhan 2013] \label{th:Y3_cyl}
Two homology cylinders $M,M'$ are $Y_3$-equivalent if, and only if, we have
$\lambda_j(M)=\lambda_j(M')$, $\rho_3(M)=\rho_3(M')$ and $\alpha(M)=\alpha(M')$.
\end{theorem}

\begin{proof}[About the proof]
 This theorem is proved in \cite{MM13}.
An important step  in the proof consists  in identifying the abelian  group $Y_2 \mathcal{IC}(\Sigma)/Y_3$
and, for this,   the ``general strategy'' by clasper calculus (see page \pageref{page:method}) is applied (with $k:=2$).
But, the difficulty is to assemble all three invariants that are expected 
to characterize the $Y_3$-equivalence  (namely $\lambda_j,\rho_3,\alpha$)
into a single homomorphism $Z_2$ defined on $Y_2 \mathcal{IC}(\Sigma)/Y_3$.
This role of ``unifying invariant'' is played by the degree $2$ part of the LMO homomorphism $Z$ \cite{CHM,HM09},
 whose behaviour under $Y_2$-surgery is well-understood. (See also the end of \S \ref{subsec:higher} in this connection.)  
\end{proof}

It is also explained in \cite{MM13} how to deduce from Theorem \ref{th:Y2_cyl} and Theorem  \ref{th:Y3_cyl}
characterizations of the  $J_2$-equivalence and $J_3$-equivalence, respectively. Specifically,
$J_2$ is classified by $\rho_2$ and $J_3$ is classified by the couple $(\rho_3,\alpha)$.
In genus $g=0$, Theorem \ref{th:Pitsch} is thus recovered with a completely different proof than \cite{Pi08}.
Besides, the same strategy of proof (i.e., use $Y_k$ to understand $J_k$) is used in \cite{Faes22} for proving Theorem \ref{th:Faes}.

\begin{remark}
  Nozaki, Sato and Suzuki  \cite{NSS20}  have determined the abelian group $Y_3 \mathcal{IC}(\Sigma)/Y_4$.
Their description too involves  a ``clasper surgery'' map $\psi_k$ of the type described on page \pageref{page:method} (with $k:=3$),
and their arguments involve some (reductions of) higher-degree parts of the LMO homomorphism $Z$.
It still remains to deduce from their result a characterization of the $Y_4$-equivalence relation  on the full monoid~$\mathcal{IC}(\Sigma)$.   \hfill $\blacksquare$
\end{remark}

\begin{remark}
  In contrast with the case of closed $3$-manifolds, 
the above characterizations  of  $Y_k$-equivalence and $J_k$-equivalence relations   for homology cylinders 
do  not lead to  ``isomorphism problems''  of the type mentioned in Remark \ref{rem:two_problems}. 
\hfill $\blacksquare$  
\end{remark}

\subsection{Characterization in higher degrees} \label{subsec:higher}

To conclude, we now survey  what is known about the characterization in arbitrary high degrees
of the three main families of  relations that have been considered in these notes: namely
the $k$-surgery equivalence, the $J_k$-equivalence  and the $Y_k$-equivalence.

First of all, we consider the family of $k$-surgery equivalence relations on $\calV(\varnothing)$.
We start with an easy observation.

\begin{proposition} \label{prop:triviality}
Any  homology $3$-sphere $M$  is $k$-surgery equivalent to $S^3$, for every $k\geq 1$.
\end{proposition}

\begin{proof}
 By Corollary \ref{cor:J_2}, there is a sequence
$$
S^3=M_0 \leadsto M_1 \leadsto \cdots  \leadsto M_r=M
$$
where each move $M_i \leadsto M_{i+1}$ is a $(\pm1)$-framed surgery along a knot $K_i$ in a homology $3$-sphere $M_i$.
Since $\pi_1(M_i)$ has  trivial abelianization, we have $\pi_1(M_i) = \Gamma_k \pi_1(M_i)$ for all $k\geq 1$:
hence the move $M_i \leadsto M_{i+1}$ can be viewed as a $k$-surgery for every $k\geq 1$.
\end{proof}

Nevertheless, as was shown in \cite{CGO},
the family of $k$-surgery relations is very interesting for $3$-manifolds that are homologically non-trivial.
Following Turaev~\cite{Tur82}, we define the \emph{$k$-th nilpotent (oriented) homotopy type} of a closed $3$-manifold $M$ as
$$
\mu_{k}(M) := f_*([M]) \ \in H_3\left(\frac{\pi_1(M)}{\Gamma_{k+1} \pi_1(M)};\Z\right)
$$
where $f:M \to K\big(\pi_1(M)/\Gamma_{k+1} \pi_1(M),1\big)$ is a continuous map in an Eilenberg--MacLane space 
inducing the canonical homomorphism $\pi_1(M) \to \pi_1(M)/\Gamma_{k+1} \pi_1(M)$ at the level of $\pi_1$.
(Of course, for $k:=1$, we recover what we called  in \eqref{eq:mu_ab} the ``abelian homotopy type'' of $M$.)

One can view $\mu_{k}(M)$ as an approximation of the (oriented) homotopy type of $M$
since, according to \cite{Thomas,Swarup}, the latter is encoded by $\pi_1(M)$ and the image of 
the fundamental class $[M]$ in $H_3\big(\pi_1(M);\Z\big)$. 
Then we have the following generalization of the equivalence (1)$\Leftrightarrow$(3) in Theorem \ref{th:CGO}.

\begin{theorem}[Cochran--Gerges--Orr 2001] \label{th:CGO-high}
Let $k\in \N^*$.
Two closed $3$-manifolds $M$ and $M'$ are $(k+1)$-surgery equivalent if, and only if, there is an isomorphism
$\psi: \pi_1(M)/ \Gamma_{k+1} \pi_1(M) \longrightarrow \pi_1(M')/ \Gamma_{k+1} \pi_1(M')$ 
mapping $\mu_k(M)$ to $\mu_{k}(M')$.
\end{theorem}

\noindent  
Although the realization problem for nilpotent homotopy types of $3$-manifolds has been (formally) solved in \cite{Tur82},
it seems to be  really difficult to classify the $k$-surgery equivalence relations,
especially because the third homology groups of finitely-generated nilpotent groups do not seem to be well understood.
Yet, Cochran, Gerges \& Orr have been able to apply Theorem~\ref{th:CGO-high} in one particular case:
using a good knowledge \cite{IO} of the third homology group of finitely-generated  free-nilpotent groups,
they prove that a closed $3$-manifold $M$ is $k$-surgery equivalent to $\sharp^m (S^1 \times S^2)$
if, and only if, we have $H_1(M;\Z)\simeq \Z^m$ and all Massey products  of $M$ of length $\leq 2k-1$ vanish.
(For $k:=2$, this is an instance of the equivalence ``(1)$\Leftrightarrow$(2)'' in Theorem \ref{th:CGO}.)

Here is another consequence of  Theorem~\ref{th:CGO-high}, which does not seem to have been observed before.

 \begin{cor} \label{cor:J_k-surg}
 Let $M,M' \in \calV(\varnothing)$ and let $k\geq 2$ be an integer.
 If $M$ and $M'$ are $J_{2k-2}$-equivalent, then they are $k$-equivalent.
 \end{cor}
 
\begin{proof}
Let $j\in \N^*$ and assume a Torelli twist $M \leadsto M_s$ along a surface $S \subset M$ with an $s\in J_{j} \calI(S)$.
 The Seifert--Van Kampen theorem shows the existence of a unique isomorphism
$$
\psi_s: \pi_1(M)/\Gamma_{j+1} \pi_1(M) \stackrel{\simeq}{\longrightarrow}
\pi_1(M_s)/\Gamma_{j+1} \pi_1(M_s) 
$$
that fits into the commutative diagram:
$$
\xymatrix{
&{\frac{\pi_1\left(M \setminus \interior \operatorname{N}(S) \right)}
{\Gamma_{j+1}\pi_1\left(M \setminus \interior \operatorname{N}(S) \right)}} 
\ar@{->>}[ld] \ar@{->>}[rd] & \\
{\frac{\pi_1(M)}{\Gamma_{j+1} \pi_1(M)}} \ar@{-->}[rr]_-\simeq^-{\psi_s}
& & {\frac{\pi_1(M_s)}{\Gamma_{j+1} \pi_1(M_s)}}.
}
$$
In order to compare $\mu_j(M)$ and $\mu_j(M_s)$ via $\psi_s$, 
we consider the \emph{mapping torus} of~$s$ which, with the notation \eqref{eq:closure}, can be defined as
$$
\operatorname{tor}(s) := \wideparen{\operatorname{cyl}(s)}
$$
where $\operatorname{cyl}(s)\in \mathcal{IC}(S)$ denotes the mapping cylinder of $s$.
This is a closed $3$-manifold whose $j$-th nilpotent fundamental group
can be identified to that of $S$ by the isomorphism
$$
 \varphi_s: \pi_1(S)/\Gamma_{j+1} \pi_1(S)
\stackrel{\simeq}{\longrightarrow}  \pi_1(\operatorname{tor}(s))/ \Gamma_{j+1} \pi_1(\operatorname{tor} (s))
$$
that is induced by the inclusion $S= S \times 1 \hookrightarrow \operatorname{tor}(s)$.
Besides, the inclusion $S \hookrightarrow M$ induces a homomorphism
$$
\iota: \pi_1(S)/\Gamma_{j+1} \pi_1(S) \longrightarrow \pi_1(M)/\Gamma_{j+1} \pi_1(M).
$$
Then, a simple homological computation
in a singular $3$-manifold that contains the three of $M$, $M_s$ and $\operatorname{tor}(s)$ shows that 
\begin{equation} \label{eq:mu_variation}
\psi_{s,*}^{-1 }\big(\mu_j(M_s)\big) = \mu_j(M) + \iota_* \varphi^{-1}_{s,*}\big(\mu_j\big(\operatorname{tor}(s)\big) \big).
\end{equation}
This variation formula for the $j$-th nilpotent homotopy type is established
in the introduction of \cite{Mas12}, generalizing \cite[Theorem 2]{GL05} and \cite[Theorem 5.2]{Heap}.

The same formula shows that, given a compact surface $\Sigma$ with $\partial \Sigma \cong S^1$, 
the following map is a group homomorphism:
$$
M_j:J_j \calI(\Sigma) \longrightarrow H_3 \left(\frac{\pi_1(\Sigma)}{\Gamma_{j+1} \pi_1(\Sigma)}; \Z  \right),
\ f  \longmapsto \mu_j\big( \operatorname{tor}(f) \big).
$$
This is essentially the \emph{$j$-th Morita homomorphism}, introduced in \cite{Mor93}
as a refinement of the ``$j$-th Johnson homomorphism''. As shown by Heap in \cite{Heap},
the kernel of $M_j$ is $J_{2j} \calI(\Sigma)$. Therefore, if $M'$ is the result of a Torelli twist $M\leadsto M_s$
with an $s\in J_{2(k-1)}\calI(S)$, we have $ \mu_{k-1}\big( \operatorname{tor}(s)\big)=0$.
So, we conclude thanks to~\eqref{eq:mu_variation}
that $M$ and $M'$ are $k$-surgery equivalent.
 \end{proof}
 
 \begin{remark}
It would be interesting to have a direct proof of Corollary \ref{cor:J_k-surg}, 
which would apply to $\calV(R)$ for any compact surface $R$. 
Indeed, surgery along a connected graph clasper of degree $2k-2$ can always be realized
as a sequence of three $k$-surgeries (see \cite[Fig. 3]{MM03} for $k=2$): therefore, 
by  Proposition~\ref{prop:Y_clasper}, the $Y_{2k-2}$-equivalence is stronger than the $k$-surgery equivalence \cite{Hab00}.
Given that  ``$Y_{2k-2} \Rightarrow J_{2k-2}$'', 
it is likely that  Corollary \ref{cor:J_k-surg} is true in  $\calV(R)$ for any $R$.
 $\hfill$ $\blacksquare$
\end{remark}  

The following question now arises for the family of $J_k$-equivalence relations: 
can we expect a result analogous  to Theorem~\ref{th:CGO-high}?
This seems to be currently out of reach, as revealed already  by the case of homology $3$-spheres.
Indeed, the methods for proving  the triviality of the $J_3$-equivalence (resp$.$, $J_4$-equivalence) in \cite{Pi08}  (resp$.$, in \cite{Faes22}) 
 seem to be hard to adapt to arbitrary high degrees.

 \begin{remark}
So, in view of Proposition~\ref{prop:triviality}, we can  hardly imagine a kind of converse to Corollary \ref{cor:J_k-surg}. \hfill $\blacksquare$
 \end{remark}

In contrast with the $J_k$-equivalence, we know (at least, theoretically) how to characterize the $Y_k$-equivalence relation
 in any degree $k\geq 1$ by means of a certain family of topological invariants of $3$-manifolds.
In the sequel, we fix a compact surface $R$ and  a $Y_1$-equivalence class $\calV_0$ in $\calV(R)$.

\begin{definition} \label{def:FTI}
Let $A$ be an abelian group.
A map $F: \calV_0 \to A$
is a \emph{finite-type invariant} of \emph{degree} at most $d$ if, for any $M \in \calV_0$, 
for any pairwise-disjoint compact surfaces $S_0, S_1,\dots, S_d\subset \interior(M)$ with $\partial S_i \cong S^1$,
and for all $s_0 \in \calI(S_0), s_1 \in \calI(S_1),\dots, s_{d}  \in \calI(S_{d})$, we have
$$
\sum_{J \subset \{0,1,\dots,d\} } (-1)^{\vert J \vert} \cdot F(M_{J}) =0 \in A
$$ 
where $M_J\in \calV_0$ is obtained from $M$ by twist along  $\sqcup_{j\in J}S_j$ with $\sqcup_{j\in J}s_j$.  \hfill $\blacksquare$
\end{definition}

\begin{remark}
  The notion of ``finite-type invariants'' for homology $3$-spheres has been introduced by Ohtsuki in \cite{Oht96},
as an analogue of the notion of ``Vassiliev invariants'' for knots and links in $S^3$.
This notion has been extended and studied  by Cochran \& Melvin \cite{CM}, 
who considered arbitrary $3$-manifolds. In this Ohtsuki--Cochran--Melvin theory, the basic operation
is the $2$-surgery instead of the Torelli~twist. 

The rich interplay between the theory of finite-type invariants
and the study of mapping class groups was firstly considered by Garoufalidis \& Levine \cite{Gar96,GL97,GL98,GL96}.
Next, came the ``clasper calculus'' of Goussarov and Habiro  \cite{Go00,Hab00},
which offered very efficient tools to study and enumerate  finite-type invariants.
Their works also revealed that the Torelli twist
(or any equivalent type of modification, like the $Y$-surgery or the borromean surgery) 
is the appropriate operation to  define finite-type invariants as we did in Definition \ref{def:FTI}.

We refer to \cite{GGP,Hab00} for a comparison of the various notions of finite-type invariants:
they  happen to be all equivalent one to the other  for homology $3$-spheres (up to some degree rescalings), 
but they are \emph{not} equivalent for arbitrary $3$-manifolds.~$\blacksquare$
\end{remark}

In order to explain the relationship between finite-type invariants and the $Y_k$-equivalence relations,
we need a little bit of algebraic context. Let $G$ be an arbitrary group, and denote its \emph{group ring} by 
$
\Z[G],
$
 which is the abelian group freely generated by the set $G$
and has the multiplication inherited from the group operation of $G$. The \emph{augmentation ideal} of $G$ is
$$
I:= I_G = \ker\big(\varepsilon:  \Z[G] \longrightarrow \Z \big)
$$ 
where the \emph{augmentation} $\varepsilon$ is the ring homomorphism mapping any $g\in G$ to $1 \in \Z$.
The \emph{$I$-adic filtration} of $\Z[G]$ is the  sequence $\Z[G]=I^0 \supset I=I^1 \supset I^2 \supset \cdots$
defined by the powers of $I$. The following classical fact relates this to  the lower central series \eqref{eq:LCS} of $G$.
 
\begin{lemma} \label{lem:LCS_I}
Let $k\in \N^*$. For any $g\in \Gamma_k G$, we have $(g-1) \in I^k$.
\end{lemma}

\begin{proof}
The statement is obviously true for $k=1$.
Next, for any $k\in \N^*$, an element of $\Gamma_{k+1} G$ is (by definition)
a product of commutators of the form $[x,y]$ or $[y,x]$ where $x\in G$ and $y\in \Gamma_k G$.
Besides, we have the following identities in $\Z[G]$, for any $g,h \in G$:
\begin{eqnarray*}
g h -1 &=& \left( (g-1) - ( h^{-1} - 1)\right) \cdot h \\
{[g,h]} - 1 &=&  \left((g-1)(h-1) - (h-1)(g-1)\right)g^{-1} h^{-1}.
\end{eqnarray*}
Hence the statement is justified by an induction on $k\geq 1$.
\end{proof}

We can now prove the following.

\begin{proposition} \label{prop:Y_FTI}
Let $M,M'\in \calV_0$ and let $d\in \N$.
If $M$ and $M'$ are $Y_{d+1}$-equivalent, then $F(M)=F(M')$ for any finite-type invariant $F: \calV_0 \to A$ of degree at most $d$.
\end{proposition}

\begin{proof}
Assume that $M\leadsto M_s\cong M'$ by a Torelli twist along $S \subset \interior(M)$ with $s\in \Gamma_{d+1} \calI(S)$.
Consider the map $f: \calI(S) \to A$ defined by $f(u) := F(M_u)$ and extend it by additivity to
$$
f: \Z\big[  \calI(S)  \big] \longrightarrow A.
$$
The fact that $F$ is of finite type of degree at most $d$ implies that $f$ vanishes on all elements of
the form $(s_0-1)(s_1-1)\cdots (s_d-1)$ with $s_0,s_1,\dots, s_d \in \calI(S)$. Since those elements generate $I^{d+1}$
addivitely, we have $f(I^{d+1})=0$. We conclude using the fact that $(s-1)\in I^{d+1}$ by Lemma \ref{lem:LCS_I}.
\end{proof}

If Proposition \ref{prop:Y_FTI} had a converse, then we would get (at least, theoretically)
 a characterization of the $Y_k$-equivalence relation. Indeed,  the converse is true for 
 the class $\calV_0:= \mathcal{S}$.
 
 \begin{theorem}[Habiro 2000] \label{th:Y_ZHS}
 Any two homology $3$-spheres are $Y_{d+1}$-equivalent if, and only if, 
 they are not distinguished by finite-type invariants of degrees at most~$d$.
 \end{theorem}
 
 \noindent
    Thus, Corollary \ref{cor:Y_2} and Theorem \ref{th:Y_34} are  proved by identifying
all (the few) finite-type invariants of homology $3$-spheres of degrees 1, 2 and 3.
 
 \begin{proof}[About the proof of Theorem \ref{th:Y_ZHS}]
 The theorem is announced in \cite{Hab00} and it is proved there in the analogous case of knots in $S^3$.
 See \cite{Mas05} for a proof, which involves clasper calculus.
 \end{proof}  

Let $\Sigma$ be a compact surface with one boundary component, and consider now 
the class $\calV_0 := \mathcal{IC}(\Sigma)$ of homology cylinders over $\Sigma$.
Except in the case $\Sigma=D^2$, it is not known whether the converse to Proposition~\ref{prop:Y_FTI}
holds true for $\mathcal{IC}(\Sigma)$.\\[-0.2cm]

\noindent
\textbf{Goussarov--Habiro Conjecture (GHC).}
\emph{Let $d\in \N^*$. Any two homology cylinders over $\Sigma$ are $Y_{d+1}$-equivalent if, and only if, 
 they are not distinguished by finite-type invariants of degree at most~$d$.}\\[-0.2cm]
 
Currently, the GHC is only known to be true up to degree $d=4$,
the most recent result in this direction being obtained in \cite{NSS21}.
By comparing Lemma \ref{lem:LCS_I} to Proposition \ref{prop:Y_FTI}, 
we see that the GHC is an analogue of the following  problem in group theory,
which can be stated for any group $G$.\\

\noindent
\textbf{Dimension  Subgroup Problem (DSP).}
\emph{Let $k\in \N^*$.  Determine the gap between $\Gamma_{k}  G$  and $(1+I^k) \cap G$ in $\Z[G]$.}\\[-0.2cm]

\noindent
It had been conjectured during a long time that the inclusion $\Gamma_{k}  G \subset(1+I^k) \cap G$ 
should be an equality, until Rips found the first counter-example for $k=4$ and a finite $2$-group $G$ \cite{Rips}.

In fact, the DSP can be generalized 
replacing the lower central series of $G$ by any series
$G=N_1G \supset N_2 G \supset N_3 G \supset \cdots$ of subgroups which is strongly central 
(i.e$.$ $[N_iG,N_j G] \subset N_{i+j}G$ for all $i,j\in \N^*$), and by replacing the $I$-adic filtration by an appropriate filtration of $\Z[G]$.
Furthermore, some versions of the DSP can be formulated in the group algebra $\mathbb{F}[G]$
for any commutative field $\mathbb{F}$, rather than in the group ring $\mathbb{Z}[G]$,
and these versions  of the problem have an explicit solution whose nature depends on the characteristic of $\mathbb{F}$.
(See, for instance, the monograph~\cite{Passi}.)

It is observed in \cite{Mas07} that some results of Goussarov \cite{Go00} and Habiro \cite{Hab00} 
about the $Y$-filtration on $\mathcal{IC}(\Sigma)$ can be interpreted as follows:
the GHC in degree $d$ is  an instance of the DSP for the group $G:= \mathcal{IC}(\Sigma)/Y_{d+1}$.
Thus,   analogues of the GHC for finite-type invariants with values  in commutative fields are obtained in  \cite{Mas07},
and the following weak version of the GHC is then derived:

\begin{theorem}[Massuyeau 2007]
Let $d\in \N^*$.  There exists an integer $D$, depending on $d$ and the topological type of $\Sigma$, with the following property: 
 if  two homology cylinders are not distinguished by finite-type invariants of degree at most~$D$,
then they are $Y_{d+1}$-equivalent.
\end{theorem}

\noindent
We mention the following corollary: 
two homology cylinders  are not distinguished by finite-type invariants if, and only if, 
they are $Y_{k}$-equivalent for any integer $k\geq 1$. 
Actually, it is conjectured that finite-type invariants classify homology cylinders
(and, in particular, homology $3$-spheres).

We conclude with two questions which naturally arise from our discussion on Theorem \ref{th:Y_ZHS} 
and its expected generalization, namely the GHC.

\begin{itemize}[label=$\diamond$]
\item \emph{Does one know well enough all finite-type invariants of a given degree $d$?} 
For homology $3$-spheres, one can construct infinite  series of finite-type invariants  following Ohtsuki's original idea \cite{Oht95},
by appropriate expansions of quantum invariants.
Furthermore, there is a very powerful invariant of homology $3$-spheres:
the \emph{LMO invariant} \cite{LMO}, which is known to be universal  among $\Q$-valued finite-type invariants \cite{Le}
and to dominate large families of quantum invariants \cite{KLO}.
For homology cylinders too,
there is a universal  $\Q$-valued finite-type invariant: the \emph{LMO homomorphism} defined on the monoid $\mathcal{IC}(\Sigma)$,
which  allows for an explicit diagrammatic description of the Lie algebra $\operatorname{Gr}^Y\!\mathcal{IC}(\Sigma)$ 
with rational coefficients \cite{CHM,HM09}. (See \cite{HM12} for a survey.)
But computing those universal invariants is a challenge in high degrees (despite their combinatorial construction)
and, moreover, it is not known whether they dominate all finite-type invariants (including those with values in torsion abelian groups).
Nevertheless, recent works of Nozaki, Sato \& Suzuki provide encouraging perspectives  \cite{NSS20,NSS21}. \\
\item \emph{Can we hope an analogue of Theorem \ref{th:Y_ZHS}  for  arbitrary  closed $3$-manifolds?}
The answer is  trivially ``\emph{yes}'' in degree $0$, but it is certainly ``\emph{no}'' in higher degrees: for instance,
$\sharp^4 (S^1 \times S^2)$ and $(S^1 \times S^1 \times S^1) \sharp (S^1 \times S^2)$ are not $Y_2$-equivalent
(because their cohomology rings are not isomorphic), 
although they are not distinguished by finite-type invariants of degree at most one \cite[Ex$.$ 3.4]{Mas07}.
Yet, this negative answer is not necessarily disappointing.
It rather suggests that the notion of finite-type invariant (as given in Definition \ref{def:FTI}) is not appropriate
for homologically non-trivial  $3$-manifolds: the notion probably needs to be refined, 
by adding a kind of homological structures to $3$-manifolds, 
like a (complex) spin structure or a parametrization of its first homology group.
 \end{itemize}

\end{document}